\documentclass[aap]{imsart}

\RequirePackage{amsthm,amsmath,amsfonts,amssymb,mathrsfs,dsfont,mathtools,thmtools}
\RequirePackage{natbib}
\RequirePackage[colorlinks,citecolor=blue,urlcolor=blue]{hyperref}
\usepackage[capitalize,nameinlink]{cleveref}
\RequirePackage{graphicx}

\startlocaldefs

\newtheorem{theorem}{Theorem}[section]
\newtheorem{lemma}[theorem]{Lemma}
\newtheorem{corollary}[theorem]{Corollary}
\theoremstyle{remark}

\newtheorem*{remark}{Remark}
\usepackage{anyfontsize}
\newcounter{ctr}\loop\stepcounter{ctr}\edef\X{\@Alph\c@ctr}%
	\expandafter\edef\csname s\X\endcsname{\noexpand\mathscr{\X}}
	\expandafter\edef\csname c\X\endcsname{\noexpand\mathcal{\X}}
	\expandafter\edef\csname b\X\endcsname{\noexpand\boldsymbol{\X}}
	\expandafter\edef\csname I\X\endcsname{\noexpand\mathbb{\X}}
\ifnum\thectr<26\repeat
\let\@IE\IE\let\IE\undefined
\newcommand{\IE}{\mathop{{}\@IE}\mathopen{}}
\newcommand{\E}{\IE}
\let\@IP\IP\let\IP\undefined
\newcommand{\IP}{\mathop{{}\@IP}\mathopen{}}
\newcommand{\Var}{\mathop{\mathrm{Var}}\mathopen{}}

\newcommand{\bigo}{\mathop{{}\mathrm{O}}\mathopen{}}

\newcommand{\law}{\mathop{{}\sL}\mathopen{}}

\numberwithin{equation}{section}
\def\cite{\citet*}

\usepackage{environ}
\NewEnviron{equ}
{\begin{equation} 
		\begin{split}
		\BODY
		\end{split} 
\end{equation}}
\AfterEndEnvironment{equ}{\noindent\ignorespaces}

\crefname{equation}{}{}
\crefformat{equation}{\textup{(#2#1#3)}}
\crefrangeformat{equation}{\textup{(#3#1#4)--(#5#2#6)}}
\crefdefaultlabelformat{#2\textup{#1}#3}
\crefname{lemma}{Lemma}{Lemmas}
\crefname{page}{p.}{pp.}

\def\^#1{\relax\ifmmode {\mathaccent"705E #1} \else {\accent94 #1} \fi}
\def\~#1{\relax\ifmmode {\mathaccent"707E #1} \else {\accent"7E #1} \fi}
\def\*#1{\relax#1^\ast}
\edef\-#1{\relax\noexpand\ifmmode {\noexpand\bar{#1}} \noexpand\else \-#1\noexpand\fi}
\def\>#1{\vec{#1}}
\def\.#1{\dot{#1}}
\def\1{\mathop{\mathbf{1}}\mathopen{}}

\def\atop{\@@atop}
\def\emptyset{\varnothing}
\def\epsilon{\varepsilon}

\endlocaldefs

\begin{document}

\begin{frontmatter}
\title{Cram\'er-type Moderate deviations under local dependence}
\runtitle{Moderate deviation by Stein's method}

\begin{aug}
\author[A,C]{\fnms{Song-Hao}
\snm{Liu}\ead[label=e3]{liush@sustech.edu.cn}},
\and
\author[B,D]{\fnms{Zhuo-Song} \snm{Zhang}\ead[label=e2,mark]{zszhangstat@gmail.com}}
\address[A]{Department of Statistics and Data Science,
Southern University of Science and Technology,
\printead{e3}}
\address[C]{Department of Statistics,
The Chinese University of Hong Kong,
}

\address[B]{Department of Statistics and Data Science,
National University of Singapore,
\printead{e2}}
\address[D]{Department of Mathematics,
University of California, San Diego,
}

\end{aug}

\begin{abstract}
      We establish Cram\'er-type moderate deviation theorems for sums of locally
	dependent random variables  and  combinatorial central limit
    theorems. Under some mild exponential moment conditions, 
	optimal error bounds
	and convergence ranges are obtained. Our main results are more general or shaper than the existing results in the literature. 
	 The main results follows from a more general 
 Cram\'er-type moderate deviation theorem for dependent random variables without any boundedness assumptions, which is of independent interest. The proofs couple Stein's method with a recursive argument.
\end{abstract}

\begin{keyword}[class=MSC2020]
\kwd[Primary ]{60F10}
\kwd[; secondary ]{60F05}
\end{keyword}

\begin{keyword}
\kwd{Stein's method}
\kwd{Cram\'er-type moderate deviation}
\kwd{local dependence}
\kwd{combinatorial central
limit theorem}
\kwd{Stein identity}
\end{keyword}

\end{frontmatter}
\section{Introduction}
Moderate deviations estimate the relative errors for distributional approximations. Since
\cite{Cra38}  proved a moderate deviation result for tail probabilities of sums of independent random variables, Cram\'er-type moderate deviation
theorems have been  widely applied to
estimate rare event probabilities. Specially, for independent and identically distributed (i.i.d.) random variables
$X_1, \dots, X_n$ with zero mean and unit variance  satisfying that $\IE e^{t_0 |X_1|} \leq c$ for some $t_0 > 0$, it follows that 
\begin{align*}
	\biggl\lvert\frac{\IP(W_n > x)}{ 1 - \Phi(x) } - 1\biggr\rvert \leq A n^{-1/2} (1 +
	x^{3}) \text{\quad for\quad $0 \leq x \leq a n^{1/6},$}
\end{align*}
where $W_n = (X_1 + \dots + X_n)/\sqrt{n} $, $\Phi(x)$ is the standard normal distribution function, and $A$ and $a$ are positive constants depending only on $t_0$ and $c$. We remark that the range $0 \leq x
\leq an^{1/6}$ and the error term $n^{-1/2}(1 + x^3)$ are optimal for i.i.d. random variables. For
other results on Cram\'er-type moderate deviations, we refer the reader to \cite{Li61P}  and \cite{petrov2012sums} . 

Moderate deviation theorems for independent random variables have been well studied in the literature.  
However, the data may not be independent in the era of big data. It is necessary to develop the
corresponding limit theory for dependent random variables.

In this paper, we focus on Cram\'er-type moderate deviations for sum of locally dependent random variables (see \cref{applocal}) and combinatorial central
limit theorems (see \cref{CBCLT}).
A family of locally dependent random variables means that certain subset of the random variables are
independent of those outside their respective neighborhoods, which is a generalization of
$m$-dependence. 
Although absolute error bounds of normal approximation for sums of locally dependent random variables  have been well studied in the 
literature \citep[see, e.g.,][]{Bal89,Ba89A,Ri94O,Dem96,chenshao2004,fang2019},  
few results for Cram\'er-type moderate deviation theorem for locally dependent random fields have been proved even when assuming that the random variables are bounded. Under certain dependence structures, \cite[]{raic2007}  proved a large deviation result
with some  sophisticated assumptions, which,  however, seems to be too restricted to apply to other applications.
In \cref{thm:9_27_2}, we provide a Cram\'er-type moderate deviation result under local dependence and some mild exponential moment conditions. 

Combinatorial central limit theorem is the central limit theorem for a family of permutation statistics $\sum_{i = 1}^n X_{i, \pi(i)}$, where $n \geq 1$, $\mathbf{X} \coloneqq \{ X_{i,j} : 1 \leq i , j \leq n \}$ is an $n \times n
$ array of random variables, 
and $\pi$ is a uniform random permutation of $\left\{ 1,2,\dots,n \right\}$, independent of
$\mathbf{X}$. Absolute error bounds of normal approximation for $\sum_{i = 1}^n X_{i, \pi(i)}$ have also been well studied
in the literature  
\citep[see][]{Hoe51,Ho78,Goldstein2005,Chen2013From,chen2015error}.  
For relative error bounds, \cite{frolov2019large} obtained
a moderate deviation result under
some Bernstein-type conditions. 
 However, he did not provide the error
 bound. 
In \cref{thm:11_24_1}, we prove a Cram\'er-type moderate deviation result for combinatorial central limit theorems with best possible convergence rates and ranges.

Classical proofs of Cram\'er-type moderate deviations are based on the conjugate method and Fourier
transforms, which perform well when dealing with independent random variables. Nevertheless, it is not easy to apply the Fourier transform without
independence assumptions. Alternatively, Stein's method is a powerful tool in dealing with
dependent structures.  
Since introduced by \cite{stein1972}, Stein's method has been widely applied to prove
optimal
Berry--Esseen bounds and $L_1$ bounds with explicit constant factors for many distributional
approximations \citep{chen2010normal,Cha14}, and moreover, it turns out that Stein's
method can also be used to obtain moderate deviation theorems.
For examples, \cite{Chen2013From}  first applied Stein's method to prove Cram\'er-type moderate deviation
results for normal approximation via Stein identity, and recently,
\cite{shao2018cram}  further obtain a
Cram\'er-type moderate deviation result for nonnormal approximations. In both papers, the authors
made some boundedness assumptions about the random variables of interest. To relax
boundedness assumptions,
\cite{zhang2019cramertype} applied Stein's
method using the exchangeable pair approach to develop a Cram\'er-type
moderate deviation result for unbounded case. However, \cite{zhang2019cramertype}'s result cannot be
applied to deal with locally dependent random variables.

In order to prove \cref{thm:9_27_2} and \cref{thm:11_24_1}, we consider the Stein identity approach of Stein's method. 
Specifically, let $W$ be a random variable, and assume that there exists a random function $\hat{K}(u)$ and a
random variable $R$ such that for all absolutely continuous functions $f$, the following identity holds: 
\begin{equation}
	\IE \{ W f(W) \} = \IE \biggl\{ \int_{-\infty}^{\infty} f'(W + u) \hat{K}(u) du \biggr\} + \IE \{ R f(W) \}. 
    \label{e1}
\end{equation}
The equality \cref{e1} is called Stein identity 
(see Section 2.5 of \cite{chen2010normal}).
 Both $L_1$ bounds and Berry--Esseen bounds via Stein's identity have been well studied in the literature, and we refer the readers to  \cite{chen2010normal} for a detailed survey.  
Based on \cref{e1}, \cite{Chen2013From} proved a Cram\'er-type moderate deviation theorem  for $W$
under the following conditions:
there exists
$\delta_0, \delta_1, \delta_2$ and $\theta$ such that 
\begin{align}
	\begin{aligned}
			& \hat K (u)  = 0 \text{ for $|u| > \delta_0$}, & 
		\lvert \IE \{ \hat{K}_1 \vert W \} - 1 \rvert & \leq \delta_1 (1 + \lvert W \rvert	) , \\
														  & \IE \{ \hat{K}_1 \vert W \} \leq \theta	, 
							&
		\lvert \IE \{R \vert W\} \rvert & \leq \delta_2 (1 + \lvert W \rvert) .  
	\end{aligned}
    \label{eq-con-chen}
\end{align}
However, the conditions may be restricted to apply in some applications. First, the random function $\hat K(u)$
is assumed to be positive and supported on a bounded interval $[-\delta_0, \delta_0]$, where the constant $\delta_0$ is of order $\bigo(n^{-1/2})$ in
some typical applications. Second, the conditional expectations may not be easy to calculate if we
know few on the distribution of $W$. 

To improve \cite{Chen2013From}'s moderate deviation result, we establish a general
Cram\'er-type moderate deviation result (\cref{thm1}) without assuming that the random function
$\hat K(u)$ is positive and supported on a bounded interval, which may be of independent interest for other applications. 
There are several advantages of our result. First, 
optimal error bounds and optimal ranges are obtained for moderate deviations of locally
dependent sums and combinatorial central limit theorems. 
Second, we relax the boundedness assumption and thus
our general theorem can be applied to a much wider class of statistics.

The rest of this paper is organized as follows. We give the result for  locally dependent random
variables in \cref{applocal}. Moderate deviation for combinatorial central limit theorems are
discussed in \cref{CBCLT}. Our
general theorem is given in \cref{sec:main_results}. We prove our general result in
\cref{sec:p_general_thm}. Finally, the proofs
of our results in \cref{applocal,CBCLT} are presented in \cref{sec:general method,sec:p_th3}. Some supplementary materials are given in the appendix. 
    \section{Moderate deviation for sums of locally dependent random variables}\label{applocal}
In this section, we prove a Cram\'er-type moderate deviation theorem for  sums of locally dependent random variables.

We follow the notation in
\cite{chenshao2004}. Let $\mathcal{J}$ be an index set and let $\{ X_i, i \in \mathcal{J} \}$ be
a field of random variables with zero means and finite variances. Let $W = \sum_{i
\in \mathcal{J}} X_i$ and assume that $\Var (W) = 1$. 
For $A \subset \mathcal{J}$, write $X_A = \{ X_i , i \in A\}$, $A^c = \{ j \in \mathcal{J} : j \not\in
A\}$ and denote by $\lvert A \rvert$ the cardinality of $A$. 

We now introduce the following local dependence conditions:  
\begin{enumerate}
	\item [(LD1)] For each $i \in \mathcal{J}$, there exists $A_i \subset \mathcal{J}$ such that
	    $X_i$ is independent of $X_{A_i^c}$.  
    \item [(LD2)] For each $i \in \mathcal{J}$, there exists $B_{i}\subset \mathcal{J}$ such that
        $B_{i}\supset A_{i}$ and  $X_{A_i}$ is independent of $X_{B_i^c}$.  
\end{enumerate}
We note that local dependence satisfying (LD1) and (LD2) is a generalization of $m$-dependence. These local dependence conditions were firstly introduced by \cite{chenshao2004}, and we refer the
reader to other types of local dependence structures in \cite{Bal89,Ba89A,Ri94O,Dem96,fang2019}. 
Absolute error bounds such as $L_1$ bounds and Berry--Esseen bounds for locally dependent random variables have also been well studied in
the literature. For example, in Section 4.7 of \cite{chen2010normal}, an $L_1$ bound was established under (LD1)
and (LD2). \cite{chenshao2004} proved several sharp Berry--Esseen bounds under different local dependence conditions and some polynomial moment
conditions. 
Although Cram\'er-type moderate deviations have been proved for $m$-dependent random variables (see, e.g., \cite{Hei82}), as far as we know, no Cram\'er-type moderate
deviation results have been obtained for locally dependent random variables even for bounded cases.

Let {$N_{i} = \{ j \in \mathcal{J} : 
B_{i} \cap B_{j} \neq\emptyset\}$} and let $\kappa \coloneqq \max_{i \in \mathcal{J}} \lvert N_i \rvert$. 
Let $n = \lvert \mathcal{J} \rvert$. 
Assume that there exist $a_n \geq 1$ and $b \geq 1$ such that 
for all $i \in \mathcal{J}$, 
\begin{align}
	\label{eq:9_19_2}
	\IE \Bigl\{ \exp \Bigl( a_n \sum_{j \in B_i} \lvert X_j \rvert \Bigr) \Bigr\} \leq b. 
\end{align}
We have the following theorem.

\begin{theorem}
	 	 \label{thm:9_27_2}
		 Under \textup{(LD1)} and \textup{(LD2)}, and  assume that \cref{eq:9_19_2} holds. Then
	\begin{equation}
        \biggl\lvert \frac{ \IP [ W > z ] }{1 - \Phi(z)} - 1 \biggr\rvert \leq  C
         \delta_{n} ( 1 + z^3 ) \quad 
		\label{eq:9_27_1}
	\end{equation}
    for $0\leq z\leq c a_{n}^{1/3}\min \{ 1,  \kappa^{-1/3} (1+\theta_{n})^{-2/3}\},$ 
    where $C$ and $c$ are absolute constants and $\delta_{n}=   \kappa^2
	a_n^{-1} ( 1 +  \theta_n^6)$ and $\theta_n = b^{1/2}n^{1/2} a_n^{-1}$.
\end{theorem}
\begin{remark}
    When $a_n$ is of order $\bigo(n^{1/2})$ and $\kappa$ and  $b$ are of order $\bigo(1)$, we have
    $\theta_n = \bigo(1)$ and $\delta_n = \bigo(n^{-1/2})$. Therefore, the error bound in
    \cref{eq:9_27_1} is of order $(1+z^{3})/\sqrt{n}$ and the range is $0\leq z\leq
    c n^{1/6}$.
	Specifically, for i.i.d. random variables $\xi_1, \dots, \xi_n$ satisfying that $\E \xi_1 = 0$,
	$\Var (\xi_1) = 1/n$ and $\IE e^{\sqrt{n} |\xi_1 |} \leq b_0$ for some $b_0 > 0$, we have that
    \cref{eq:9_19_2} holds with $B_{i}=\{i\}$, $a_n = \sqrt{n} $ and $b = b_0$. Hence, \cref{thm:9_27_2} reduces to 
	\begin{align*}
		\biggl\lvert \frac{ \IP( \sum_{i = 1}^n \xi_i > z ) }{1 - \Phi(z)} - 1 \biggr\rvert \leq C n^{-1/2} (1 + z^3) 
		\text{ for $0 \leq z \leq cn^{1/6}$},
	\end{align*}
	where $c,C$ are constants depending only on $b_0$. Thus, \cref{thm:9_27_2} is optimal in the
	sense that it provides optimal error bounds and ranges for sum of i.i.d. random variables. 
\end{remark}

\begin{remark}
	We remark that there are some different dependence structures other than (LD1) and (LD2) in the literature, e.g., decomposable random variables, dependency graphs, and so on.  
For decomposable random variables, \cite[]{raic2007}  proved a large deviation result
with some  sophisticated assumptions, which maybe too strict to apply to other applications. 
\end{remark}

To illustrate that our result gives optimal error bounds and ranges for other settings, we consider the following corollary for $m$-dependent random fields. 
Let $d \geq 1$ and let $\mathbb{Z}^d$ denote the $d$-dimensional space of positive integers. For any $i = (i_1, \dots, i_d), j = (j_1, \dots, j_d) \in \mathbb{Z}^d$, we define the distance by $\lvert i - j \rvert \coloneqq \max_{1  \leq k \leq d} \lvert i_k -
j_k \rvert$, and for $A, B \subset \mathbb{Z}^d$, we define the distance between $A$ and $B$ by $\rho(A, B) = \inf \{ \lvert i - j  \rvert : i \in A, j \in B \}$.
Let $\mathcal{J}$ be a subset of $\mathbb{Z}^d$, and we say a field of random variables $\{ X_i : i \in \mathcal{J} \}$ is an \emph{$m$-dependent random field} if $\{ X_i , i \in A \}$ and $\{ X_j , j \in B \}$ are independent whenever $\rho(A, B) > m$ for any $A, B
\subset \mathcal{J}$. 
If we choose $A_i = \{ j \in \mathcal{J} : \lvert i - j \rvert \leq m \}$, $B_i = \{ j \in \mathcal{J} :
\lvert i - j \rvert \leq 2m \}$, 
then (LD1) and (LD2) are satisfied with $\kappa = (8 m+1)^d$, and \cref{thm:9_27_2} reduces to the following corollary. 

\begin{corollary}
	Let $\{ X_i : i \in \mathcal{J} \}$ be an $m$-dependent random field on $\mathbb{Z}^d$ with $\IE \{ X_i \} = 0$,
	$W = \sum_{i \in \mathcal{J}} X_i$ and $\Var (W) = 1$. If \cref{eq:9_19_2} is satisfied,  
	 then 
	 \cref{eq:9_27_1} holds with $\kappa = (8m + 1)^{d}$. 
	\label{cor1}
\end{corollary}
\begin{remark}
	Under the conditions of \cref{cor1} with $d=1$, Corollary 4.1 of \cite{Hei82} 
reduces to the following result:
\begin{equation}
		\begin{aligned}
			\left|\frac{\IP( \sum^{n}_{i=1} X_{i}>z)}{1-\Phi(z)}-1\right|\leq C 
			n a_{n}^{-3}(1+z^{3})
		\end{aligned}
		\label{eq:1000}
	\end{equation}
    for $0\leq z\leq c a_{n} n^{-1/3}$ where $C$ and $c$ are constants depending only on
    $m$ and $b$. Since $\IE W^{2}=1$, we have $m \sum^{n}_{i=1} \IE X_{i}^{2}> 1 $. On the other hand, by
    \cref{eq:9_19_2}, $\sup_{1\leq i\leq n} \IE X_{i}^{2}\leq C b a_{n}^{-2}$. Thus, $n
    a_{n}^{-2}\geq C_{0}$, for some constant $C_{0}$ depend on $m$ and $d$.  Up this constant
    $C_{0}$,
    \citeauthor{Hei82}'s result is
    not better than ours and  
	the moderate deviations for $m$-dependent random field on $\mathbb{Z}^d$ seems to be new. 
	\label{rem:10}
\end{remark}
\section{Moderate deviation for combinatorial central limit theorems}\label{CBCLT}
Let $n \geq 1$, and let $\mathbf{X} \coloneqq \{ X_{i,j} : 1 \leq i , j \leq n \}$ be an $n \times n $ array of independent random variables with $\IE \{ X_{i,j} \} = a_{i,j}$ and $\Var (X_{i,j}) = c_{i,j}^2$. 
Moreover, assume that
\begin{equation}
	\sum^{n}_{i=1}  a_{i,j}=0 \quad \text{for all $1\leq j\leq n$,}
	\quad \text{ } \quad 
    \sum^{n}_{j=1}  a_{i,j}=0\quad \text{for all $1\leq i\leq n$},
	\label{Ex2}
\end{equation}
and
\begin{equation}
        \frac{1}{n-1} \sum^{n}_{i=1} \sum^{n}_{j=1}  a^{2}_{i,j} + \frac{1}{n} \sum_{i = 1}^n \sum_{j = 1}^n c_{i,j}^2=1.
	\label{eq:2020_11_8}
\end{equation}
Let $\mathcal{S}_{n}$ be the collection of all permutations
over  $[n]:= \left\{ 1,2,\dots,n \right\}$ and let $\pi$ be a random permutation chosen uniformly from
$\mathcal{S}_{n}$ independent of $\mathbf{X}$.   Let  
\begin{equation}
    \begin{aligned}
      W= \sum^{n}_{i=1} X_{i,\pi(i)}.
    \end{aligned}
    \label{eq:sec4_1}
\end{equation}
Combinatorial central limit theorems for $\tilde{W} \coloneqq \sum_{i = 1}^{n} a_{i,\pi(i)}$, which is a special case of $W$, was firstly introduced by \cite{Hoe51}. 
For the random variable $\tilde{W}$, \cite{Goldstein2005} proved a
Berry--Esseen theorem for $\tilde{W}$ by Stein's method and zero--bias
coupling, and  
\cite{Chen2013From} also gives the moderate deviation result of the normal approximation, where the convergence rate and range depend on $\max_{i,j} \lvert a_{ij} \rvert$.  \cite{hu2007cramer}
proved a moderate deviation result for the simple random sample problem, which is an application of the combinatorial central
limit theorems. 
The Berry--Esseen bounds of combinatorial central limit theorems for $W$ was firstly studied by \cite{Ho78}, who proved an error bound  using the concentration inequality approach, and \cite{chen2015error} obtained a new error bound
$451n^{-1}\sum_{i = 1}^n\sum_{j = 1}^n \E \lvert X_{i,j} \rvert^3$ via exchangeable pair approach.  
Recently, \cite{frolov2019large} gave
a Cram\'er-type moderate deviation result for general combinatorial central limit theorems under
some Bernstein type conditions, but the author did not provide the error
bounds.

The following theorem provides a Cram\'er-type moderate deviation result for $W$. 

\begin{theorem}
    \label{thm:11_24_1}
Assume that there exist $\alpha_{n}\geq 1$ and $b\geq1$ such that 
    \begin{equation}
		\max_{1 \leq i , j \leq n} \E \{\exp( \alpha_{n} \left|X_{i,j}\right| ) \}\leq b. 
        \label{eq:11_24_6}
    \end{equation}
    Then 
    \begin{equation}
        \begin{aligned}
            \left \vert  \frac{\IP(W> z)}{1-\Phi(z)}-1  \right \vert \leq C \delta_{n} (1+z^{3}), 
        \end{aligned}
        \label{eq:11_24_7}
    \end{equation}
 for $ 0\leq z\leq c \alpha_{n}^{1/3}\min\{ 1,
 b^{-1} (\theta_{n}^{-1/2}+\theta_{n})^{-1} \},$
    where $C$ and $c$ are absolute constants, $\theta_{n}= n^{1/2}\alpha_{n}^{-1}$ and 
    $\delta_{n}=b^{2} (\alpha_{n}^{-1}+n^{-1/2}) (\theta_{n}^{-2}+\theta_{n}^{6})$.
\end{theorem}

\begin{remark} 
	If $\max_{1\leq i,j\leq n} |X_{i,j}|$ is of order $\bigo(n^{-1/2})$, then we can choose $\alpha_n = \bigo(n^{1/2})$ and
	$b = \bigo(1)$, and  \cref{eq:11_24_7} reduces to 
	\begin{align*}
		\left \vert  \frac{\IP(W> z)}{1-\Phi(z)}-1  \right \vert \leq C n^{-1/2}(1 + z^3)
	\end{align*}
	for $0 \leq z \leq c n^{1/6}$ for some constants $c , C > 0$. 
\end{remark}
\begin{remark}
	In  \cite[]{Chen2013From}, the authors proved a moderate deviation for the case where $X_{i,j} = a_{i,j}$ is nonrandom. Specially, our result recovers (4.1) in \cite[]{Chen2013From}. 
\end{remark}

\section{A general theorem via Stein identity}%
\label{sec:main_results} 
In this section, we proof a general theorem for dependent random variables, which will be used to
prove \cref{thm:9_27_2,thm:11_24_1}. The theorem is based on Stein identity, and it is also of independent interest and can be applied to many other applications. 
Let $W$ be the random variable of interest satisfying the Stein identity \cref{e1} with a random function
$\hat{K}(u)$ and a random variable $R$. To give our general theorem, we first introduce the
following notation. For $t \geq 0$ and for $u \in \IR$, let 
\begin{align}
    K(u)&=\IE \{\hat{K}(u)\}, \quad 
	\hat{K}_1 =  \int_{-\infty}^{\infty} \hat{K}(u) du , \label{ek1}\\
	\hat{K}_{2,t} & =  \int_{-\infty}^{\infty} |u| e^{t |u|} \vert \hat{K}(u) \vert du , \label{ek2}\\
	\hat{K}_{3,t} & = \int_{ \lvert u \rvert \leq 1 }^{} e^{ 2 t |u| } \bigl( \hat{K}(u) - K(u)
	\bigr)^2 du ,\label{eq:sec2_2}\\
	\hat{K}_{4,t} & = \int_{ \lvert u \rvert \leq 1 }^{} |u| e^{ 2 t |u| } \bigl( \hat{K}(u) - K(u)
	\bigr)^2 du \label{eq:sec2_3},  
\end{align}
and 
\begin{align}
	M_t & = \int_{|u| \leq 1} e^{t|u|} \lvert {K}(u) \rvert du . 
	\label{eq:sec2_1}
\end{align}
For any $\beta \geq 0$ 
and $t \geq 0$, let
\begin{align}
	\Psi_{\beta, t}(w)  = 
	\begin{cases}
		e^{t w} + 1 & \text{ if $w \leq \beta$, }\\
		2 e^{ t \beta } - e^{ t ( 2 \beta - w ) } + 1 & \text{ if $w > \beta$. }
	\end{cases}
\label{eq:sec1_1}
\end{align}
We remark that the function $\Psi_{\beta,t}$ is a smoothed version of the truncated exponential
function, which plays an important role in relaxing the boundedness assumption in applications.  

Our general result is based on the following conditions: 
\begin{itemize}
	\item [(A1)]  Assume that there exist constants $m_0 > 0, \rho > 0$ and $r_j \geq 0, \tau_j \geq 0$ for $j = 0, 1, \dots, 4$ such that for all
		$\beta, t\in [0, m_0]
		$, 
        \begin{align}
		\IE \{ |R| \Psi_{\beta, t}(W) \}                        
		& \leq r_0 (1 + t^{\tau_0}) \IE \{ \Psi_{\beta, t}(W)  \}, \label{el3a}\\
		\IE \{ \lvert \IE\{\hat{K}_1\vert W\} - 1 \rvert \Psi_{\beta, t}(W) \} 
		& \leq r_1 (1 + t^{\tau_1}) \IE \{ \Psi_{\beta, t}(W)  \}, \label{el3b}\\
		\IE \{ \hat{K}_{2,t} \Psi_{\beta, t}(W) \}               
		& \leq r_2 (1 + t^{\tau_2}) \IE \{ \Psi_{\beta, t}(W)  \}, \label{el3c}\\
            \IE \{ \hat{K}_{3,t}  \Psi_{\beta,t}(W) \} & \leq  r_{3}( 1 + t^{\tau_3} ) \IE \{
			\Psi_{\beta, t}(W)  \}, \label{eq:9_20_11}\\
                \IE \{ \hat{K}_{4,t}  \Psi_{\beta,t}(W) \} & \leq  r_4( 1 + t^{\tau_4} ) \IE \{
			\Psi_{\beta, t}(W)  \}\label{eq:9_20_12} 
	   \end{align}
	   and 
	   \begin{align}
			\sup_{0 \leq t \leq m_0}M_t & \leq \rho.
	\label{eq:r_3}
	   \end{align}
   \end{itemize}
       We now state our general result. 
\begin{theorem}
	\label{thm1}
	Under condition (A1).
	Let $\tau = \max\{ \tau_0 + 1, \tau_1 + 2, \tau_2 + 3, \tau_3 + 1, \tau_4 + 1 \}$ and let 
	\begin{align}
	\label{eq-z0}
	z_0 =  \min \bigl\{   m_0 , 0.02 e^{-\tau/2}( r_0^{1/(\tau_0 + 1)}+ 
	r_1^{1/(\tau_1 + 2)}+ r_2^{1/(\tau_2 + 3)})^{-1} \bigr\}.
	\end{align}
	We have  
	\begin{equ}
		\left \vert\frac{ \IP [ W > z ] }{1 - \Phi(z)}-1  \right \vert \leq   {  \biggl(
        \frac{4}{\delta(m_0)}  + C (150^{\tau} + \rho ) e^{\tau^{2}/2}\biggr) \delta (z)} 
        \label{et1}
	\end{equ}
    for 
	$0 \leq z \leq z_0$, 
	where $\Phi(z)$ is the standard normal distribution function and  $C$ is an absolute constant, and 
   \begin{multline}
	   \label{eq-deltat}
	   \delta(z) = r_0 (1 + z^{\tau_0 + 1}) + r_1 ( 1  + z^{\tau_1 + 2} ) + r_2 ( 1 + z^{\tau_2
	   + 3}) \\+ r_3(1 + z^{\tau_3 + 1}) + r_4^{1/2} (1 + z^{\tau_4 + 1}). 
   \end{multline}
\end{theorem}
	We give some remarks on our general result. 
	\begin{remark}
		\citet{Chen2013From} proved a moderate deviation for Stein identities under a boundedness assumption \cref{eq-con-chen}.  
		On that basis, for $0 \leq t \leq \delta_0^{-1}$, it can be shown that \cref{el3c} is satisfied with $r_2 = 3 \theta \delta_0$. 
		Moreover, for all $t, \beta \in(0, \delta_0^{-1})$, one can verify (see, e.g., (5.5) and (5.6) of \cite{Chen2013From}) that there exists a constant $C > 0$ depending only on $\theta$ such that 
		\begin{align*}
			\IE \bigl\{ \lvert W \rvert \Psi_{\beta, t}(W) \bigr\} \leq C (1 + t) \IE \{ \Psi_{\beta, t}(W) \}. 
		\end{align*}
		Thus, we have that \cref{el3a,el3b} are satisfied with $r_0 = C' \delta_2, r_1 = C' \delta_1$ and $\tau_0 = \tau_1 = 1$, where $C' > 0$ is an absolute constant.  
		Therefore, by \cref{thm1}, we have that \cref{et1} holds with $m_0 = \delta_0^{-1}$, $\rho = \theta$, $\tau = 3$, $\tau_0 = \tau_1 = 1$, $\tau_2 = \tau_3 = \tau_4 = 0$, $r_0 = C' \delta_2$, $r_1 = C'\delta_1$, $r_2 = 3 \theta \delta_0$ and 
		\begin{align*}
			r_3 = 8 \int_{ |u| \leq \delta_0 } \IE \{ (\hat{K}(u) - K(u))^2  \} du, \quad 
			r_4 = 8 \delta_0 \int_{ |u| \leq \delta_0 } \IE \{ (\hat{K}(u) - K(u))^2  \} du. 
		\end{align*}
		Note that this result involves two terms $r_3$ and $r_4^{1/2}$ that did not appear in \cite{Chen2013From}. 
		However, in many applications, both $r_3$ and $r_4^{1/2}$ have the same order as $\delta_0$. 
		This shows that our result \cref{thm1} covers Theorem 3.1 of \cite{Chen2013From} with the cost of two additional terms. 
	\end{remark}
	\section{Proof of \autoref{thm1}}\label{sec:p_general_thm}%
	In this section, we provide the proof of \cref{thm1}. Our proof is novel in two ways.  
	On one hand, the proof of \cref{thm1} is a combination of Stein's method and a recursive
	method. The recursive method has been applied to obtain optimal Berry--Esseen bounds for both
	univariate and multivariate normal approximations; see \cite{Ra03N,Rai19a} and \cite{chen_adrian_xia2020}
	for examples.  
	On the other hand, use a truncated exponential function to control tail probabilities.
	It is known that exponential-type tail probabilities play a
crucial role in the proof of Cram\'er-type moderate deviations. In \cite{Chen2013From} and
\cite{shao2018cram}, the authors used  exponential
	functions directly to prove upper bounds for such tail probabilities. In the present paper, 
	a key observation is that the exponential function can be replaced by a {smoothed}
    truncated exponential function $\Psi_{\beta,t}$ (defined in \cref{eq:sec1_1} in \cref{sec:main_results}) when proving exponential-type
	tail probabilities, and the function $\Psi_{\beta,t}$ plays an important role in relaxing the
	boundedness assumption when applying our general theorem.

This section is organized as follows. 
We first develop two preliminary lemmas, 
\cref{expbound,l-BE}, whose proofs are postponed to \cref{sec_proof_lemma}. In \cref{expbound},  we
establish an upper bound of the ratio for the expectation of a smoothed indicator function, and we provide a 
Berry--Esseen bound under \cref{e1} in \cref{l-BE}.  The proof of \cref{thm1} is given in
\cref{proof_thm}, where we apply \cref{expbound,l-BE} and a smoothing inequality. 
\subsection{Preliminary lemmas}
\label{sec:proof_of_main_result}

We first introduce some notation. 
Let $Z \sim N(0,1)$,  $\phi(w) = (1 / \sqrt{2\pi} ) e^{-w^2/2}$ and $\Phi(w) = \int_{-\infty}^w \phi(t) dt$. In what follows, we write $Nh = \IE \{ h(Z) \}$ for any function $h$. 
For any $z \geq 0$ and $\epsilon > 0$, let 
\begin{align*}
	h_{z, \epsilon}(w) = 
	\begin{cases}
		1                         & \text{if $w \leq z$}, \\
		0                         & \text{if $w > z + \epsilon$}, \\
		1 + \epsilon^{-1} (z - w) & \text{if $z < w \leq z + \epsilon$}.
	\end{cases}
\end{align*}
Let 
\begin{equ}
	C_0 = \sup_{0 \leq z \leq {z_0}} \biggl\lvert \frac{ \IP [ W > z ] - (1-\Phi(z)) }{\delta(z)
	(1-\Phi(z))} \biggr\rvert, 
	\label{deltan}
\end{equ}
where $\delta(z)$ is given in \cref{eq-deltat}.
The following lemma gives a relative error for the test function $h_{z,\epsilon}$.
\begin{lemma}
\label{expbound}
Assume that condition (A1) holds and $z_{0}$ in \cref{eq-z0} satisfies $z_{0}\geq 8$.
	Let $z$ be a fixed real number satisfying $8 \leq z \leq z_0$, and let 
	${\epsilon \coloneqq \epsilon(z) = 40 e^{\tau/2}r_2}(1 + z^{\tau_2})$. We have 
	\begin{equation}
		\frac{\lvert \IE \{ h_{z,\epsilon}(W) \} - Nh_{z,\epsilon} \rvert}{\delta(z) (1 - \Phi(z))} 
        \leq   0.75 \Bigl(C_0 + \frac{1}{\delta(m_0)}\Bigr) + (184 + 2\rho) e^{\tau/2}+ (150 e^{\tau/2})^{\tau}.
        \label{eHbound}
	\end{equation}
\end{lemma}
We also need to develop a Berry--Esseen bound to prove \cref{thm1}. The following lemma is a slight modification of Theorem 2.1 in  \cite{chen_adrian_xia2020}, because the Stein identity \cref{e1} in our paper involves an additional error
term $\IE \{Rf(W)\}$ compared with that in \cite{chen_adrian_xia2020}. 
The proof is given in \cref{sec_proof_lemma}, where we used a similar argument to the proof of Theorem 2.1 in  \cite{chen_adrian_xia2020}.
\begin{lemma}
	\label{l-BE}
	Let $W$ be a random variable satisfying $\IE  W  = 0$ and $\IE  W^2  = 1$. Assume that \cref{e1} and (A1) hold. Then, 
	\begin{align*}
		\sup_{z \in \IR} \lvert \IP ( W \leq z ) - \Phi(z) \rvert
		& \leq 4 r_0 + 4 r_1 + 28 r_2 + 20 r_3 + 13 r_4^{1/2}. 
	\end{align*}
\end{lemma}
\subsection{Proof of \autoref{thm1}}\label{proof_thm}
\begin{proof}
	[Proof of \cref{thm1}]
It follows from \cref{l-BE} that 
	\begin{equation}
		\sup_{0 \leq z \leq 9}\biggl\lvert \frac{ \IP (W \leq z) - \Phi(z) }{ \delta(z) (1 - \Phi(z)) } \biggr\rvert
		\leq \frac{ 28 }{ 1 - \Phi(9) } . 
        \label{ePPe}
	\end{equation}
    It now suffices to prove \cref{et1} for the case $9 \leq z \leq z_0$. When $z_{0}\geq 8$, from
    \cref{eq-z0}, we have 
    \begin{equation}
        \begin{aligned}
            0.02 e^{-\tau/2} \min \bigl\{  r_0^{-1/(\tau_0 + 1)},
    r_1^{-1/(\tau_1 + 2)}, r_2^{-1/(\tau_2 + 3)} \bigr\}\geq 8,
        \end{aligned}
        \label{eq:5_3}
    \end{equation}
    and one can verify that  
    \begin{equation}
        \begin{aligned}
            \max\{r_{0},r_{1},r_2\} \leq 0.02 e^{-\tau/2}.
       \end{aligned}
        \label{eq:ass_1}
    \end{equation}
    Next, we use a smoothing inequality to prove the upper bound for the case $9\leq z\leq
    z_{0}$. 
	Let $\epsilon \coloneqq	 \epsilon(z) = 40 e^{\tau/2}r_2(1 + z^{\tau_2})$.
	By the following well-known inequalities: 
	\begin{equation}
		\frac{1}{1 + x} \phi(x) \leq 1 - \Phi(x) \leq \frac{1}{x} \phi(x) \quad \text{ for all $x \geq 1$, }
        \label{eNorm}
	\end{equation}
    we have
    \begin{equation}
        \label{eq-51}
        \phi(z - \epsilon)  \leq  e^{ z \epsilon} \phi(z) \leq e^{ z \epsilon} (1 + z) ( 1 - \Phi(z) ). 
    \end{equation}
    By \cref{eq-z0}, we have ${z_0} \leq 0.02 e^{-\tau/2} r_2^{-1/(\tau_2+3)}$, in other word 
    \begin{equation}
        \begin{aligned}
            r_2 \leq
            \Bigl(\frac{0.02e^{-\tau/2}}{z_{0}}\Bigr)^{\tau_{2}+3}.
        \end{aligned}
        \label{eq:2_1}
    \end{equation}
   It  then follows that for $z_0 \geq 8$, 
    \begin{equation}
		\begin{aligned}
			z_0 \epsilon 
			& \leq 40 e^{\tau/2} r_2 z_0 + 40 e^{\tau/2} r_2 z_0^{1 + \tau_2}\\ 
            & \leq 40 e^{\tau/2} \Bigl(\frac{0.02e^{-\tau/2}}{z_{0}}\Bigr)^{\tau_{2}+3} z_0 + 40 e^{\tau/2} \Bigl(\frac{0.02e^{-\tau/2}}{z_{0}}\Bigr)^{\tau_{2}+3} z_0^{1 + \tau_2} \\
            & \leq \frac{80 e^{\tau/2}(0.02 e^{-\tau/2})^{3}}{z_{0}^{2}} \\
			& \leq  0.03, 
		\end{aligned}
		\label{eq:zebound}
    \end{equation}
    where we use the fact that $\tau\geq 3$.
	Thus, it follows that 
	\begin{equation}
		e^{z_0 \epsilon} \leq 1.05. 
        \label{eq-52}
	\end{equation}
    Also, note that for $8 \leq z \leq z_0$,
    \begin{equation}
        \begin{aligned}
            (1 + z^{\tau_2}) \leq 2 z^{\tau_{2}}\leq  \frac{1}{256} z^{\tau_{2}+3} \leq
    0.004(1 + z^{\tau_2 + 3}),
        \end{aligned}
        \label{eq:5_4}
    \end{equation}
    and by \cref{eq:2_1},
    \begin{equation}
        r_2(1 + z^{\tau_2 + 3}) \leq(0.02e^{-\tau/2})^{\tau_{2}+3}
    \frac{1+z_{0}^{\tau_{2}+3}}{z_{0}^{\tau_{2}+3}}\leq  0.04 e^{-\tau/2}.
        \label{eq:5_5}
    \end{equation}
 Then, by \cref{eq:5_4,eq:5_5}, we have 
	for $8 \leq z \leq z_0$,
	\begin{equation}
		\begin{split}
			\epsilon(z) & =  40 e^{\tau/2} r_2(1 + z^{\tau_2}) \\
						& \leq 0.16 e^{\tau/2} r_2(1 + z^{\tau_2 + 3})\\
						& \leq 0.1. 
		\end{split}
		\label{eq-eps-bound}
	\end{equation}
    Noting that $1 + z \leq 1.25 {z_0}$ for all $8 \leq z \leq {z_0}$, by
    \cref{eq-51,eq:zebound,eq-52} we have 
    \begin{align*}
        \begin{split}
            \Phi(z) - \Phi(z - \epsilon) 
            & \leq \epsilon \phi(z - \epsilon) \\
            & \leq \epsilon e^{ z \epsilon } (1 + z) ( 1 - \Phi(z) )\\
            & \leq 1.25  {z_0} \epsilon e^{{z_0} \epsilon} (1 - \Phi(z))\\
            & \leq 0.05 ( 1 - \Phi(z )) , 
        \end{split}
    \end{align*}
	and thus, 
	\begin{equation}
		( 1 - \Phi(z - \epsilon) ) \leq (1 - \Phi(z)) + ( \Phi(z) - \Phi(z - \epsilon) ) \leq 1.05 ( 1 - \Phi(z) ). 
        \label{ePPc}
	\end{equation}
	On the other hand, 
	for $8 \leq z \leq {z_0}$, we have  
	\begin{equation}
        (1 + z)(1 + z^{\tau_2})\leq 1.25^{2} z^{\tau_{2}+1} \leq \frac{1.25^2}{64} z^{\tau_{2}+3}\leq 0.04 (1 + z^{\tau_2 + 3}).
        \label{ezz21}
	\end{equation}
    Then, by \cref{eq-52,ezz21}, we have
	\begin{equ}
		\Phi(z) - \Phi(z - \epsilon)
		& \leq \epsilon ( 1 + z ) e^{z \epsilon} (1 - \Phi(z) ) \\
        & \leq 40 r_2 e^{\tau/2} e^{z \varepsilon} (1 + z^{\tau_2})(1 + z) ( 1 - \Phi(z) )\\
		& \leq 2 r_2 e^{\tau/2} (1 + z^{\tau_2 + 3}) (1 - \Phi(z)),
        \label{ePPb}
	\end{equ}
	and similarly, 
	\begin{equation}
		\Phi(z + \epsilon) - \Phi(z) \leq 2 r_2 e^{\tau/2}(1 + z^{\tau_2 + 3}) (1 - \Phi(z)). 
        \label{ePPa}
	\end{equation}
	Recall that $C_0$ is defined as in \cref{deltan}. By \cref{eHbound,ePPa}, we have for $8 \leq z \leq {z_0}$, 
	\begin{equation}
		\begin{split}
            \MoveEqLeft \IP [ W \leq z ] - \Phi(z) \\
		& \leq \IE \{ h_{z, \epsilon}(W) - Nh_{z, \epsilon} \} 
		+ \Phi(z + \epsilon) - \Phi(z) \\
		  &
                        \leq   
						\bigl( 0.75(C_0 + \delta(m_0)^{-1}) + (186 + 2 \rho) e^{\tau/2} + (150 e^{\tau/2})^{\tau}\bigr) \delta(z) (1 - \Phi(z)).
        \label{eUpper}
		\end{split}
	\end{equation}
	Let 
	$
		\epsilon' = 40 e^{\tau/2} r_2(1 + (z - \epsilon)^{\tau_2}). 
	$
	By \cref{eHbound} with replacing $z$ and $\epsilon$ by $z - \epsilon$ and $\epsilon'$, respectively, and by \cref{ePPc}, we have for $9 \leq z \leq {z_0}$, 
	\begin{equation}
		\begin{split}
            \MoveEqLeft  \vert \IE \{ h_{z - \epsilon, \epsilon'}(W) - Nh_{z - \epsilon, \epsilon'}
		\}\vert \\
		& 
		\leq \bigl( 0.75(C_0 + \delta(m_0)^{-1}) + (184 + 2\rho) e^{\tau/2} + (150 e^{\tau/2})^{\tau}\bigr)
					\delta(z) (1 - \Phi(z - \epsilon)) 
        \\
		& 
		\leq \bigl( 0.8(C_0 + \delta(m_0)^{-1}) + (194 + 2.1 \rho) e^{\tau/2} + 1.05(150 e^{\tau/2})^{\tau}\bigr)
					\delta(z) (1 - \Phi(z )) 
        \label{eHbound2}
		\end{split}
	\end{equation}
	Thus, by \cref{eHbound2,ePPb}, we have for $9 \leq z \leq {z_0}$, 
	\begin{equation}
		\begin{split}
            \MoveEqLeft[1] \IP [ W \leq z ] - \Phi(z)\\ 
        & 
                    \geq \IE \{ h_{z - \epsilon, \epsilon'} (W)  \} - N h_{z  - \epsilon, \epsilon'} - \bigl( \Phi(z) - \Phi( z- \epsilon ) \bigr) \\
        & 
		\geq - \bigl( 0.8(C_0 + \delta(m_0)^{-1}) + (196 + 2.1 \rho) e^{\tau/2}+ 1.05(150 e^{\tau/2})^{\tau} \bigr)
					\delta(z) (1 - \Phi(z )) .
        \label{eLower}
		\end{split}
	\end{equation}
	By \cref{eUpper,eLower}, we have for $9 \leq z \leq {z_0}$, 
	\begin{equation}
        \begin{split}
            		\MoveEqLeft \bigl\lvert  \IP (W \leq z) - \Phi(z)  \bigr\rvert\\
					& \leq  \bigl( 0.8(C_0 + \delta(m_0)^{-1}) + (196 + 2.1 \rho) e^{\tau/2}+ 1.05(150 e^{\tau/2})^{\tau} \bigr)
					\delta(z) (1 - \Phi(z ))  . 
                    \label{ePPd}
        \end{split}
	\end{equation}
	Moving $\delta(z) (1 - \Phi(z ))$ in \cref{ePPd} to the left-hand side (LHS) and 
	taking the supremum over $9 \leq z \leq z_0$, and by \cref{ePPe}, we have 
	\begin{equation}
		\label{5.19}
        C_0 
		\leq  \bigl( 0.8(C_0 + \delta(m_0)^{-1}) + (530 + 2.1 \rho) e^{\tau/2} + 1.05(150 e^{\tau/2})^{\tau}\bigr)
					   +  \frac{ 28 }{ 1 - \Phi(9) }. 
    \end{equation}
	Solving the recursive inequality \cref{5.19}, we obtain 
	\begin{equation*}
        C_0 \leq  4 \delta(m_0)^{-1}  + C (150^{\tau} + \rho) e^{\tau^2/2}, 
	\end{equation*}
	which proves \cref{et1}. 
\end{proof}
\subsection{Proofs of \texorpdfstring{\cref{expbound,l-BE}}{Lemmas 5.1 and 5.2}}\label{sec_proof_lemma}
{In order to prove \cref{expbound}, we need to prove the following lemmas. 
Recall that $\Psi_{\beta,t}(w)$ is defined as in \cref{eq:sec1_1}.
The first lemma gives an upper bound for the truncated exponential moment $\IE \{ \Psi_{\beta,
t}(W)\}$.}

\begin{lemma}[Exponential bound]
    \label{l5.1}
    Assume that conditions \cref{el3a,el3b,el3c} hold and $z_{0}$ in \cref{eq-z0} satisfies $z_{0}\geq 8$. We have  
	\begin{equation}
		\begin{aligned}
				\IE \{ \Psi_{\beta, t}(W)\} \leq 4 e^{t^2/2} \quad \text{for $0 \leq t \leq z_0$},
		\end{aligned}
		\label{eq:rmexp}
	\end{equation}
where $\Psi_{\beta,t}(w)$ is defined in \cref{eq:sec1_1}.
\end{lemma}
\begin{proof}
Since $z_{0}\geq 8$, we have \cref{eq:ass_1} holds.
Let 
\begin{align*}
	\delta_1 (t) = r_0 (1 + t^{\tau_0 + 1} ) + r_1 (1 + t^{\tau_1 + 2}) + r_2 ( 1 + t^{\tau_2 + 3} ) \quad \text{for $t \geq 0$}.
\end{align*}
Then, recalling that $\tau \geq 3$ and that $z_0$ is defined in \cref{eq-z0}, by a similar
argument to that in \cref{eq:5_5}, we have 
$\delta_1 (z_0) \leq 0.1 e^{-\tau/2} \leq 0.1$.
We prove a more general result as follows: for all $0 \leq t \leq  m_0$, 
	\begin{equation}
        \label{eq:l5.1}
		\IE \{ \Psi_{\beta, t}(W) \} \leq 2 e^{t^2/2 + 2 \delta_1(t)}, 
	\end{equation}
	which, together with the fact that $\delta_1(z_0) \leq 0.1$,  implies \cref{eq:rmexp} immediately. 

	It now suffices to prove \cref{eq:l5.1}.
For $t \geq 0$, let $h(t) = \IE \{ \Psi_{\beta, t}(W) \}$. As $\Psi_{\beta, t}(w) \leq 2 e^{ t \beta } + 1$, then $h(t) < \infty$ for all $0 \leq t \leq  m_0$. 
Write 
\begin{align*}
	\Psi'_{\beta,t}(w) = \frac{\partial}{\partial w} \Psi_{\beta, t}(w), \quad \Psi''_{\beta,t}(w) = \frac{\partial^2}{\partial w^2} \Psi_{\beta, t}(w). 
\end{align*}
By the definition of $\Psi_{\beta,t}(W)$, 
\begin{align}
	\label{ePsit}
	\frac{\partial}{\partial t} \Psi_{\beta, t}(w)
	& = 
	\begin{cases}
		w e^{tw} & \text{if $w \leq \beta$, }\\
		2 \beta e^{t \beta} - ( 2\beta - w ) e^{t (2\beta - w)} & \text{if $w > \beta$},
	\end{cases}\\
	\label{ePsiw}
	\Psi_{\beta,t}'(w) 
	& = 
	\begin{cases}
		\mathrlap{t e^{ tw }}\hphantom{2 \beta e^{t \beta} - ( 2\beta - w ) e^{t (2\beta - w)}} & \text{if $w \leq \beta$}, \\
		t e^{ t (2 \beta - w) } & \text{if $w > \beta$,}
	\end{cases}
	\intertext{and}
	\label{ePsiww}
	\Psi_{\beta,t}''(w) 
	& = 
	\begin{cases}
		\mathrlap{t^2 e^{ tw }}\hphantom{2 \beta e^{t \beta} - ( 2\beta - w ) e^{t (2\beta - w)}} & \text{if $w \leq \beta$}, \\
		- t^2 e^{ t (2 \beta - w) } & \text{if $w > \beta$}.
	\end{cases}
\end{align}
Also, 
\begin{equation}
	\frac{\partial}{\partial t} \Psi_{\beta, t}(w) 
	\leq w (\Psi_{\beta, t}(w) - 1), \quad 
	\Psi_{\beta,t}' (w) 
	\leq t \Psi_{\beta, t}(w). 
    \label{ePsitt}
\end{equation}

By \cref{ePsit} and the first inequality of \cref{ePsitt}, it follows that 
\begin{equ}
	h'(t) 
    & = \IE \biggl\{ \frac{\partial}{\partial t} \Psi_{\beta, t}(W) \biggr\} \leq \IE \bigl\{ W
	(\Psi_{\beta, t}(W) - 1) \bigr\}. 
    \label{el31}
\end{equ}
By \cref{e1,ePsitt} and noting that $\IE W = 0$ and $|\Psi_{\beta,t}(w) - 1| \leq \Psi_{\beta,t}(w)$ for all $w \in \IR$, we have 
\begin{equ}
	\IE \bigl\{ W (\Psi_{\beta, t}(W)-1) \bigr\}  
	& = \IE \biggl\{ \int_{-\infty}^{\infty} \Psi_{\beta,t}'(W + u) \hat{K}(u) du \biggr\} + \IE \bigl\{ R (\Psi_{\beta, t}(W) - 1) \bigr\}\\
	& =  \IE \biggl\{ \int_{-\infty}^{\infty} \Psi_{\beta,t}'(W) \hat{K}(u) du \biggr\} + \IE \bigl\{ R (\Psi_{\beta, t}(W) - 1) \bigr\}\\
    & \quad +  \IE \biggl\{  \int_{-\infty}^{\infty} \bigl( \Psi_{\beta,t}' (W + u) -
	\Psi_{\beta,t}'(W)\bigr) \hat{K}(u) du \biggr\}\\
    & \leq  t h(t) + t \IE \{ \vert \IE\{\hat{K}_1\vert W\} - 1 \vert \Psi_{\beta, t}(W) \} + \IE \{ \lvert R \Psi_{\beta, t}(W)  \rvert \} \\
	& \quad +  \IE \biggl\{  \int_{-\infty}^{\infty} \bigl( \Psi_{\beta,t}' (W + u) -
	\Psi_{\beta,t}'(W)\bigr) \hat{K}(u) du \biggr\}.
    \label{el32}
\end{equ}
By \cref{ePsiww}, for all $w \in \IR$, we have
\begin{equation}
    \begin{aligned}
\left|\Psi_{\beta, t}^{\prime}(w+u)-\Psi_{\beta, t}^{\prime}(w)\right| & \leq|u| \sup _{s \leq|u|}\left|\Psi_{\beta, t}^{\prime \prime}(w+s)\right| \\
& \leq|u| t^{2} \sup _{s \leq|u|} \Psi_{\beta, t}(w+s) \\
& \leq|u| t^{2} e^{t|u|} \Psi_{\beta, t}(w).
   \end{aligned}
    \label{eq:5_6}
\end{equation}
By \cref{ek2,el3c,eq:5_6}, we have that the last term of \cref{el32} can be bounded by 
\begin{equ}
	\MoveEqLeft 
	\biggl\lvert  \IE \biggl\{  \int_{-\infty}^{\infty} \bigl( \Psi_{\beta,t}' (W + u) -
	\Psi_{\beta,t}'(W)\bigr) \hat{K}(u) du \biggr\} \biggr\rvert \\
	& \leq t^2 \IE \{ \hat{K}_{2,t} \Psi_{\beta, t}(W) \}  \leq r_2 t^2 ( 1 + t^{\tau_2} ) \IE \{ \Psi_{\beta, t}(W) \} . 
    \label{el33}
\end{equ}
Substituting \cref{el3a,el3b,el33,el32} into \cref{el31} yields 
\begin{align*}
	h'(t) \leq t h(t) + r_0  (1 + t^{\tau_0}) h(t)  + r_1 t(1 + t^{\tau_1}) h(t)  + r_2 t^2 (1 + t^{\tau_2}) h(t) . 
\end{align*}
Solving the differential inequality yields
\begin{equation}
    \begin{aligned}
        h(t)\leq 2 \exp \biggl\{\frac{t^{2}}{2}+ r_{0} \biggl(t+\frac{t^{\tau_{0}+1}}{\tau_{0}+1}\biggr)+
            r_{1} \biggl(\frac{t^{2}}{2}+\frac{t^{\tau_{1}+2}}{\tau_{1}+2}\biggr)+ r_{2}
        \biggl(\frac{t^{3}}{3}+\frac{t^{\tau_{2}+3}}{\tau_{2}+3}\biggr)\biggr\}.
    \end{aligned}
    \label{eq:5_1}
\end{equation}
Since $\tau_{0},\tau_{1},\tau_{2}\geq 0$, by Young's inequality, we have 
\begin{equation}
    \begin{aligned}
		t+\frac{t^{\tau_{0}+1}}{\tau_{0}+1}& \leq 2(1+t^{\tau_{0}+1}),\\
		\frac{t^{2}}{2}+\frac{t^{\tau_{1}+2}}{\tau_{1}+2}& \leq (1+t^{\tau_{1}+2}),\\
		\frac{t^{3}}{3}+\frac{t^{\tau_{2}+3}}{\tau_{2}+3}& \leq (1+t^{\tau_{2}+3}).
    \end{aligned}
    \label{eq:5_2}
\end{equation}
Combining \cref{eq:5_1,eq:5_2} yields \cref{eq:l5.1}, 
as desired. 
\end{proof}
{The following lemma establish an error bound for differences between tail probabilities 
of $W$ and $Z$. This is one of the key observations in our proof, since we can give a bound  between
$\IP(W> z)$ and $\IP(Z> z)$ for $z$ slightly large than $z_{0}$, which is very important in the
     recursive argument.}
\begin{lemma}
	\label{ll2}
	Assume that the conditions in \cref{expbound} hold. For $8 \leq z \leq {z_0}$, $0 \leq \epsilon \leq 2$, $\lvert u \rvert \leq 1$ and $u \wedge 0 \leq s \leq u \vee 0$, we have  
	\begin{equation}
		\lvert \IP [ W + s > z ] - \IP [Z + s > z] \rvert 
		\leq 2 e^{\tau/2}  e^{z |u|} \delta(z) (1 - \Phi(z)) \bigl( C_0 + c_0 \bigr)
        \label{ePWz}
	\end{equation}
	and 
	\begin{equation}
		\lvert \IP [ W + s > z + \epsilon ] - \IP [Z + s > z + \epsilon] \rvert\leq 2 e^{\tau/2}  e^{z |u| + z \epsilon} \delta(z) (1 - \Phi(z)) \bigl( C_0 + c_0 \bigr), 
        \label{ePWy}
	\end{equation}
 where $C_0$ is defined as in \cref{deltan}, $\tau$ is as in \cref{thm1}, 
and $c_{0} \coloneqq 1/\delta(m_{0}) + (150 e^{\tau/2})^{\tau}$. 
\end{lemma}
\begin{proof}[Proof of \cref{ll2}]
We first introduce some inequalities. For $z \geq 8$ and $0 \leq a \leq 3$, we have  
\begin{equ}
	1 - \Phi( z - a ) & \leq \frac{ 1 }{ z - a } \phi(z - a) \leq \frac{ e^{za} }{   z - a } \phi(z) \\
						& \leq {\frac{ (1 + z^2) e^{za} }{ z ( z - a )	 } ( 1 -
						\Phi(z) )}\leq 2 e^{z a} ( 1 - \Phi(z) ).  
	\label{ePhi}
\end{equ}
Moreover, noting that
\begin{equation*}
    (1 + (z + 3)^\ell) \leq 1.1 (z+3)^{\ell}\leq 1.1 \times  1.375^{\ell} z^{\ell}\leq e^{\tau/2} ( 1 + z^\ell ) \quad \text{ for all $1 \leq \ell \leq \tau$ and $z \geq 8$}, 
\end{equation*}
and by the definition of $\delta(z)$ as in \cref{thm1}, 
we have 
\begin{equation}
	\label{ez3}
    \delta(z + a) \leq e^{\tau/2} \delta(z) \quad \text{for all $0\leq a\leq 3$ and $z \geq 8$}. 
\end{equation}
We first prove \cref{ePWz}. To this end, we consider three cases.

(1). If $s > 0$, then by the definition of $C_{0}$ in \cref{deltan}, by \cref{ePhi} and noting that  $ \lvert s  \rvert \leq |u| \leq 1 $, we have  
	\begin{equ}
		\label{e5.4-01}
		\lvert \IP [ W + s > z ] - \IP [ Z + s > z ] \rvert
		& \leq C_0 \delta(z - s) ( 1 - \Phi(z - s) ) \\
		& \leq 2 C_0e^{z|u|}  \delta(z)  (1 - \Phi(z)) . 
	\end{equ}

(2). 	If $s < 0$ and $z - s \leq {z_0}$, by \cref{ez3} and noting that $|s| \leq 1$,
	\begin{equ}
		\label{e5.4-02}
		\lvert \IP [ W + s > z ] - \IP [ Z + s > z ] \rvert
		& \leq C_0 \delta(z - s) ( 1 - \Phi(z - s) )  \\
		  & \leq C_0 e^{\tau/2} \delta(z)  (1 - \Phi(z)) . 
	\end{equ}

(3). If $s < 0$ but $z - s > {z_0}$, it then follows that ${z_0} \geq z \geq 8$ and $\lvert z - {z_0} \rvert \leq |s| \leq |u| \leq 1$.
By \cref{deltan,ePhi,ez3}, 
	\begin{align*}
		\begin{split}
			|\IP [ W + s > z ] - \IP [ Z + s  >z ] |
            & \leq \IP [W > {z_0}] + \IP [ Z > z_{0} ] \\
            & \leq ( 1 - \Phi({z_0})) + C_0 \delta({z_0}) ( 1 - \Phi({z_0})) + 1 -
            \Phi(z_{0})\\
            & \leq 2 e^{z|u|} (1 - \Phi(z)) (1 + e^{\tau/2} C_0 \delta({z}) ). 
		\end{split}
	\end{align*}
	By \cref{eq-z0}, we have 
	\begin{align*}
        z_0 \geq 0.02 e^{-\tau/2} \min \Bigl\{ 50 e^{\tau/2} m_0, \frac{1}{3} r_0^{-1/(\tau_0 + 1)}, \frac{1}{3} r_1^{-1/(\tau_1 + 2)}, \frac{1}{3} r_2^{-1/(\tau_2 + 3)} \Bigr\}. 
	\end{align*}
	Hence, by \cref{eq-deltat} and recalling that $c_{0}=1/\delta(m_{0}) + (150 e^{\tau/2})^{\tau}$, we have 
	\begin{align*}
		\frac{1}{\delta({z_0})} \leq c_0.
	\end{align*}
	By \cref{ez3} and the fact that $z \leq {z_0} \leq z + 1$, we now have  
	\begin{align*}
		1 = \frac{\delta(z)}{\delta(z)} \leq e^{\tau/2}\delta(z) \frac{1}{\delta({z_0})} \leq c_0 e^{\tau/2} \delta(z)	.    
	\end{align*}
	Therefore, it follows that 
	\begin{equation}
		|\IP [ W + s > z ] - \IP [ Z + s  >z ] |
		\leq 2 (C_0 + c_0) e^{\tau/2} e^{z|u|} \delta(z) (1 - \Phi(z)). 
		\label{e5.4-03}
	\end{equation}
	Combining \cref{e5.4-01,e5.4-02,e5.4-03} yields \cref{ePWz}. The inequality \cref{ePWy} can be shown similarly. 
\end{proof}
{We next give the proof of \cref{expbound}, which couples Stein's method with the recursive method. The proof includes two parts. First, based on Stein's method, we split the numerator on the left hand side of
\cref{eHbound} into several terms. Second, with the help of \cref{l5.1,ll2}, we bound these parts by recursive
arguments. }
\begin{proof}[Proof of \cref{expbound}]
We first introduce some notation and inequalities. We fix $8 \leq z \leq z_0$ in this proof. 
Because $z_{0}\geq 8$, we have that \cref{eq:ass_1} holds.	
	We also choose $\beta \coloneqq {z_0}$ in the function $\Psi_{\beta, t}(w)$, and 
	let $\epsilon = 40 e^{\tau/2} r_2(1 + z^{\tau_2})$. By \cref{eq-52,eq-eps-bound}, we have
    $e^{\beta \epsilon} \leq 1.05$ and $\epsilon \leq 0.1$.   

	Now, 
	consider the Stein equation 
	\begin{equ}
		f'(w) - wf(w) = h_{z, \epsilon}(w) - \IE \{ h_{z, \epsilon}(Z) \}, 
        \label{eSte}
	\end{equ}
	and let $f:= f_{z, \epsilon}$ be its solution. Let $g(w) = w f(w)$ and  let  
	\begin{math}
		\upsilon (w) = (2\pi)^{-1/2} \int_{0}^{ w } s e^{ -(z + \epsilon - \epsilon s)^2/2 } ds. 
	\end{math}
	Recall that $Z \sim N(0,1)$,  $Nh_{z,\epsilon} = \E h_{z,\epsilon}(Z)$, $\phi(\cdot)$ is the standard normal probability density function and $\Phi(\cdot)$ is the standard normal distribution function.
	It can be shown that (see, e.g., Lemma 5.3 of \cite{chenshao2004})
	\begin{align}
		Nh_{z, \epsilon} & =  \Phi(z) + \epsilon \upsilon(1) = \Phi(z) + \int_z^{z + \epsilon} \Bigl( 1
        + \frac{ z - s }{\epsilon} \Bigr) \phi(s) ds ,  
        \label{eNh}
		\\
		f(w) & =  
            \label{eSol}
		\begin{dcases}
			\frac{\Phi(w)}{\phi(w)} \bigl( 1 - Nh_{z, \epsilon}\bigr) & \text{if $w \leq z$},\\
			\frac{ 1 - \Phi(w) }{\phi(w)} Nh_{z, \epsilon} -  \frac{\epsilon}{\phi(w)} \upsilon \Bigl( 1 + \frac{z - w}{\epsilon} \Bigr)  & \text{if $z < w \leq z + \epsilon$}, \\
			\frac{1 - \Phi(w)}{\phi(w)} Nh_{z, \epsilon} & \text{if $w > z + \epsilon$}, 
		\end{dcases}
	\end{align}
	and 
	\begin{align}
		g'(w) = 
		\begin{dcases}
			\biggl( \frac{ (1 + w^2) \Phi(w) }{\phi(w)} + w \biggr) \bigl( 1 - N h_{z, \epsilon} \bigr) & \text{ if $w \leq z$, }\\
			\biggl( \frac{ (1 + w^2) (1 - \Phi(w)) }{\phi(w)} - w \biggr) Nh_{z, \epsilon}\\
			\qquad - \frac{\epsilon (1 + w^2)}{\phi(w)} \upsilon \biggl( 1 + \frac{z - w}{\epsilon} \biggr) +  \frac{w(z - w + \epsilon)}{\epsilon} 
			& \text{ if $z < w \leq z + \epsilon$, }\\
			\biggl( \frac{ (1 + w^2) (1 - \Phi(w)) }{\phi(w)} - w \biggr) Nh_{z, \epsilon} & \text{ if $w > z + \epsilon$. }
		\end{dcases}
		\label{egprime}
	\end{align}
	Thus, by \cref{e1,eSte}, 
	\begin{equ}
		\MoveEqLeft \bigl\lvert \IE \{ h_{z,\epsilon}(W) \} - Nh_{z,\epsilon} \bigr\rvert\\
		& = \bigl\lvert \IE \{ f'(W) - W f(W) \} \bigr\rvert\\
		& = \biggl\lvert \IE \{ f'(W)\} - \IE \biggl\{ \int_{-\infty}^{\infty} f'(W + u) \hat{K}(u) du \biggr\} - \IE \{ R f(W) \} \biggr\rvert\\
		& \leq  |I_1| + |I_2| + |I_3|, 
        \label{eII123}
	\end{equ}
	where 
	\begin{align*}
		I_1 & =   \IE \biggl\{ \int_{-\infty}^{\infty} f'(W + u)-f'(W) \hat{K}(u) du \biggr\} , & 
		I_2 & =   \IE \{ f'(W) ( 1 - \hat{K}_1 ) \} , & 
		I_3 & =   \IE \{ Rf(W) \} . 
	\end{align*}
	For $I_1$, by \cref{eSte}, we have 
	\begin{equ}
		I_1 = I_{11} + I_{12} + I_{13} + I_{14}, 
        \label{eI1}
	\end{equ}
	where 
	\begin{align*}
		I_{11} & = \IE \biggl\{ \int_{-\infty}^{\infty} \bigl( g(W + u) - g(W) \bigr) \hat{K}(u) du \biggr\}, \\
		I_{12} & = \IE \biggl\{ \int_{ \lvert u \rvert > 1 }^{} \bigl( h_{z,\epsilon}(W + u) - h_{z, \epsilon}(W) \bigr) \hat{K}(u) du \biggr\}, \\
		I_{13} & = \IE \biggl\{ \int_{ \lvert u \rvert \leq 1 } \bigl( h_{z, \epsilon} (W + u) - h_{z, \epsilon}(W) \bigr) K(u) du \biggr\}, \\
		I_{14} & = \IE \biggl\{ \int_{ \lvert u \rvert \leq 1 } \bigl( h_{z, \epsilon}(W + u) - h_{z, \epsilon}(W) \bigr) ( \hat{K}(u) - K(u) ) du \biggr\}. 
	\end{align*}

	In what follows, we prove the following inequalities: 
	\begin{align}
        \lvert	 I_{11}\rvert  & \leq 41 r_2 (1 + z^{\tau_2 + 3}) (1 - \Phi(z)), \label{l5.2-I11}\\
\lvert	I_{12} \rvert & \leq  r_2 (1 + z^{\tau_2 + 3}) (1 - \Phi(z)), \label{l5.2-I12}\\
\lvert	I_{13} \rvert & 
		\leq 0.31 (C_0 + c_0) \delta(z) (1 - \Phi(z))  + (1 + 2\rho e^{\tau/2}) r_2 (1 + z^{\tau_2 + 3}) (1 - \Phi(z)), 
	\label{l5.2-I13}
	\\
\lvert	I_{14} \rvert & \leq  (0.44C_0 + 0.44c_0 + 100 e^{\tau/2} ) \delta(z) (1 - \Phi(z)) , \label{l5.2-I14}\\
\lvert	I_2 \rvert & \leq   66 r_1 (1 + z^{\tau_1 + 2}) (1 - \Phi(z)),\label{l5.2-I2}\\
\lvert	I_3 \rvert & \leq 82 r_0 (1 + z^{\tau_0 + 1})(1 - \Phi(z)). \label{l5.2-I3}
	\end{align}
	Combining \crefrange{l5.2-I11}{l5.2-I3}, we complete the proof of \cref{expbound}. It now suffices to prove \crefrange{l5.2-I11}{l5.2-I3}. We remark that we use a
	recursive method in  the proofs of \cref{l5.2-I13,l5.2-I14}, and  the proofs for $I_{11},I_{12},I_2$ and $I_3$ are routine. 

	{\medskip\noindent\it (i) Proof of \cref{l5.2-I11}.}
	For $I_{11}$, we have 
	\begin{align*}
		\lvert I_{11} \rvert
		& \leq \biggl\lvert \IE \biggl\{ \int_{-\infty}^{\infty} \int_0^u g'(W + s) \hat{K}(u) ds du \biggr\} \biggr\rvert \leq I_{111} + I_{112} + I_{113} , 
	\end{align*}
	where 
	 \begin{align*}
	 I_{111} & = \biggl\vert \IE \biggl\{ \int_{-\infty}^{\infty} \int_0^u g'(W + s) \1( W + s \leq
	 0 ) \hat{K}(u) ds du \biggr\} \biggr\vert, \\
	 I_{112} & = \biggl\vert \IE \biggl\{ \int_{-\infty}^{\infty} \int_0^u g'(W + s) \1( 0 < W + s \leq z ) \hat{K}(u) ds du \biggr\} \biggr\rvert, \\
	 I_{113} & = \biggl\vert \IE \biggl\{ \int_{-\infty}^{\infty} \int_0^u g'(W + s) \1(W + s > z) \hat{K}(u) ds du \biggr\} \biggr\rvert. 
	 \end{align*}
	We now bound these terms separately. 

	\medskip 
	{\it\noindent (1) Bound of $I_{111}$. }
	Observe that
	\begin{align*}
		0 \leq \frac{(1 + w^2) \Phi(w)}{\phi(w)} + w \leq 2 \quad \text{ for $w \leq 0$, } 
	\end{align*}
	and thus, by \cref{egprime}, 
	\begin{math}
		| g'(w) | \1 [ w \leq 0 ] \leq 2 ( 1 - Nh_{z, \epsilon} ).  
	\end{math}
	By \cref{el3c} with $t = 0$, noting that $\Psi_{\beta,0}(w) \equiv 2$, we have 
    \begin{align}
        \int_{ -\infty }^{\infty} \IE \{ |u \hat{K}(u)|\} du \leq 2 r_2. 
        \label{eq-2r3}
    \end{align}
	Moreover, 
	\begin{equation}
		\label{eq-Nhbound}
		1 - Nh_{z, \epsilon} \leq 1 - \Phi(z). 
	\end{equation}
	Thus,
	\begin{equation}
		 I_{111}  
		 \leq 2 ( 1 - Nh_{z, \epsilon} ) \IE \biggl\{  \int_{-\infty}^{\infty} |u \hat{K}(u)| du \biggr\} \leq 2 r_2 (1 - \Phi(z)) . 
		\label{l5.2-I111}
	\end{equation}

	\medskip 
	{\it\noindent (2) Bound of $I_{112}$. }
	Observe that 
	\begin{equation}
		0 \leq \frac{(1 + w^2)}{\phi(w)} + w \leq 3 (1 + w^2) e^{w^2/2} \quad  \text{for $0 \leq w \leq z$}.
		\label{l5.2-z1}
	\end{equation}
	For any $0 \leq a \leq b \leq z$ and for any $u \wedge 0 \leq s \leq u \vee 0$, we have 
	\begin{equ}
		\label{l5.2-z2}
		\MoveEqLeft \IE \{ (1 + (W + s)^2) e^{ (W + s)^2 / 2 } \lvert \hat{K}(u)\rvert  \1 ( a \leq W + s \leq b ) \} \\
		& \leq (1 + b^2) \IE \{ \lvert \hat{K}(u) \rvert e^{(W + s)^2/2 - b (W + s) + b (W + s)} \1 (a \leq W + s \leq b) \}\\
		& \leq (1 + b^2) e^{a^2/2 - ab} \IE \{ \lvert \hat{K}(u) \rvert e^{b (W + s)} \1 (a \leq W + s \leq b) \}\\
		& \leq (1 + b^2) e^{(b - a)^2/2} e^{-b^2/2} \IE \{ \lvert \hat{K}(u) \rvert e^{b |u|}
        \Psi_{z_{0}, b}(W) \}.
	\end{equ}
	Denote by $\lfloor z \rfloor$ the greatest integer which is smaller than or equal  to
	$z$. Noting that for  $u\wedge 0 \leq s\leq u\vee 0$,
	by \cref{egprime,eq-Nhbound,l5.2-z1} and applying \cref{l5.2-z2} with $a = j - 1, b = j$ and $a = \lfloor z \rfloor, b = z$, respectively, we have 
	\begin{align*}
		\MoveEqLeft | \IE \{ g'(W + s) \hat{K}(u) \1 (0 \leq W + s \leq z) \} |\\
		& \leq 3 ( 1 - \Phi(z) )\IE \{ ( 1 + (W + s)^2 ) e^{ (W + s)^2/2 } \vert
		\hat{K}(u)\vert  \1 ( 0 \leq W + s \leq z ) \} 
		\\
		& \leq 3 ( 1 - \Phi(z) )\sum_{j = 1}^{ \lfloor z \rfloor } \IE \{ ( 1 + (W + s)^2 ) e^{ (W + s)^2/2 } |\hat{K}(u)| \1 ( j - 1 \leq W + s \leq j ) \} \\
		& \quad + 3 ( 1 - \Phi(z) )\IE \{ ( 1 + (W + s)^2 ) e^{ (W + s)^2/2 } \vert\hat{K}(u)\vert  \1 ( \lfloor z \rfloor \leq W + s \leq z ) \}\\
		& \leq 3 e^{1/2} ( 1 - \Phi(z) )\sum_{j = 1}^{ \lfloor z \rfloor }  ( 1 + j^2 )
        e^{-j^2/2}\IE \{ |\hat{K}(u)| e^{ j|u| } \Psi_{z_{0}, j}(W)  \} \\
		&  \quad + 3 e^{1/2}(1 - \Phi(z)) ( 1 + z^2 ) e^{-z^2/2}\IE \{ |\hat{K}(u)| e^{ z |u|}
        \Psi_{z_{0}, z}(W) \} . 
	\end{align*}
	Thus, by the definition of $I_{112}$, and by \cref{el3c,l5.1}, we have for $0 \leq z \leq z_0$, 
	\begin{align*}
		I_{112}
		& \leq 5 ( 1 - \Phi(z)  ) \sum_{j = 1}^{ \lfloor z \rfloor } (1 + j^2) e^{-j^2/2} \IE
        \bigl\{ \hat{K}_{2,j} \Psi_{z_{0}, j}(W)  \bigr\}\\
		& \quad  +  5 ( 1 - \Phi(z)  )  (1 + z^2) e^{ - z^2/2 }\IE \bigl\{
        \hat{K}_{2,z} \Psi_{z_{0}, z}(W)  \bigr\} \\
		& \leq 20 r_2 (1 - \Phi(z)) \biggl(\sum_{j = 1}^{ \lfloor z \rfloor } (1 + j^2) (1 + j^{\tau_2})  + (1 + z^2) (1 + z^{\tau_2})\biggr). 
	\end{align*}
	For all $\ell \geq 0$ and $z \geq 8$, it can be shown that 
    \begin{equation}
        \begin{aligned}
           \MoveEqLeft \sum_{j = 1}^{ [z] } (1 + j^2)(1 + j^{\ell})\\&\leq  \sum_{j = 1}^{ [z] }
            1+j^{2}+j^{\ell}+j^{2+\ell}\leq z+\frac{z^{3}}{3}+ \frac{z^{2}}{2}+
           \frac{z}{6}+\sum_{j = 1}^{ [z] }(j^{\ell}+j^{2+\ell})\\&\leq 0.51 z^{\ell+3}+ \sum_{j = 1}^{ [z] }(j^{\ell}+j^{2+\ell})
        \end{aligned}
        \label{eq:5_7}
    \end{equation}
    and for any $m\geq 0$ and $n\geq 1$ we have 
    \begin{equation}
        \begin{aligned}
            \sum_{j = 1}^{ n }j^{m} \leq n^{m}+ \sum^{n-1}_{j=1} \int_{j}^{j+1} x^{m}dx=
            n^{m}+
            \frac{1}{m+1} n^{m+1}\leq \Bigl(\frac{1}{m+1}+ \frac{1}{n}\Bigr)  n^{m+1}.
        \end{aligned}
        \label{eq:5_8}
    \end{equation}
    Recalling that $z\geq 8$ and $\ell\geq 0$, by \cref{eq:5_8} with $n=[z]$ and $m=\ell$ or $\ell+2$,  we have 
    \begin{equation}
        \begin{aligned}
            \sum_{j = 1}^{ [z] } (1 + j^2)(1 + j^{\ell})&\leq  0.51 z^{\ell+3}+ \Bigl(\frac{1}{\ell+1}+\frac{1}{8}\Bigr)
            z^{\ell+1}+\Bigl(\frac{1}{\ell+1}+\frac{1}{8}\Bigr) z^{\ell+3}\\
                                                        &\leq
                                                        \Bigl(0.51+
                                                        \frac{1}{8^{2}(\ell+1)}+\frac{1}{8^{3}}+\frac{1}{\ell+3}+\frac{1}{8}\Bigr)z^{\ell+3}\leq
                                                        z^{\ell+3}.
        \end{aligned}
        \label{eq:5_9}
    \end{equation}
  Then,  for all $\ell \geq 0$ and $z \geq 8$,
    \begin{equation}
        \begin{aligned}
    &	\sum_{j = 1}^{ [z] } (1 + j^2)(1 + j^{\ell})\leq (1+z^{\ell + 3}),\\
        \quad
    &	(1+z^{2})(1 + z^{\ell})\leq 2 \Bigl(1+\frac{1}{64}\Bigr) z^{\ell+2}\leq  \frac{2.032}{8} z^{\ell+3}\leq 0.26 (1+z^{ \ell + 3}) . 
        \end{aligned}
		\label{epbound}
    \end{equation}
    Thus, 
    \begin{align}
		\label{l5.2-I112}
        I_{112} & \leq 26 r_2 (1 + z^{\tau_2 + 3}) ( 1 - \Phi(z) ). 
    \end{align}
	
	\medskip 
	{\it\noindent (3) Bound of $I_{113}$. }
    According to Eqs. (4.5) and (4.6) in \cite{chenshao2004} , we have  
    $|f(w)|\leq 1$ and $|f'(w)|\leq 1$ for $w\in \IR$. Thus, recalling that $g(w)= w f(w)$ and by the fact that $\varepsilon\leq 1$, we have 
	\begin{equation}
        \lvert g'(w) \rvert\leq \lvert f(w)+ w f'(w)\rvert\leq 1+z+\varepsilon \leq 4 ( z + 1 )
        \quad \text{if $z+\varepsilon\geq w \geq z$}.   
        \label{egwa_1}
	\end{equation}
    By \cref{egprime,egwa_1} and the fact that 
    \begin{equation}
        \begin{aligned}
            \Bigl\lvert \frac{ (1 + w^2) (1 - \Phi(w)) }{\phi(w)} - w \Bigr\rvert\leq 1 \text{ for }
            w\geq 8,
        \end{aligned}
        \label{eq:5_10}
    \end{equation}
  we have 
  \begin{equation}
        \lvert g'(w) \rvert\leq   4 ( z + 1 )
        \quad \text{if $w \geq z$}.   
        \label{egwa}
	\end{equation}
	For any $\ell \geq 0$ and $z \geq 8$, we have 
	\begin{equation}
        (1 + z)^2 (1 + z^{\ell})\leq 2\times 1.125^2 z^{\ell+2} \leq 0.32 (1 + z^{\ell + 3}). 
		\label{epbound3}
	\end{equation}
	By \cref{el3c,eq:rmexp,egwa,eNorm,epbound3} and the Markov's inequality, 
	\begin{equ}
		I_{113} 
		& \leq 4 (1 + z) \IE \biggl\{ \int_{-\infty}^{\infty} \int_{ 0 \wedge u }^{ 0 \vee u } \1 (W + s > z) \lvert \hat{K}(u) \rvert  du \biggr\} \\
        & \leq 4 (1 + z)\Psi_{z_{0}, z}(z)^{-1}\IE \biggl\{ \int_{-\infty}^{\infty} \lvert u
        \hat{K}(u) \rvert \Psi_{z_{0}, z}(W + |u|) du \biggr\}\\
		& \leq 4 (1 + z)e^{-z^2}\IE \biggl\{ \int_{-\infty}^{\infty} e^{ z |u| } \lvert u \hat{K}(u)
        \rvert \Psi_{z_{0}, z}(W) du \biggr\}\\
		& \leq 16  (2 \pi)^{1/2}  r_2 (1 + z)(1 + z^{\tau_2}) \phi(z) \\
		& \leq 40.2  r_2 (1 + z)^2(1 + z^{\tau_2})( 1 - \Phi(z)) \\
		& \leq 13 r_2(1 + z^{\tau_2 + 3})(1 - \Phi(z)). 
		\label{l5.2-I113}
	\end{equ}
	Therefore, \cref{l5.2-I11} follows from \cref{l5.2-I111,l5.2-I112,l5.2-I113}.

	{\medskip\noindent\it (ii) Proof of \cref{l5.2-I12}.}
	By the Markov inequality, 
	\begin{equation*}
        \begin{split}
            		\lvert I_{12} \rvert 
                    & \leq \IE \biggl\{ \int_{ \lvert u \rvert > 1 }^{} \1 (W + u > z ) \lvert
            			\hat{K}(u) \rvert du \biggr\} \\
            		& \quad + \IE \biggl\{ \int_{ \lvert u \rvert > 1 }^{} \1
            			({W > z }) 
            \lvert \hat{K}(u) \rvert du \biggr\} \\
            		& \leq \IE \biggl\{ \int_{-\infty}^{\infty}  \vert u\vert e^{-z^2}
                    \Psi_{z_{0}, z}(W + |u| ) |\hat{K}(u)| du \biggr\} \\
            		& \quad + \IE \biggl\{ \int_{-\infty}^{\infty} \vert u\vert e^{-z^2}
                    \Psi_{z_{0}, z}(W) |\hat{K}(u)| du \biggr\} \\
                    & \leq 2 \IE \biggl\{ \int_{-\infty}^{\infty} e^{-z^{2}} \Psi_{z_{0}, z}(W) |u| e^{z|u| } \lvert \hat{K}(u) \rvert du \biggr\} \\
                    & \leq 2 e^{-z^{2} } \IE \{ \hat{K}_{2,z} \Psi_{z_{0}, z}(W ) \}\\
					& \leq 8 r_2 (1 + z^{\tau_2}) e^{-z^2/2}, 
        \end{split}
    \end{equation*}
    where we used \cref{el3c,eq:rmexp} in the last line. 
	By \cref{ezz21,eNorm,eq-52},  
	we have 
	\begin{equation*}
		\begin{split}
		        \lvert I_{12} \rvert 
		& \leq  8 (2\pi)^{1/2} r_2 (1 + z^{\tau_2})\phi(z) \\
		& \leq 21 r_2 (1 + z)(1 + z^{\tau_2}) ( 1 - \Phi(z) ) \\
		& \leq  r_2 (1 + z^{\tau_2 + 3}	) (1 - \Phi(z)),  
		\end{split}
	\end{equation*}
	which proves \cref{l5.2-I12}.

	{\medskip\noindent\it (iii) Proof of \cref{l5.2-I13}.}
	Observe that (see also (2.5) of \cite{chen_adrian_xia2020})
	\begin{equ}
		\lvert h_{z,\epsilon}(w + u) - h_{z,\epsilon}(w) \rvert
		& \leq \frac{1}{\epsilon} \int_{u \wedge 0}^{u \vee 0} \1 [ z < w + s \leq z + \epsilon ] ds \\
		& \leq \1 ( z - u \vee 0 < w \leq z - u \wedge 0 + \epsilon ). 
        \label{hineq}
	\end{equ}
	Recall that $8 \leq z \leq {z_0}$, and by \cref{hineq} and Fubini's theorem,
	\begin{equ}
		\label{eI13a}
		| I_{13} | 
		& \leq  \int_{ \lvert u \rvert \leq 1 }  \IE \{ | h_{z,\epsilon}(W + u) -
		h_{z,\epsilon}(W) \vert\}  |K(u)| du \\
		& \leq \frac{1}{\epsilon} \int_{ \lvert u \rvert \leq 1 } \int_{ u \wedge 0 }^{ u \vee 0 } \IP [ z < W + s \leq z + \epsilon ] |K(u)| ds du \\
		& \leq  I_{131} + I_{132} + I_{133},
	\end{equ}
	where 
	\begin{equ}
        I_{131} & \coloneqq  \int_{ \lvert u \rvert \leq 1 }  \bigl(\Phi( z - 0 \wedge u + \epsilon ) -
        \Phi(z - 0 \vee u) \bigr)|K(u)| du ,\\
		I_{132} & \coloneqq \frac{1}{\epsilon} \biggl\lvert \int_{ \lvert u \rvert \leq 1 } \int_{ u \wedge 0 }^{ u \vee 0 } ( \IP [ W + s > z ] - \IP [ Z + s > z ] ) |K(u)| ds du\biggr\rvert,\\
		I_{133} & \coloneqq \frac{1}{\epsilon} \biggl\lvert \int_{ \lvert u \rvert \leq 1 } \int_{ u \wedge 0 }^{ u \vee 0 } ( \IP [ W + s > z + \epsilon ] - \IP [ Z + s > z + \epsilon ] ) |K(u)| ds du \biggr\rvert. 
        \label{eI13}
	\end{equ}
    One can easily verify that  $(1 + z^{\tau_2})(1 + z^{3})/(1+z^{\tau_{2}+3})$ is a decreasing
    function for $z\geq 1$ and $\tau_{2}\geq 0$, and 
    \begin{equation}
        \begin{aligned}
            \sup_{z\geq 2 }  \frac{ (1 + z^{\tau_2})(1 + z^{3})}{1+z^{\tau_{2}+3}}\leq 2.
        \end{aligned}
        \label{eq:5_11}
    \end{equation}
	Thus, for $z \geq 8$, 
	\begin{align}
		\label{ezboun5}
        (1 + z)  \leq 1.125 z\leq  0.02 (1 + z^{3}) ,\\
        (1 + z^{\tau_2})(1 + z^{3})  \leq 2 (1 + z^{\tau_2 + 3}) .
		\label{ezbound6}
	\end{align}
	Moreover,  if $\beta=0$, then $1 \leq \Psi_{\beta, t}(w) \leq 3$ for all $w$ and $t$, and by \cref{el3c} with $\beta = 0$, we have 
	\begin{equation}
		\label{eq-2r5}
        \int_{|u| \leq 1} e^{z |u|} \lvert u K(u) \rvert du \leq \IE \{\hat{K}_{2,z}\} \leq 3 r_2(1 + z^{\tau_2}).
	\end{equation}
	For $I_{131}$, 
    by \cref{eNorm,eq-51,ezboun5}, for $z\geq 8$ and $ \lvert u\rvert\leq 1$,
	\begin{equation}
        \begin{split}
            \Phi( z - 0 \wedge u + \epsilon ) - \Phi(z - 0 \vee u) 
            & \leq ( |u| + \epsilon ) \phi(z - 0 \vee u) \\
            & \leq ( |u| + \epsilon ) e^{ z |u| } \phi(z) \\
            & \leq ( |u| + \epsilon ) e^{ z |u| } ( 1 + z ) ( 1 - \Phi(z) )\\
            & \leq ( 0.02 |u| + 0.02 \epsilon ) e^{z|u|} (1 + z^{3}) (1 - \Phi(z)). 
        \end{split}
        \label{ePhibound}
	\end{equation}
	Then, recalling that $\epsilon = 40 e^{\tau/2} r_2(1 + z^{\tau_2})$, by
	\cref{eq-2r5,ezbound6,eq:r_3},   
	\begin{equation}
        \begin{split}
            I_{131} & \leq 0.02 (1 + z^{ 3}) (1 - \Phi(z))\int_{ \lvert u\rvert \leq 1 }^{} ( |u| + \epsilon ) e^{z |u|}  \lvert K(u) \rvert du \\
					& \leq (0.06 + 0.8\rho e^{\tau/2}) r_2 (1 + z^{ 3})(1 + z^{\tau_2}) ( 1 - \Phi(z) )\\
					& \leq (0.12 + 1.6\rho e^{\tau/2}	  ) r_2 (1 + z^{\tau_2 + 3}) (1 - \Phi(z)).
        \end{split}
        \label{eI131}
	\end{equation}

	As for $I_{132}$,
	by \cref{ll2,eq-2r5} and recalling that $\epsilon = 40 e^{\tau/2} r_2(1 + z^{\tau_2})$, we have  for $8 \leq z \leq {z_0}$ and  $|u| \leq 1$,
	\begin{equ}
		 I_{132}  
		& \leq 2 e^{\tau/2} \epsilon^{-1} (C_0 + c_0) \delta(z) ( 1 - \Phi(z) )\int_{ \lvert u \rvert \leq 1 } e^{z |u|} |uK(u)| du \\
		& \leq 6 e^{\tau/2} r_2 \epsilon^{-1} (C_0 + c_0) \delta(z) ( 1 - \Phi(z) )(1 + z^{\tau_2}) \\
		& {\leq 0.15 (C_0 + c_0) \delta(z) ( 1 - \Phi(z) )}. 
        \label{eI132}
	\end{equ}

	As for $I_{133}$, as $8 \leq z \leq {z_0}$ and $0 \leq \epsilon \leq 1$, by \cref{ll2} and \cref{eq-52} again, we have  
	\begin{equation}
		\begin{split}
			I_{133} & \leq 0.15 e^{z_0 \epsilon} ( C_0 + c_0	) \delta(z) ( 1 - \Phi(z) ) \\
					& \leq 0.16 ( C_0 + c_0) \delta(z) ( 1 - \Phi(z) ) .
        \label{eI133}
		\end{split}
	\end{equation}
	By \cref{eI13,eI131,eI132,eI133}, we have 
	\begin{equation*}
        \begin{split}
            		\lvert I_{13} \rvert 
					& \leq 0.31 (C_0 + c_0)\delta(z) (1 - \Phi(z)) + (0.12 + 1.6 \rho e^{\tau/2}) r_2 (1 + z^{\tau_1 + 3}) (1 - \Phi(z)).
        \end{split}
	\end{equation*}
	This proves \cref{l5.2-I13}.

	{\medskip\noindent\it (iv) Proof of \cref{l5.2-I14}.}	
	Without loss of generality, we assume that $r_4 > 0$. Without this assumption, the proof would
	be even easier. 
    Note that by \cref{hineq},
	\begin{align*}
		\lvert I_{14} \rvert 
		& \leq \IE \biggl\{ \int_{ \lvert u  \rvert \leq 1 }^{} \1 [ z - 0 \vee u < W \leq z - 0 \wedge u + \epsilon ] \lvert \hat{K}(u) - K(u) \rvert du \biggr\} . 
	\end{align*}
	Recall  Young's inequality 
	\begin{align*}
		a b \leq \frac{a^2}{2c} + \frac{cb^2}{2} \quad \text{for $a, b \geq 0$ and  $c > 0$.}
	\end{align*}
    Applying Young's inequality with  
	$a = \1 [ z - 0 \vee u < W \leq z - 0 \wedge u + \epsilon ] , b = \lvert \hat{K}(u) - K(u) \rvert \1 (W > z - 0 \vee u) $ and 
	$$c = \frac{ e^{z|u|} }{99 e^{\tau/2}} \bigl( e^{\tau/2} (4.1 C_0 + 4.1 c_0 +1.6) + r_4^{-1/2}|u| \bigr),  $$ 
	we have 
	\begin{equ}
		\lvert I_{14} \rvert
		& \leq \frac{1}{2}  \int_{ \lvert u \rvert \leq 1 }^{} c^{-1}{ \IP [ z - 0 \vee u < W \leq z - 0 \wedge u + \epsilon ] } d u  \\
		& \quad + \frac{1}{2}\IE \biggl\{ \int_{ \lvert u \rvert \leq 1 }^{} c \bigl( \hat{K}(u) - K(u) \bigr)^2 \1 [ W > z - 0 \vee u ] du \biggr\}\\
		& \coloneqq I_{141} + I_{142}.
        \label{eI14a}
	\end{equ}

	Using a similar argument to \cref{ePhibound}, and by \cref{ezbound6} and the fact that
	$\tau_{4}\geq 0$, we have 
	\begin{equ}
		\MoveEqLeft 
		\IP [  z - 0 \vee u < Z \leq z - 0 \wedge u + \epsilon  ]\\
		& \leq |u| e^{z|u|} (1 + z^{\tau_4 + 1})(1 - \Phi(z)) + 0.8 e^{\tau/2} r_2(1 + z^{\tau_2})(1 + z^{3})(1 - \Phi(z))\\
		& \leq r_4^{-1/2}|u|e^{z|u|} \delta(z) (1 - \Phi(z)) + 1.6 e^{\tau/2} \delta(z) (1 - \Phi(z)).
		\label{eq-5.64}
	\end{equ}
    By \cref{eq-5.64,ll2}, and noting that $e^{{z_0} \epsilon} \leq 1.05$ in \cref{eq-52}, 
    \begin{equation}
        \begin{aligned}
		   \MoveEqLeft \IP [ z - 0 \vee u < W \leq z - 0 \wedge u + \epsilon ] 
            \\
       &   \leq \IP [  z - 0 \vee u < Z \leq z - 0 \wedge u + \epsilon  ] \\
       &    \quad + \lvert \IP [ W > z - 0 \vee u ] - \IP [ Z > z - 0 \vee u ] \rvert \\
       &    \quad + \lvert \IP [ W > z - 0 \wedge u + \epsilon ] - \IP [ Z > z - 0 \wedge u + \epsilon] \rvert \\
	   &     \leq \bigl( e^{\tau/2} (4.1 C_0 + 4.1 c_0 + 1.6) + r_4^{-1/2}|u|\bigr)  e^{ z |u| } \delta(z) (1 - \Phi(z)) . 
        \end{aligned}
       \label{eI14b}  
   \end{equation}
   Then, we have 
   \begin{equation}
	   \label{l5.2-I141}	
	   I_{141}  \leq 99 e^{\tau/2} \delta(z) (1 - \Phi(z)).
   \end{equation}

   	Moreover, as $z \geq 8$, by the Markov inequality and by \cref{l5.1}, 
	\begin{equation}
        \begin{split}
            \MoveEqLeft \IE \biggl\{  \int_{ \lvert u\rvert \leq 1 }^{} |u| e^{z|u|}\bigl( \hat{K}(u) - K(u) \bigr)^2 \1 ( W > z - 0 \vee u) du  \biggr\} \\
            & \leq e^{-z^2} \IE \biggl\{   \int_{ \lvert u\rvert \leq 1 }^{} |u| e^{2z|u|}\bigl(
            \hat{K}(u) - K(u) \bigr)^2 \Psi_{z_{0}, z}(W)du   \biggr\} \\
            & \leq r_4  ( 1 + z^{\tau_4} ) e^{-z^2} \IE \{ \Psi_{z_{0}, z}(W) \} \\
            & \leq 4 (2\pi)^{1/2} r_4 (1 + z^{\tau_4})(1 + z) ( 1 - \Phi(z) )\\
            & \leq 21 r_4 (1 + z^{\tau_4 + 1}) ( 1 - \Phi(z) ), 
        \end{split}
        \label{eI14c}
	\end{equation}
    where in the last line we used the inequality that 
    \begin{align*}
        (1 + z)(1 + z^{\tau_4}) \leq 2 (1 + z^{\tau_4 + 1}) \quad \text{for $z \geq 8$}.
    \end{align*}
	Similarly, 
	\begin{equation}
        \begin{split}
            \MoveEqLeft \IE \biggl\{  \int_{ \lvert u\rvert \leq 1 }^{} e^{z|u|}\bigl( \hat{K}(u) - K(u) \bigr)^2 \1 ( W > z - 0 \vee u) du  \biggr\} \\
            & \leq 4 (2\pi)^{1/2} r_3 (1 + z) (1 + z^{\tau_3}) ( 1 - \Phi(z) )\\
            & \leq 21 r_3 (1 + z^{\tau_3 + 1}) ( 1 - \Phi(z) ). 
        \end{split}
        \label{eI14d}
	\end{equation}
	Then, by \cref{eI14c,eI14d},  we have 
	\begin{equ}
		I_{142} & \leq \frac{e^{\tau/2} (4.1 C_0  + 4.1 c_0 + 1.6)}{198 e^{\tau/2}} \times 21 r_3(1 + z^{\tau_3 + 1}) (1 - \Phi(z)) \\
				& \quad + \frac{21r_4^{1/2} }{198 e^{\tau/2} } (1 + z^{\tau_4 + 1}) (1 - \Phi(z))\\
		 & \leq 0.44 (C_0 + c_0 ) \delta(z)(1 - \Phi(z)) + r_3(1 + z^{\tau_3 + 1})(1 - \Phi(z))\\
		 & \quad +  r_4^{1/2} (1 + z^{\tau_4 + 1})(1 - \Phi(z))\\
		 & \leq (0.44 C_0 + 0.44 c_0 + 1) \delta(z) (1 - \Phi(z)).
        \label{bI14}
	\end{equ}
	Combining \cref{l5.2-I141,bI14} yields \cref{l5.2-I14}. 

	{\medskip\noindent\it (v) Proof of \cref{l5.2-I2}.}	
	Note that (see, e.g., p. 2010 of \cite{chenshao2004})
	\begin{align}
		| f'(w) |
		& \leq 
		\begin{dcases}
			1 - \Phi(z) & \text{if $w < 0$},\\
			\biggl( \frac{w \Phi(w)}{\phi(w)} + 1 \biggr) \bigl( 1 - \Phi(z) \bigr) & \text{if $0 \leq w \leq z$, }\\
			1 & \text{otherwise}.
		\end{dcases}
        \label{bfprime}
	\end{align}
	Observe that 
	\begin{align*}
		\lvert I_{2} \rvert
		& \leq \IE \{ \lvert f'(W) ( 1 - \IE\{\hat{K}_1\vert W\} ) \rvert \1 (W \leq 0) \} \\
		& \quad + \IE \{ \lvert f'(W) ( 1 - \IE\{\hat{K}_1\vert W\} ) \rvert \1 (0 < W \leq z) \} \\
		& \quad + \IE \{ \lvert f'(W) ( 1 - \IE\{\hat{K}_1\vert W\}) \rvert \1 (W > z) \} \\
		&\coloneqq I_{21} + I_{22} + I_{23}. 
	\end{align*}
	For $I_{21}$, since $-1 \leq w \Phi(w)/\phi(w) \leq 0$ for $w \leq 0$, and by
	\cref{el3b,bfprime}, we have for $z \geq 8$
	\begin{equ}
        I_{21} & \leq (1 - \Phi(z)) \IE \{ \lvert \IE\{\hat{K}_1\vert W\} - 1 \rvert \} \\
			   & \leq 2 r_1  (1 - \Phi(z)) \\
               & \leq 0.1 r_1 (1 + z^{\tau_1 + 2}) (1 - \Phi(z)) .  
        \label{bI21}
	\end{equ}
	For $I_{22}$, by \cref{el3b,bfprime}, we have 
	\begin{equ}
		I_{22} & \leq (2 \pi)^{1/2} ( 1 - \Phi(z) ) \IE \{ (W e^{W^2/2} + 1) \lvert \IE\{\hat{K}_1\vert W\} - 1 \rvert \1 ( 0 < W \leq z ) \} \\
			   & \leq 2.6 ( 1 - \Phi(z) ) I_{24} + 5.2 r_1  (1 - \Phi(z)) , 
        \label{eI22}
	\end{equ}
	where 
	\begin{align*}
		I_{24} & \coloneqq \IE \{ W e^{W^2/2} \lvert \IE\{\hat{K}_1\vert W\} - 1 \rvert \1 ( 0 < W \leq z ) \} \\
			   & = \sum_{j = 1}^{ [z] } \IE \{ W e^{W^2/2} \lvert \IE\{\hat{K}_1\vert W\} - 1\rvert \1 ( j - 1 < W \leq j ) \} \\
			   & \quad + \IE \{ W e^{W^2/2} \lvert \IE\{\hat{K}_1\vert W\} - 1\rvert \1 ( [z] < W \leq z ) \} \\
			   & \coloneqq I_{241} + I_{242}. 
	\end{align*}
	For $I_{241}$, noting that 
	$$w^2/2 - jw \leq (j - 1)^2/2 - j(j - 1) = - j^2/2 + 1/2 \quad  \text{ for } j - 1 < w \leq j, $$ 
	we have for $z \geq 8$, by \cref{el3b,eq:rmexp}, 
	\begin{equ}
		I_{241} & \leq \sum_{j = 1}^{ [z] } j \IE \{ e^{W^2/2 - j W + jW} \lvert \IE\{\hat{K}_1\vert W\} - 1 \rvert \1( j - 1 \leq W \leq j ) \} \\
				& \leq e^{1/2} \sum^{ [z] }_{j = 1}  j e^{-j^2/2} \IE \{ \lvert \IE\{\hat{K}_1\vert W\} - 1 \rvert
                \Psi_{z_{0}, j}(W) \} \\
                & \leq 4 e^{1/2}  r_1 \sum_{j = 1}^{\lfloor z \rfloor} j (1 + j^{\tau_1})
              \\&{\leq 13.2 r_1 (1 + z^{\tau_1 + 2})}. 
        \label{eI241}
	\end{equ}
	The last inequality is similar to that in
    \cref{epbound}.
    Similarly, for $z \geq 8$, 	\begin{equ}
        I_{242} & \leq 4 e^{1/2} r_1 z(1 + z^{\tau_1}) \leq  1.7 r_1(1 +z^{\tau_1 + 2}).  
        \label{eI242}
	\end{equ}
	By \cref{eI241,eI242}, we have 
	\begin{equ}
		I_{24} & \leq 15 r_1 (1 + z^{\tau_1 + 2}) \quad \text{for $z \geq 8$}.  
        \label{eI24}
	\end{equ}
	By \cref{eI22,eI24}, we have for $z \geq 8$, 
	\begin{equ}
		I_{22} 
		& \leq 39 r_1 (1 + z^{\tau_1 + 2}) (1 - \Phi(z)) + 5.2 r_1 (1 - \Phi(z)) \\
		& \leq 40  r_1 (1 + z^{\tau_1 + 2}) (1 - \Phi(z)).
        \label{bI22}
	\end{equ}
	As for $I_{23}$, by \cref{l5.1,el3b,eNorm,bfprime} and recalling $z \geq 8$, we have 
	\begin{equ}
		I_{23} & \leq \IE \{ \lvert\IE\{\hat{K}_1\vert W\} - 1 \rvert \1 ( W > z ) \} \\
               & \leq e^{-z^2} \IE \{ \lvert \IE\{\hat{K}_1\vert W\} - 1 \rvert \Psi_{z_{0}, z}(W) \} \\
			   & \leq 4  r_1 (1 + z)(1 + z^{\tau_1})e^{-z^2/2} \\
			   &\leq 10.1 r_1 (1 + z)(1 + z^{\tau_1})\phi(z) \\
			   & \leq 10.1 r_1 (1 + z)^2(1 + z^{\tau_1}) ( 1 - \Phi(z) ) \\
			   & \leq 25.2 r_1 (1 + z^{\tau_1 + 2}) (1 - \Phi(z)). 
        \label{bI23}
	\end{equ}
	By \cref{bI21,bI22,bI23}, we have 
	\begin{equ}
		|I_2| \leq 66  r_1 (1 + z^{\tau_1 + 2}) (1 - \Phi(z)).  
        \label{bI2}
	\end{equ}

	{\medskip\noindent\it (vi) Proof of \cref{l5.2-I3}.}	
	 It is known that $0 \leq f(w) \leq 1$ (see p.\ 2010 of \cite{chenshao2004}). Note that by
	 \cref{eq-Nhbound,eSol},
	\begin{align*}
		\lvert I_3 \rvert
		& \leq I_{31} + I_{32} + I_{33}, 
	\end{align*}
	where 
	\begin{align*}
        I_{31} & = (1 - \Phi(z)) \IE \{ \lvert R \rvert \1 (W \leq 0)\}, \\
        I_{32} & = \sqrt{2\pi} (1 - \Phi(z)) \IE \{ e^{W^2/2} \lvert R \rvert \1 (0 < W \leq z	) \}, \\
		I_{33} & = \IE \{ \lvert R \rvert \1 ( W > z ) \}. 
	\end{align*}
    For $I_{31}$, by \cref{el3a} with $t=0$, we have 
	\begin{align*}
		I_{31} & \leq  r_0 (1 + z^{\tau_0 + 1}) (1 - \Phi(z)) \quad    \text{for $z \geq 8$}. 
	\end{align*}
	For $I_{32}$, similar to \cref{bI22}, we have 
	\begin{align*}
		I_{32} & \leq 60 r_0 (1 + z^{\tau_0 + 1}) (1 - \Phi(z)). 
	\end{align*}
	For $I_{33}$, similar to \cref{bI23}, we have 
	\begin{align*}
		I_{33} & \leq 21 r_0 (1 + z^{\tau_0 + 1}) ( 1 - \Phi(z) ). 
	\end{align*}
	Combining the foregoing inequalities, we have 
	\begin{equation*}
		\lvert I_3 \rvert \leq 82 r_0 (1 + z^{\tau_0 + 1}) (1 - \Phi(z)). 
	\end{equation*}
	This proves \cref{l5.2-I3}.
\end{proof}

Now, we prove \cref{l-BE}.
\begin{proof}
	[Proof of \cref{l-BE}]
	In this proof, we develop a Berry--Esseen bound using the idea in \cite{chen_adrian_xia2020}. Let 
	\begin{align*}
		\gamma := \sup_{z \in \IR} \bigl\lvert \IP(W \leq z) - \Phi(z) \bigr\rvert, 
	\end{align*}
	and let $\epsilon = \gamma/2$. Consider the Stein equation \cref{eSte}, and denote by $f_{z, \epsilon}$ the solution to \cref{eSte}, which is given in \cref{eSol}. Let $g_{z, \epsilon}(w) = wf_{z,\epsilon}(w)$. By \cite{chenshao2004}, we have 
	\begin{equation}
		0 \leq f_{z, \epsilon} \leq 1, \quad \vert f_{z, \epsilon}'\vert \leq 1.
		\label{l5.b-01}
	\end{equation}
	Note that 
	\begin{equation}
		\label{l5.b-03}
		\gamma \leq \sup_{z \in \IR} \lvert \IE \{ h_{z, \epsilon}(W) \} - Nh_{z, \epsilon} \rvert + 0.4 \epsilon. 
	\end{equation}
	Now, we bound the first term on the right hand side of \cref{l5.b-03}.
	By \cref{e1} and \cref{eSte}, we have 
	\begin{equ}
		\MoveEqLeft  \IE \{ h_{z, \epsilon}(W) \} - Nh_{z, \epsilon}  \\
		& = \IE \{ f_{z,\epsilon}'(W) \} - \IE \{ W f_{z, \epsilon}(W) \}\\
		& = \IE \{ f_{z,\epsilon}'(W) \} - \IE \{ R f_{z,\epsilon}(W) \} - \IE \biggl\{ \int_{-\infty}^{\infty} f_{z,\epsilon}'(W + u) \hat{K}(u) du \biggr\} \\
		& \coloneqq J_1 + J_2 + J_3 + J_4 + J_5, 
		\label{l5.b-04}
	\end{equ}
	where 
	\begin{align*}
		J_1 & =  \IE \{ f_{z,\epsilon}'(W) (1 - \hat{K}_1) \}, \\
		J_2 & = - \IE \{R f_{z,\epsilon}(W)\}, \\
		J_3 & = - \IE \biggl\{\int_{ \lvert u \rvert > 1 } \bigl( f_{z,\epsilon}'(W + u) - f_{z, \epsilon}'(W) \bigr) \hat{K}(u) du\biggr\}, \\
		J_4 & = - \IE \biggl\{ \int_{ \lvert u \rvert \leq 1 } \bigl( f_{z,\epsilon}'(W + u) - f_{z, \epsilon}'(W) \bigr) {K}(u) du \biggr\}, \\
		J_5 & = - \IE \biggl\{ \int_{ \lvert u \rvert \leq 1 } \bigl( f_{z,\epsilon}'(W + u) - f_{z, \epsilon}'(W) \bigr) ( \hat{K}(u) - {K}(u) ) du\biggr\}.
	\end{align*}
	By \crefrange{el3a}{eq:9_20_12} with $t = 0$ and noting $1 \leq \Psi_{\beta, 0}(w) \leq 2$, we have 
	\begin{equation}
		\label{l5.b-02}
		\begin{split}
			\IE  |R|  & \leq 2 r_0, \qquad  \IE  \lvert \IE \{ \hat{K}_1 \vert W \} - 1 \rvert   \leq 2 r_1, \\
            \IE  \hat{K}_{2,0}  & \leq 2 r_2, \qquad  \IE  \hat{K}_{3,0}   \leq 2 r_3,
            \qquad \IE  \hat{K}_{4,0}  \leq 2 r_4. 
		\end{split}
	\end{equation}
	Now, by \cref{l5.b-01,l5.b-02}, we have 
	\begin{equation}
		\lvert J_1 \rvert \leq 2 r_1, \quad \lvert J_2 \rvert \leq 2 r_0, 
	\quad 
    \lvert J_3 \rvert \leq 2 \IE \{ \hat{K}_{2,0} \} \leq 4 r_2. 
		\label{l5.b-05}	
	\end{equation}
	By Eqs. (2.12), (2.15) and (2.16) of \cite{chen_adrian_xia2020}, we have  with $a = 0.18$, 
	\begin{equation}
		\begin{split}
			\lvert J_4 \rvert  & \leq 4r_2 + \frac{4\gamma + 0.8 \epsilon}{\epsilon} r_2, \\ 
			\lvert J_5 \rvert  & \leq a \gamma + 0.2 a \epsilon + \frac{2r_3}{a} + (2a + 0.4/a) r_4^{1/2} + 5 r_4^{1/2}. 
		\end{split}
		\label{l5.b-06}
	\end{equation}
	Combining \cref{l5.b-03,l5.b-04,l5.b-05,l5.b-06} gives
	\begin{align*}
		\gamma \leq 0.4 \gamma + 2r_0 + 2r_1 + 16.8 r_2 + 11.2 r_3 + 7.6 r_4^{1/2}. 
	\end{align*}
	Solving the recursive inequality yields the desired result. 
\end{proof}

\section{Proof of \autoref{thm:9_27_2}}%
\label{sec:general method}
We apply \cref{thm1} to prove \cref{thm:9_27_2}. To this end, 
we first construct Stein identity. Then, we prove a preliminary lemma which help us to prove
\cref{thm:9_27_2}.  
Denote by $C, C_{1}, C_{2}, \dots$  absolute constants, which may take different values in different
places. For $a,b\in \mathbf{R}$, let $a\wedge b:= \min \{a,b\}$ and $a\vee b:=\max \{a,b\}$.

For each $i \in \mathcal{J}$, let $Y_i = \sum_{j \in A_i} X_i$. Further, define 
\begin{equ}
	\hat{K}_i(u) & = X_i \{ \1 ( - Y_i \leq u < 0 ) - \1 ( 0 \leq u \leq - Y_i ) \}, \quad
	\hat{K}(u) = \sum_{i \in \mathcal{J}} \hat{K}_i(u), 
    \label{eki}
\end{equ}
and define 
\begin{equ}
	K_i(u) = \IE \{ \hat{K}_i(u) \}, \quad K(u) =  \sum_{i \in
	\mathcal{J}} K_i(u). 
    \label{eeki}
\end{equ}
Note that $\IE X_i=0$ and  $X_i$ and $W - Y_i$ are independent, thus $\IE \{ X_i f(W - Y_i) \} = 0$. Therefore, 
\begin{align*}
	\IE \{ W f(W) \}
	& = \sum_{i \in \mathcal{J}} \IE \{ X_i ( f(W) - f(W - Y_i) ) \}\\
	& = \IE \int_{-\infty}^{\infty} f'(W + t) \hat K(t) dt. 
\end{align*}
Hence, it follows that \cref{e1} holds with $R = 0$ and $\hat K(t)$ defined as in \cref{eki}.
 
\subsection{A preliminary lemma}
Let $\hat{K}_{1}, \hat{K}_{2,t}, \hat{K}_{3,t}, \hat{K}_{4,t}$ and $M_t$  be as in
\cref{ek1,ek2,eq:sec2_2,eq:sec2_3,eq:sec2_1} with $\hat{K}(u)$ in \cref{eki} and $h(t)= \IE
\{\Psi_{\beta,t}(W)\}$.
The following lemma provides the upper bounds of the terms in Condition (A1), whose proof is put in \cref{apdxA}.
\begin{lemma}
	\label{l31} 
Under \textup{(LD1)} and \textup{(LD2)}, let $m_0=(a_{n}^{1/3}/4)\wedge (a_{n}/16).$ 
	For $0 \leq t, \beta \leq m_0$,  we have 
\begin{align}
		\IE \{ \lvert \IE\{\hat{K}_1\vert W\} - 1 \rvert \Psi_{\beta, t}(W) \} 
		& \leq C_{1}  (b \kappa n a_{n}^{-3} + b^{1/2} \kappa^{1/2}
		n^{1/2} a_{n}^{-2}) (1 + t) \IE \{
			\Psi_{\beta, t}(W)  \}, \label{ap1el3b}\\
		\IE \{ \hat{K}_{2,t} \Psi_{\beta, t}(W) \}               
		& \leq C_{2} b  n a_{n}^{-3}  \IE \{
			\Psi_{\beta, t}(W)  \}, \label{ap1el3c}\\
            \IE \{ \hat{K}_{3,t}  \Psi_{\beta,t}(W) \} & \leq  C_{3} (\kappa b n a_{n}^{-3}+  \kappa^{2} b^{2}
			n^{2} a_{n}^{-5})( 1 + t^{2} ) \IE \{
			\Psi_{\beta, t}(W)  \} , \label{ap1eq:9_20_11}\\
                \IE \{ \hat{K}_{4,t}  \Psi_{\beta,t}(W) \} & \leq C_{4} (\kappa b n a_{n}^{-4}+  \kappa^{2} b^{2}
				n^{2} a_{n}^{-6})( 1 + t^{2} ) \IE \{
			\Psi_{\beta, t}(W)  \}, 
				\label{ap1eq:9_20_12} 
	   \end{align}
	   and 
	   \begin{align}
			\sup_{0 \leq t \leq m_0}M_t & \leq C_{5} b n a_{n}^{-2}.
	\label{eq:ap1r_3}
	   \end{align}
\end{lemma}
	\subsection{Proof of  \autoref{thm:9_27_2}}
	\proof [Proof of \cref{thm:9_27_2}] We apply \cref{thm1} to prove  \cref{thm:9_27_2}. Recalling
	that $\theta_n = b^{1/2} n^{1/2} a_n^{-1}$, by \cref{l31}, we have that condition (A1) is satisfied with $r_0= \tau_0 = 0$ and 
	\begin{align}
		r_1 & = C_1 \kappa  \theta_n (\theta_n  + 1) a_n^{-1}, & \tau_1 = 1, \label{eq:sec6_17} \\
		r_2 & = C_2 \theta_n^2 a_n^{-1}, & \tau_2 =  0, \label{eq:sec6_18}\\
		r_3 & = C_3 (\kappa \theta_n^2 + 1)^2a_n^{-1}, & \tau_3 = 2, \nonumber\\
		r_4 & = C_4 (\kappa \theta_n^2 + 1)^2a_n^{-2}, & \tau_4 = 2,  \nonumber\\
		\rho & =  C_5 \theta_n^2, \label{eq:sec6_19} 
	\end{align}
	where $C_1, C_2, C_3, C_4$ and $C_5$ are absolute constants.
     Recalling the definition of $\delta(t)$ in \cref{eq-deltat}, and that
     $m_0=(a_{n}^{1/3}/4)\wedge (a_{n}/16)$,  we have
  \begin{equation}
    \begin{aligned}
        \delta(m_{0})&\geq  (r_{0}+r_{1}+r_{2}+r_{3}+r_{4}^{1/2}) (1+m_{0}^{3})\\&\geq C 
        \theta_{n} \kappa (1+\theta_{n}+ \kappa \theta_{n}^{3}) (a_{n}^{-1}+a_{n}^{2}\wedge 1)\geq C
         \theta_{n}\kappa.
    \end{aligned}
    \label{eq:3_11_5}
\end{equation}
Combining \cref{eq:sec6_19,eq:3_11_5} and noting that $\kappa\geq 1$, we can see that the right hand side of \cref{et1} is less than
\begin{equation}
    \begin{aligned}
		C\Bigl( \frac{1}{\kappa\theta_{n}}+ \theta_{n}^{2}+1\Bigr) \delta(z) & \leq 2 C\Bigl( \frac{1}{\theta_{n}}+
		\theta_{n}^{2}\Bigr) (r_1 + r_2 + r_3 + r_4^{1/2}) (1 + z^3) \\
																				   & \leq C' \delta_n (1 + z^3),
    \end{aligned}
    \label{eq:3_11_6}
\end{equation}
where $\delta_n = \kappa^2 a_n^{-1} (1 + \theta_n^6)$, $C'$ is an absolute constant and we use the
fact that $x^{2}+1/x>1$ for $x>0$ in the first inequality.
On the other hand, by \cref{eq:sec6_17,eq:sec6_18}, we have
\begin{equation}
    \begin{aligned}
        r_{1}^{1/3} + r_2^{1/3} \leq C \kappa^{1/3} a_{n}^{-1/3} (1 + \theta_n)^{2/3}.
    \end{aligned}
    \label{eq:3_14_1}
\end{equation}
Applying \cref{thm1} and by \cref{eq:3_11_6,eq:3_14_1}, we obtain the desired result. 
\endproof

\section{Proof of \autoref{thm:11_24_1}}\label{sec:p_th3}
In this section, we use the exchangeable pair to construct Stein identity \cref{e1}.
For any $k \geq 1$ and $k$-fold index $\mathbf{i} \in \IN^k$, we denote by $i_j$  its $j$-th element. 
Let $[n]_k \coloneqq \{ \mathbf{i} = (i_1, \dots, i_k) \in \IN^k : 1 \leq i_1 \neq \dots \neq i_k \leq n\}$ be a class of $k$-fold indices.
Let $\mathbf{I} \coloneqq (I_1,I_2)$ be chosen uniformly from $[n]_2$ and be independent of $\pi$ and $\mathbf{X}$, and let $W' = W - X_{I_1,\pi(I_1)} - X_{I_2,\pi(I_2)} + X_{I_1,\pi(I_2)} + X_{I_2,\pi(I_1)}$. Then, it follows that $(W,W')$ is an
exchangeable pair. Moreover, we have 
\begin{align*}
	\E \{ W - W' \vert \mathbf{X}, \pi \}
	& = \frac{1}{n(n-1)} \sum_{\mathbf{i} \in [n]_2} \E \{ X_{i_1,\pi(i_1)} + X_{i_2,\pi(i_2)} - X_{i_1,\pi(i_2)}  - X_{i_2,\pi(i_1)} \vert \mathbf{X}, \pi \}\\
	& = \frac{2}{n-1} (W - R), 
\end{align*}
where 
\begin{equ}
	R = - \frac{1}{n} \sum_{i = 1}^n \sum_{j = 1}^n X_{i,j}.
    \label{eq-s7-R}
\end{equ}
By exchangeability,  with $\lambda = 2/(n - 1)$ and $\Delta = W - W'$, we have 
\begin{equ}
	0 & =  \E \{ (W - W') (f(W) + f(W')) \} = 2 \E \{ \Delta f(W) \} - \E \{ \Delta ( f(W) - f(W - \Delta) ) \} \\
	  & = 2 \lambda \E \{ (W - R) f(W) \} - \E \{ \Delta ( f(W) - f(W - \Delta) ) \}.
    \label{eq-s7-exch}
\end{equ}
Rearranging \cref{eq-s7-exch} yields
\begin{align*}
	\E \{ W f(W) \}
	& = \E \int_{-\infty}^{\infty} f'(W + u) \hat K(u) du + \E \{ R f(W) \}, 
\end{align*}
where 
\begin{equ}
	\hat K(u) & =  \frac{1}{2\lambda} \E\{\Delta \bigl( \1 (- \Delta \leq u \leq 0) - \1 (0 < u \leq -\Delta)\bigr) \vert \mathbf{X},\pi \}\\
			  & = \frac{1}{4n} \sum_{\mathbf{i} \in [n]_2} D_{\mathbf{i},\pi(\mathbf{i})} \bigl( \1 ( - D_{\mathbf{i}, \pi(\mathbf{i})}  \leq u \leq 0) - \1 ( 0 < u \leq -D_{\mathbf{i},\pi(\mathbf{i})} )	 \bigr), 
    \label{4khatu}
\end{equ}
and $D_{\mathbf{i},\mathbf{j}} = X_{i_1,j_1} + X_{i_2,j_2} - X_{i_1,j_2} - X_{i_2,j_1}$ for any $\mathbf{i} = (i_1, i_2)$ and $\mathbf{j} = (j_1,j_2)$.
Therefore, the condition \cref{e1} is satisfied. 

In what follows, denote by $C, C_{1}, C_{2},\dots$  absolute constants, which may take different values in different places. 
\subsection{A preliminary lemma}
The following lemma will be useful in the proof of
\cref{thm:11_24_1}, and the proof of this lemma is put in \cref{apdxB}. Let $\hat{K}_{1},
\hat{K}_{2,t}, \hat{K}_{3,t}, \hat{K}_{4,t}$ and $M_t$ be as in
\cref{ek1,ek2,eq:sec2_2,eq:sec2_3,eq:sec2_1} with $\hat{K}(u)$ defined as in \cref{4khatu} and
$h(t)=\IE\Psi_{\beta,t}(W)$. 
\begin{lemma}
    \label{lem:11_24_3}
    For $n\geq 4$ and $0 \leq t, \beta \leq \alpha_n^{1/3}/64 $,   we have 
\begin{align}
		\IE \{ |R| \Psi_{\beta, t}(W) \}                        
        & \leq  C_{0} b \alpha_n^{-1} \IE \{
			\Psi_{\beta, t}(W)  \}, \label{a21el3a}\\
    \IE \bigl\{ \lvert \hat K_1 - 1 \rvert \Psi_{\beta,t}(W) \bigr\}   & \leq C_{1} b
    (n \alpha_{n}^{-3}+ n^{1/2} \alpha_{n}^{-2}+ n^{-1/2}) \IE \{
			\Psi_{\beta, t}(W)  \}, \label{ap2el3b}\\
		\IE \{ \hat{K}_{2,t} \Psi_{\beta, t}(W) \}               
		& \leq C_{2} b  n \alpha_{n}^{-3}  \IE \{ \Psi_{\beta, t}(W)  \}, \label{ap2el3c}\\
    \IE \{ \hat{K}_{3,t}  \Psi_{\beta,t}(W) \} & \leq  C_{3}b^{2}  (  n \alpha_{n}^{-3}+   
             n^{2} \alpha_{n}^{-5})( 1 + t^{2} ) \IE \{
			\Psi_{\beta, t}(W)  \}, \label{eq-7.4}\\
            \IE \{ \hat{K}_{4,t}  \Psi_{\beta,t}(W) \} & \leq C_{4}b^{2} (  n \alpha_{n}^{-4}+   
            n^{2} \alpha_{n}^{-6})( 1 + t^{2} ) \IE \{
			\Psi_{\beta, t}(W)  \}\label{ap2eq:9_20_12} 
	   \end{align}
	   and 
	   \begin{align}
			\sup_{0 \leq t \leq m_0}M_t & \leq C_{5} b n \alpha_{n}^{-2}.
	\label{eq:ap2r_3}
	   \end{align}
\end{lemma}
\subsection{Proof of \autoref{thm:11_24_1}}
\begin{proof}[Proof of \cref{thm:11_24_1}]
We apply \cref{thm1} to prove
\cref{thm:11_24_1}. Recalling that $\theta_{n}= n^{1/2} \alpha^{-1}$, by \cref{lem:11_24_3},
we have condition (A1) is satisfied with $m_{0}=\alpha_{n}^{1/3}/64$,
\begin{equ}
    r_0 & = C_{0} b \alpha_{n}^{-1}, \quad   \tau_0  =  0, \quad  
    r_1  = C_{1} b ((\theta_{n}^{2}+\theta_{n}) \alpha_{n}^{-1}+ n^{-1/2}), \quad  \tau_1  =  0, \quad  
		r_2  = C_2 b \theta_n^2 \alpha_n^{-1}, \quad  \tau_2  =   0, \\
    r_3 & = 2 C_3 b^{2} (\theta_n^2 + 1)^2 \alpha_n^{-1}, \quad  \tau_3  =  2, \quad  
    r_4  = 2 C_4 b^{2} (\theta_n^2 + 1)^2 \alpha_n^{-2}, \quad  \tau_4  =  2, \quad  
		\rho  =  C_5 b \theta_n^2, \label{eq:sec7_30}
	\end{equ}
	where $C_1, C_2, C_3, C_4$ and $C_5$ are absolute constants.
 Recalling the definition of $\delta(t)$ in \cref{eq-deltat}, and noting that
    $m_0=\alpha_{n}^{1/3}/64$,  we have
  \begin{equation}
    \begin{aligned}
        \delta(m_{0})&\geq  r_{2} (1+m_{0}^{3})\geq r_{2} m_{0}^{3} \geq C
        b \theta_{n}^{2}.
    \end{aligned}
    \label{eq:sec7_1}
\end{equation}
Combining \cref{eq:sec7_30,eq:sec7_1}, we have that the right hand side of \cref{et1} is less than
\begin{equation}
    \begin{aligned}
        C\Bigl( \frac{1}{\theta_{n}^{2}}+ \theta_{n}^{2}+1\Bigr)\delta(z)&\leq 2 C\Bigl(
            \frac{1}{\theta_{n}^{2}}+
        \theta_{n}^{2}\Bigr) (r_{0}+r_{1}+r_{2}+r_{3}+r_{4}^{1/2})(1+z^{3})
                                                                \leq C \delta_{n}
                                                                     (1+z^{3}),
    \end{aligned}
    \label{eq:sec7_3}
\end{equation}
where $\delta_{n}=(\alpha_{n}^{-1}+n^{-1/2}) (\theta_{n}^{-2}+\theta_{n}^{6})$, and we used the
fact that $x^{2}+1/x^{2}>1$ for $x>0$ in the first inequaltiy.
On the other hand, by \cref{eq:sec7_30}, we have
\begin{equation}
    \begin{aligned}
        r_{0}^{1/(\tau_{0}+1)}&\leq C b \alpha_{n}^{-1},\\
        r_{1}^{1/(\tau_{1}+2)}&\leq Cb^{1/2}
        (\theta_{n}+1)(1+\theta_{n}^{-1/2})\alpha_{n}^{-1/2},\\
        r_{2}^{1/(\tau_{2}+3)}&\leq C b^{1/3} \theta_{n}^{2/3} \alpha_{n}^{-1/3}. 
\end{aligned}
    \label{eq:7_1}
\end{equation}
By \cref{eq:7_1},
\begin{equation}
    \begin{aligned}
        r_{0}^{1/(\tau_{0}+1)}+r_{1}^{1/(\tau_{1}+2)}+   r_{2}^{1/(\tau_{2}+3)}\leq C b
        (1+\theta_{n})(1+\theta_{n}^{-1/2})\alpha_{n}^{-1/3}.
    \end{aligned}
    \label{eq:sec7_4}
\end{equation}
Applying \cref{thm1}, and by \cref{eq:sec7_3,eq:sec7_4}, we obtain the desired result.
\end{proof}

\begin{appendix}
\section{Proof of Lemma \ref{l31}}\label{apdxA}
Throughout this section, we follow the notation and settings in \cref{sec:general method}. We write $h(t) = \IE \{ \Psi_{\beta, t}(W) \}$ and for any $i$, let $V_i = \sum_{i \in B_i} X_i$, $T_i  = \sum_{i \in B_i} |X_i|$ and $W_i = W - V_i$. In what follows, we give
two general lemmas, which will be used in the proof of \cref{l31}. The following lemmas give us some
technical 
inequalities.
\begin{lemma} \label{lem:loc}
Under \textup{(LD1)} and \textup{(LD2)}, let $\zeta_i = \zeta(X_{A_i}) \geq 0$ be a  function of
		$X_{A_i}$. Then, for $0 \leq t \leq m_0$, 
	\begin{align}
		\label{eq:loc1}
		\IE \bigl\{ \zeta_i e^{3t T_i} \Psi_{\beta,t}(W_{i})  \bigr\} & \leq 81 b^{1/4} h(t) \IE \bigl\{ \zeta_i e^{ 3a_{n}T_{i}/8} \bigr\},  \\
		\label{eq:loc2}
		\IE \bigl\{ \zeta_i T_i^2 e^{3t T_i} \Psi_{\beta,t}(W_{i})  \bigr\} & \leq C \kappa^2 b^{1/4} h(t) \tau^{-1} \IE \bigl\{ \zeta_i^2 e^{ 3a_{n}T_{i}/4} \bigr\} + C \kappa^2 \tau a_n^{-4} b h(t),
	\end{align}
	where $C > 0$ is an absolute constant and $\tau > 0$ is any positive number.
\end{lemma}

\begin{proof}[Proof of \cref{lem:loc}]
By H\"older's inequality, for any random variables $U_1 , U_2,
	U_3 \geq 0$, we have  
	\begin{equ}
		\IE \{ U_1 U_2 U_3 \} \leq ( \IE \{ U_1 U_2^{(1+\epsilon)/\epsilon} \} )^{\epsilon/(1 + \epsilon)} ( \IE \{ U_1 U_3^{(1 + \epsilon)} \} )^{1/(1 + \epsilon)},
        \label{eq:9_12_6}
	\end{equ} 
	where $\epsilon = 16 m_0 / a_n$. 
    Then
\begin{equation}
    \begin{aligned}
		0 < \epsilon \leq 1, \quad 
        \epsilon m_{0}^2 \leq 1/3 \quad\mbox{and} \quad
        (1 +\epsilon) m_{0} / \epsilon \leq a_n / 8.
    \end{aligned}
    \label{epsprop}
\end{equation}

Applying \cref{eq:9_12_6} with $U_1 = \zeta_i, U_2 = e^{t T_{i}}$ and $U_3 =\Psi_{\beta,t}(W_{i})$, and by \cref{epsprop}, we have 
    \begin{equation}
        \begin{aligned}
           \MoveEqLeft \IE \bigl\{ \zeta_i \Psi_{\beta, t}(W_i)   e^{3 t  T_{i}}\bigr\}
           \\&\leq  \bigl (\IE \bigl\{ \zeta_i    e^{3 (1+\varepsilon)t 
            T_{i}/\varepsilon}\bigr\} \bigr )^{\varepsilon/(1+\varepsilon)}
            \bigl (\IE \bigl\{ \zeta_i \Psi_{\beta, t}(W_i)^{1+\varepsilon}  \bigr\}
            \bigr )^{1/(1+\varepsilon)}
             \\&=   \bigl (\IE \bigl\{ \zeta_i    e^{ 3(1+\varepsilon)t 
            T_{i}/\varepsilon}\bigr\} \bigr )^{\varepsilon/(1+\varepsilon)}
            \bigl (\IE  \zeta_i   
            \bigr )^{1/(1+\varepsilon)}\bigl (\IE \bigl\{ \Psi_{\beta, t}(W_i)^{1+\varepsilon}  \bigr\}
            \bigr )^{1/(1+\varepsilon)}
            \\&\leq   \IE \bigl\{ \zeta_i    e^{ 3(1+\varepsilon)t 
            T_{i}/\varepsilon}\bigr\} 
            \bigl (\IE \bigl\{ \Psi_{\beta, t}(W_i)^{1+\varepsilon}  \bigr\}
            \bigr )^{1/(1+\varepsilon)}
            \\&\leq  \IE \bigl\{ \zeta_i    e^{3 a_{n} 
            T_{i}/8}\bigr\} 
            \IE \bigl\{ \Psi_{\beta, t}(W_i)^{1+\varepsilon}  \bigr\}
            ,
        \end{aligned}
        \label{eq:1_8_1}
    \end{equation}
    where the equality in the third line follows from the fact that $W_i$ is independent of
    $\zeta_i$ and the last inequality follows from
	\cref{epsprop} and $\Psi_{\beta,t} \geq 1$. Recalling the definition of $\Psi_{\beta,t}(w)$ in \cref{eq:sec1_1}, we have for any $u$ and $v$,
    \begin{equation}
        \begin{aligned}
            \Psi_{\beta,t}(u+v)\leq \Psi_{\beta,t}(u) e^{t |v|}.
        \end{aligned}
        \label{eq:3_15_2}
    \end{equation}
    By   \cref{eq:3_15_2} and H\"older's inequality,
\begin{equation}
    \begin{aligned}
		 \IE \bigl\{ \Psi_{\beta, t}(W_i)^{1+\varepsilon}  \bigr\} 
		\leq  \IE \bigl\{ \Psi_{\beta, t}(W)^{1+\varepsilon} e^{(1+\varepsilon) t  T_{i}} \bigr\} 
	   \leq H_{1}\times H_{2},
    \end{aligned}
    \label{eq:1_8_2}
\end{equation}
where     
\begin{equation*}
    \begin{aligned}
		H_{1} & =  \IE \bigl\{ \Psi_{\beta, t}(W)^{(1+\varepsilon)^{2}} \bigr\} , &
		H_{2} & = \bigl (\IE \bigl\{ e^{(1+\varepsilon)^{2} t T_{i}/\varepsilon} \bigr\} \bigr )^{\varepsilon/(1+\varepsilon)}.
    \end{aligned}
\end{equation*}
Recalling the definition of $\Psi_{\beta,t}(w)$, we have 
\begin{equation*}
    \begin{aligned}
        \Psi_{\beta,t}(w)\leq 2e^{m_{0}^{2}}+1\leq3e^{m_{0}^{2}} \quad \mbox{ for  $0\leq \beta,t\leq
        m_{0}$},
    \end{aligned}
\end{equation*}
which further implies that 
\begin{align}
	\Psi_{\beta,t}(w)^{(1 + \epsilon)^2}
	& \leq (3 e^{m_0^2})^{2\epsilon + \epsilon^2} \Psi_{\beta,t}(w).
	\label{eq:sec6_3}
\end{align}
By \cref{eq:sec6_3,epsprop}, we have
\begin{equation}
    \begin{aligned}
        H_{1}\leq 27 e^{m_{0}^2(2\varepsilon+\varepsilon^{2})} h(t)\leq 27
        e^{3m_{0}^2\varepsilon} h(t)\leq 81
        h(t).
    \end{aligned}
    \label{eq:1_8_3}
\end{equation}
For $H_{2}$, by \cref{epsprop} and by H\"older's inequality again, we have 
\begin{equation}
    \begin{aligned}
        H_{2}&\leq \E e^{ a_n T_i/4 } \leq b^{1/4}.
    \end{aligned}
    \label{eq:1_8_4}
\end{equation}
Combining \cref{eq:1_8_1,eq:1_8_2,eq:1_8_4,eq:1_8_3} yields \cref{eq:loc1}. 

We now prove \cref{eq:loc2}. Expanding the square term of the left hand side of \cref{eq:loc2}, we have for all $\tau > 0$,  
\begin{equ}
	\IE \bigl\{ \zeta_i T_i^2 e^{3t T_i} \Psi_{\beta,t}(W_{i})  \bigr\}
	& = \sum_{j \in B_i} \sum_{k \in B_i} \IE \bigl\{ \zeta_i \vert X_j X_k\vert e^{3t T_i} \Psi_{\beta,t}(W_{i})  \bigr\}\\
	& \leq \kappa \sum_{j \in B_i} \IE \bigl\{ \zeta_i X_j^2 e^{3t T_i} \Psi_{\beta,t}(W_{i})  \bigr\}\\
	& \leq \frac{\kappa^2}{2 \tau} \IE \{ \zeta_i^2 e^{6 t T_i} \Psi_{\beta,t}(W_{i}) \} + \frac{\kappa \tau}{2} \sum_{j \in B_i} \E \bigl\{ X_j^4 \Psi_{\beta,t}(W_i) \bigr\}.
    \label{eq-loc3}
\end{equ}
For the first term of the right hand side of \cref{eq-loc3}, by \cref{eq:loc1} with replacing $\zeta_i$ by $\zeta_i^2$ and $3t T_i$ by $6tT_i$, we obtain 
\begin{equ}
	\IE \{ \zeta_i^2 e^{6 t T_i} \Psi_{\beta,t}(W_{i}) \} 
	& \leq 81 b^{1/4} h(t) \E \{ \zeta_i^2 e^{ 3 a_n T_i /4 } \}.
    \label{eq-loc4}
\end{equ}
For the second term of the right hand side of \cref{eq-loc3}, by \cref{eq:3_15_2}, we have for any $j \in B_i$, with $W_{ij} = W - \sum_{k \in B_i \cup B_j} X_k$, 
\begin{align*}
	\E \bigl\{ X_j^4 \Psi_{\beta,t}(W_i) \bigr\}
	& \leq \E \{ X_j^4 e^{t T_j}\Psi_{\beta, t}(W_{ij}) \}.
\end{align*}
Similar to \cref{eq:loc1}, we obtain 
\begin{equ}
	\E \{ X_j^4 e^{t T_j}\Psi_{\beta, t}(W_{ij}) \}
	& \leq 81 b^{1/2} h(t) \E \{ X_j^4 e^{3 a_n T_j / 8} \}. 
    \label{eq-loc6}
\end{equ}
Observing that $\vert X_j\vert \leq T_j$, we have the expectation term of the right hand side of \cref{eq-loc6} can be bounded by 
\begin{equ}
\E \{ X_j^4 e^{t T_j}\} \leq \alpha_{n}^{-4} \E \{ (a_nT_j)^4 e^{3 a_n T_j / 8} \}\leq C \alpha_n^{-4} \E e^{a_nT_j/2}  \leq C a_n^{-4} b^{1/2} h(t). 
    \label{eq-loc5}
\end{equ}
Substituting \cref{eq-loc4,eq-loc5,eq-loc6} into \cref{eq-loc3} yields \cref{eq:loc2}.
\end{proof}
\begin{lemma}
	\label{lemA.2}
  Under \textup{(LD1)} and \textup{(LD2)},  for each $i \in \mathcal{J}$, let $\xi_{i}=\xi(X_{A_{i}})$ be a function of $X_{A_{i}}$ satisfying that $\IE \xi_{i}=0$.   
	 Let $S = \sum_{i\in \mathcal{J}} \xi_i$. 
	For $0 \leq t, \beta\leq m_0$ and any positive number $\tau$, we have
	\begin{equation}
		\begin{aligned}
			\IE   \bigl\{ S^2 \Psi_{\beta,t} (W)  \bigr\}
			  & \leq 81 b^{1/4} \kappa   h(t) \sum_{ i \in \mathcal{J} }  \IE \{\xi_{i}^2 e^{ 3a_n T_i/8 } \} \\
			  & \quad + C b \kappa^2 t^2 h(t) \sum_{i \in \mathcal{J}} \sum_{j \in \mathcal{J} \setminus N_i} \E \lvert \xi_j \rvert \bigl( \tau^{-1}\E \{ \xi_i^2 e^{3 a_n T_i / 4} \} + \tau a_n^{-4} \bigr).
		\end{aligned}
		\label{eq:9_12_3}
	\end{equation}	
\end{lemma}

\begin{proof}[Proof of \cref{lemA.2}]
Expanding the left hand side of \cref{eq:9_12_3} yields
	\begin{equation}
			\IE   \bigl\{ S^{2} \Psi_{\beta,t} (W)  \bigr\} :=I_{1}+I_{2},
		\label{eq:9_12_4}
	\end{equation}
	where
    \begin{equation*}
        \begin{aligned}
			I_{1} & = \sum^{ }_{i\in \mathcal{J}} \sum^{ }_{j\in N_{i}}\IE \left\{ \xi_{i} 
			\xi_{j}\Psi_{\beta,t}(W) \right\}, & 
            I_{2} & = \sum^{ }_{i\in \mathcal{J}} \sum^{ }_{j\in \mathcal{J}\backslash N_i}\IE  \left\{ \xi_{i}
			\xi_{j} \Psi_{\beta,t}(W)
            \right\}.
        \end{aligned}
    \end{equation*}
We now give the bounds of $I_1$ and $I_2$ separately. Observe that 
	\begin{equ}
		\{ (i, j): i \in \mathcal{J} ,  j \in N_i \} = \{ (i,j) : B_i \cap B_j \neq \emptyset \}= \{ (i, j) : j \in \mathcal{J}, i \in N_j \}.  
        \label{eq:9_12_5}
	\end{equ}
	Recall that $W_i = W - \sum_{j \in B_i} X_j$. 
	For $I_1$, we have 
	\begin{equ}
		I_1 
		& \leq \frac{1}{2}\sum_{i \in \mathcal{J}, \, j \in N_i} \IE \{ (\xi_i^2 + \xi_j^2) \Psi_{\beta,t}(W) \}\\
		& = \frac{1}{2}\sum_{i \in \mathcal{J}} \sum_{j \in N_i} \IE \{ \xi_i^2 \Psi_{\beta,t}(W) \} + \frac{1}{2}\sum_{i \in \mathcal{J}} \sum_{j \in N_i} \IE \{ \xi_j^2 \Psi_{\beta,t}(W) \}\\
		& = \frac{1}{2}\sum_{i \in \mathcal{J}} \sum_{j \in N_i} \IE \{ \xi_i^2 \Psi_{\beta,t}(W) \} + \frac{1}{2}\sum_{j \in \mathcal{J}} \sum_{i \in N_j} \IE \{ \xi_j^2 \Psi_{\beta,t}(W) \}\\
		& = \sum_{i \in \mathcal{J}} \sum_{j \in N_i} \IE \{ \xi_i^2 \Psi_{\beta,t}(W) \}\leq \kappa \sum_{i \in \mathcal{J}} \IE \{ \xi_i^2 \Psi_{\beta,t}(W_{i})  e^{tT_{i}}\},
        \label{eq:9_12_7}
	\end{equ}
    where we used \cref{eq:9_12_5}, \cref{eq:3_15_2} and  $|N_{i}|\leq \kappa$ in the last line.
    By \cref{lem:loc} with  $\zeta_i =\xi_{i}^2$, we obtain 
    \begin{equation}
        \begin{aligned}
            \IE \bigl\{\xi_{i}^{2} \Psi_{\beta,t}(W_{i}) e^{tT_{i}}\bigr\}\leq 81 b^{1/4} h(t) \IE
            \bigl\{ \xi_{i}^{2} e^{3a_{n}T_{i}/8} \bigr\}. 
        \end{aligned}
        \label{eq:1_9_1}
    \end{equation}
    Substituting \cref{eq:1_9_1} into \cref{eq:9_12_7},  we have 
	\begin{equ}
		I_1 \leq 81 b^{1/4} \kappa  h(t)  \sum_{ i \in \mathcal{J} }  \IE \{\xi_{i}^2 e^{
		3 a_n T_i/8 } \} .
        \label{eq:9_12_10}
	\end{equ}
	For $i, j \in \mathcal{J}$, let  $V_{ij} = \sum_{k \in B_i \cup B_j} X_k$, 
	\begin{math}
		W_{ij} = W - V_{ij},  T_{ij} = \sum_{k \in B_i \cup B_j} |X_k|. 
	\end{math}
	It is easy to see that $\lvert T_{ij} \rvert \leq T_{i} + T_j$.  	
    In order bound $I_2$, for any $i\in \mathcal{J} $ and $j\notin N_{i}$, by Taylor's expansion, we have  
	\begin{align*}
		\begin{split}
			\MoveEqLeft
			\IE \{ \xi_i \xi_j \Psi_{\beta, t}(W) \} \\
			& = \IE \{ \xi_i \xi_j \Psi_{\beta, t}(W_i) \} + \IE \{ \xi_i \xi_j V_i
			\Psi_{\beta,t}'(W_i) \} + \int_0^1 \IE \{ \xi_i \xi_j V_i^2 \Psi_{\beta,t}'' (W_i + u
		V_i) \}(1-u) du \\
			& = \IE \{ \xi_i \xi_j \Psi_{\beta, t}(W_i) \} + \IE \{ \xi_i \xi_j V_i \Psi_{\beta, t}'(W_{ij}) \} \\
			& \quad + \int_0^1 \IE \{ \xi_i \xi_j V_i (V_{ij} - V_i) \Psi_{\beta,t}''(W_{ij} + u
			(W_i - W_{ij}))  \}(1-u) du \\
			& \quad  +  \int_0^1 \IE \{ \xi_i \xi_j V_i^2 \Psi_{\beta,t}'' (W_i + u V_i) \} du .
		\end{split}
	\end{align*}
    If $j\notin N_{i} $, then $\xi_i$ is independent of $(\xi_j, W_i)$ and $\xi_j$ is independent of $(\xi_i, V_i, W_{ij})$. Recalling that $\IE  \xi_i  = \IE  \xi_j  = 0$, we have
	\begin{math}
		\IE \{ \xi_i \xi_j \Psi_{\beta, t}(W_i) \} = \IE \{ \xi_i \xi_j V_i \Psi_{\beta, t}'(W_{ij}) \} = 0.
	\end{math}
	By \cref{ePsiww} and by the monotonicity of $\Psi_{\beta, t}(\cdot)$,  we have for any $ 0 \leq u \leq 1 $,
	\begin{align*}
		\MoveEqLeft  \lvert \IE \{ \xi_i \xi_j V_i (V_{ij} - V_i) \Psi_{\beta, t}''(W_{ij} + u (W_i - W_{ij}))  \} \rvert\\
		& \leq t^2 \IE \{ \lvert \xi_i \xi_j V_i (V_{ij} - V_i) \rvert \Psi_{\beta,t}(W_{ij} + u
		(W_i - W_{ij})) \}\\ 
        & \leq t^2 \sum_{l \in B_i}\sum_{m \in B_j \setminus B_i} \IE \{ \lvert \xi_i \xi_j X_l X_m \rvert ( \Psi_{\beta, t}(W_i) + \Psi_{\beta, t}(W_{ij}) ) \}
	\\	& \leq 2 t^2 \sum_{l \in B_i}\sum_{m \in B_j \setminus B_i} \IE \{ \lvert \xi_i \xi_j X_l X_m \rvert 
        \Psi_{\beta, t}(W_{ij})  e^{t(T_{i}+T_{j})} \}, 
	\end{align*}
	and  similarly, 
	\begin{align*}
		\lvert \IE \{ \xi_i \xi_j V_i^2 \Psi_{\beta,t}'' (W_i + u V_i) \}  \rvert
		& \leq2  t^2 \sum_{l \in B_i}\sum_{m \in B_i} \IE \bigl\{  \lvert \xi_i \xi_j X_l X_m \rvert
         \Psi_{\beta, t}(W_{ij}) e^{t(T_{i}+T_{j})}  \bigr\}. 
	\end{align*}
	Observe that 
	\begin{align*}
		(T_i + T_j)^2 e^{t(T_i + T_j)} \leq 4 (T_i^2 e^{2t T_i} + T_j^2 e^{2t T_j}).
	\end{align*}
	Hence, it follows that 
    \begin{equ}
		I_{2}&\leq 2 t^2 \sum_{i\in\mathcal{J}} \sum_{j\in \mathcal{J}\backslash N_{i}} \sum_{l \in B_i}\sum_{m \in B_{i}\cup B_j} \IE \{ \lvert \xi_i \xi_j X_l X_m \rvert \Psi_{\beta, t}(W_{ij})  e^{t(T_{i}+T_{j})} \}\\
			 & \leq 2 t^2 \sum_{i \in \mathcal{J}} \sum_{j \in \mathcal{J} \setminus N_i} \E \bigl\{ \lvert \xi_i \xi_j (T_i + T_j)^2 \rvert e^{t( T_i + T_j )} \Psi_{\beta, t}(W_{ij}) \bigr\}\\
		 & \leq 8 t^2 \sum_{i \in \mathcal{J}} \sum_{j \in \mathcal{J} \setminus N_i} \E \bigl\{ \lvert \xi_i \xi_j \rvert (T_i^2 e^{2t T_i} + T_j^2 e^{2t T_j}) \Psi_{\beta,t}(W_{ij}) \bigr\}\\
		 & = 16 t^2 \sum_{i \in \mathcal{J}} \sum_{j \in \mathcal{J} \setminus N_i} \E \bigl\{ \lvert \xi_i \xi_j \rvert T_i^2 e^{2t T_i} \Psi_{\beta,t}(W_{ij}) \bigr\}, 
        \label{eq:12_24_1}
    \end{equ}
	where we used \cref{eq:9_12_5} in the last line. If $j \in \mathcal{J} \setminus N_i$, then
	$\xi_j$ is independent of $(\xi_i, T_i, W_{ij})$. Therefore, by \cref{eq:loc2} in
	\cref{lem:loc}, we obtain
	\begin{align*}
		I_2 & \leq 16 t^2 \sum_{i \in \mathcal{J}} \sum_{j \in \mathcal{J} \setminus N_i} \E | \xi_j | \E \bigl\{ \lvert \xi_i \rvert T_i^2 e^{2t T_i} \Psi_{\beta,t}(W_{ij}) \bigr\}\\
			& \leq 16 t^2 \sum_{i \in \mathcal{J}} \sum_{j \in \mathcal{J} \setminus N_i} \E | \xi_j | \E \bigl\{ \lvert \xi_i \rvert T_i^2 e^{3t T_i} \Psi_{\beta,t}(W_{i}) \bigr\}\\
			& \leq C b \kappa^2 t^2 h(t) \sum_{i \in \mathcal{J}} \sum_{j \in \mathcal{J} \setminus N_i} \E \lvert \xi_j \rvert \bigl( \tau^{-1}\E \{ \xi_i^2 e^{3 a_n T_i / 4} \} + \tau a_n^{-4} \bigr). 
	\end{align*}	
    Combining \cref{eq:9_12_10,eq:12_24_1}, we complete the proof.
	\end{proof}
Based on \cref{lem:loc,lemA.2}, we are now ready to give the proof of \cref{l31}.
\begin{proof}
[Proof of \cref{l31}] 
Recall that $\hat K(u), K(u), \hat K_i(u)$ and $K_i(u)$ are defined as in \cref{eki,eeki}, and  
	\begin{equ}
		\hat{K}_1 = \int_{-\infty}^{\infty} \hat K(t) dt =  \sum_{i \in \mathcal{J}} X_i Y_i, \quad 
		\hat{K}_{2,t} = \sum_{i \in \mathcal{J}} \int_{-\infty}^{\infty} |u| e^{t |u|} \hat K_i(u) du. 
        \label{ebk2}
	\end{equ}
	Since $\IE  W^2  = 1$, it follows that 
	\begin{math}
		\IE \hat K_1 = 1. 
	\end{math}
	 For \cref{ap1el3b}, by Jensen's and H\"older's inequalities, we have 
	\begin{equ}
        \IE \{ \lvert \IE[\hat{K}_1|W] -\IE \hat K_1 \rvert \Psi_{\beta, t}(W) \}	&\leq\IE \{
		\lvert \hat{K}_1 - \IE \hat K_1 \rvert \Psi_{\beta, t}(W) \} 
	\\	& \leq h(t)^{1/2} (\IE \{ \lvert \hat{K}_1 - \IE\hat K_1 \rvert^2 \Psi_{\beta, t}(W) \})^{1/2}.
        \label{Ibound}
	\end{equ}
	Recall that $\hat K_{1}-\IE \hat K_{1}= \sum_{i\in \mathcal{J}}( X_{i}Y_{i}-\IE
	X_{i}Y_{i})$. Then,	applying \cref{lemA.2} with $\xi_i = X_i Y_i - \IE \{ X_i Y_i \}$ and $\tau = b^{1/2}$, we have 
	\begin{equ}
		 \IE \{ \lvert \hat{K}_1 - \IE\hat K_1 \rvert^2 \Psi_{\beta, t}(W) \}
			 \leq G_1 + G_2,  
        \label{I1and2}
	\end{equ}
	where 
    \begin{equation*}
        \begin{aligned}
			G_{1} & =  81 b^{1/4} \kappa   h(t) \sum_{ i \in \mathcal{J} }  \IE \{\xi_{i}^2 e^{ a_n T_i/8 } \} , \\
			G_{2}& = C b \kappa^2 t^2 h(t) \sum_{i \in \mathcal{J}} \sum_{j \in \mathcal{J} \setminus N_i} \E \lvert \xi_j \rvert \bigl( b^{-1/2}\E \{ \xi_i^2 e^{3 a_n T_i / 4} \} + b^{1/2} a_n^{-4} \bigr).
         \end{aligned}
    \end{equation*}
	Recalling that $T_{i}= \sum^{}_{j\in B_{i}} \vert X_{j}\vert$, we have $\lvert \xi_i \rvert \leq
	T_i^2 + \E T_i^2$, and thus, for $0<s\leq 3/4$, 
	\begin{equation}
		\begin{aligned}
			\IE (\{X_{i}Y_{i}-\IE X_{i}Y_{i}\}^{2}e^{s a_{n}T_{i}})
                & \leq a_n^{-4} \IE
                (\{a_{n}^{4} T_{i}^{4}+ \IE a_{n}^{4} T_{i}^{4} \} e^{sa_{n}T_{i}})\\
				& \leq C a_n^{-4} (\E e^{a_n T_i})^{s + 1/4}\\
				& \leq C b^{s + 1/4} a_n^{-4},
		\end{aligned}
		\label{I1bound}
	\end{equation}
    where we used the inequality that $y^{4}\leq C e^{y/4}$ for
    $y\geq 0$ and some $C > 0$. 
	Moreover, 
	\begin{equ}
		\E \lvert \xi_j \rvert & \leq 2 \E T_j^2  \leq C a_n^{-2} \E e^{a_n T_j / 2} \leq C b^{1/2} a_{n}^{-2}.
        \label{eq-xijbound6}
	\end{equ}
	Substituting \cref{I1bound,eq-xijbound6} into \cref{I1and2} gives 
	\begin{align*}
		\IE \{ \lvert \hat{K}_1 - 1 \rvert^2 \Psi_{\beta, t}(W) \}
        & \leq C  (b \kappa n a_n^{-4} + b^2 \kappa^2n^2 a_n^{-6}) h(t) (1 + t^2), 
	\end{align*}
	which proves \cref{ap1el3b} together with \cref{Ibound}.

	We next prove \cref{ap1el3c}. 
	Recalling that $\hat K_i(u) $ is defined as in \cref{eki}, by \cref{eq:3_15_2} and applying \cref{lem:loc} with $\xi_i = \lvert \hat K_i(u) \rvert$, we have 
	\begin{align*}
		|\E \{ \hat K_i(u) \Psi_{\beta,t}(W) \}|
		& \leq \E \{ |\hat K_i(u)| e^{ t T_i } \Psi_{\beta,t}(W_i) \}\\
		& \leq 81 b^{1/4} h(t) \E \{ |\hat K_i(u)| e^{ 3 a_n T_i /8 } \}. 
	\end{align*}
	Thus, by \cref{ebk2} and recalling that $t \leq a_n/16$ and $\lvert X_i \rvert \leq T_i, \lvert Y_i \rvert \leq T_i$, we have 
	\begin{align*}
		\E \{ |\hat K_{2,t}| \Psi_{\beta,t}(W) \}
		& \leq 81 b^{1/4} h(t) \sum_{i \in \mathcal{J}} \E \{ \lvert X_i^2 Y_i \rvert e^{a_n T_i / 2} \} \\
		& \leq 81 b^{1/4} h(t) \sum_{i \in \mathcal{J}} \E \{ T_i^3 e^{a_n T_i/2} \}\\
		& \leq C b n a_n^{-3} h(t). 
	\end{align*}
We now move to prove \cref{ap1eq:9_20_11,ap1eq:9_20_12} together.
By definition, 
\begin{equation}
	\begin{aligned}
		\IE \{ \hat{K}_{3,t} \Psi_{\beta, t}(W) \}= \int_{ \lvert u \rvert \leq 1 }^{} e^{ 2 t |u|
        }\IE \left\{ \bigl( \hat{K}(u) - K(u) \bigr)^2 \Psi_{\beta, t}(W) \right\}du  
	\end{aligned}
	\label{eq:9_12_1}
\end{equation}
and 
\begin{equation}
	\begin{aligned}
		\IE \{ \hat{K}_{4,t} \Psi_{\beta, t}(W) \}= \int_{ \lvert u \rvert \leq 1 }^{} |u| e^{ 2 t |u|
        }\IE \left\{ \bigl( \hat{K}(u) - K(u) \bigr)^2 \Psi_{\beta, t}(W) \right\}du.  
	\end{aligned}
	\label{eq:1_12_1}
\end{equation}
For fixed $u$, applying \cref{lemA.2} with $\xi_i = \hat K_i(u)- K_{i}(u)$ and $\tau = b^{1/2}a_{n}$, we have  
\begin{equation}
	\begin{aligned}
       \IE \bigl\{ \bigl( \hat{K}(u) - K(u) \bigr)^2 \Psi_{\beta, t}(W) \bigr\}
        =  H_{1}(u) + H_{2}(u),
	\end{aligned}
	\label{eq:9_12_2}
\end{equation}
where 
\begin{equation*}
    \begin{aligned}
        H_{1}(u) & = 81 b^{1/4} \kappa   h(t) \sum_{ i \in \mathcal{J} }  \IE \{(\hat
        K_i(u)-K_{i}(u))^2 e^{ 3a_n T_i/8 } \},\\
		H_{2}(u) & = C b \kappa^2 t^2 h(t) \sum_{i \in \mathcal{J}} \sum_{j \in \mathcal{J}
            \setminus N_i} \E \lvert \hat K_j(u) -K_{j}(u)\rvert \bigl( b^{-1/2}a_n^{-1}\E \{ (\hat
        K_i(u)-K_{i}(u))^2  e^{3 a_n T_i / 4} \} +  b^{1/2}a_n^{-3} \bigr).
    \end{aligned}
\end{equation*}
For $H_1(u)$, recalling that $\lvert X_i \rvert	 \leq T_i, \lvert Y_i \rvert \leq T_i$ and $t \leq a_n/16$, 
\begin{equ}
	\int_{-\infty}^{\infty} e^{2 t |u|}\E \{ \hat K_i(u)^2 e^{3a_nT_i/8} \} du 
	& \leq \E \{ \lvert X_i^2 Y_i \rvert e^{a_n T_i/2} \} \\
	& \leq C a_n^{-3} \E \{ (a_n T_i)^3 e^{a_n T_i / 2} \}\\
	& \leq C b^{3/4} a_n^{-3},
    \label{eq-H1u}
\end{equ}
and similarly, 
\begin{equ}
	\int_{-\infty}^{\infty} e^{2 t |u|}\E \{ K_i(u)^2 e^{3a_nT_i/8} \} du 
	\leq C b^{3/4} a_n^{-3}. 
    \label{eq-H1u2}
\end{equ}
For $H_2(u)$, note that $\lvert \hat K_i(u) \rvert^{2} \leq \lvert X_i \rvert^2$, and we have 
\begin{align*}
    &\E \{ \lvert \hat K_i(u) \rvert^2 e^{3 a_n / 4} \} 
	  \leq \E \{ \lvert X_i \rvert^2 e^{3 a_n / 4} \}
	 \leq \E \{ T_i^2 e^{3a_n / 4} \} \leq C b a_n^{-2} 
\end{align*}
and 
\begin{align*}
    &\E \{ \lvert  K_i(u) \rvert^2 e^{3 a_n / 4} \} 
	  \leq C b a_n^{-2}.
\end{align*}
Then, 
\begin{equ}
	H_2(u) 
	& \leq C b^{3/2} \kappa^2 t^2 a_{n}^{-3} h(t) \sum_{i \in \mathcal{J} }
    \sum_{j \in \mathcal{J} \setminus N_i} \E \lvert \hat K_j (u)-K_{j}(u) \rvert. 
    \label{eq-H2u}
\end{equ}
Similar to \cref{eq-H1u,eq-H1u2},  
\begin{equ}
    \int_{-\infty}^{\infty} e^{2 t |u|}\E \lvert \hat K_j(u)-K_{i}(u) \rvert du
	& \leq \E \{ (\lvert X_i Y_i  \rvert+\IE\lvert X_i Y_i  \rvert)  e^{2t T_i} \} \leq C b^{1/2} a_n^{-2}. 
    \label{eq-H2u1}
\end{equ}
Substituting \cref{eq-H1u,eq-H1u2,eq-H2u,eq-H2u1,eq:9_12_2} into \cref{eq:9_12_1} gives 
\cref{ap1eq:9_20_11}.
The inequality \cref{ap1eq:9_20_12} can be shown similarly.

    It now remains to prove \cref{eq:ap1r_3}. By definition,
	\begin{equation}
		\begin{aligned}
            \sup\limits_{0\leq t\leq m_{0} } M_{t} &= \int_{ \lvert u \rvert \leq 1 } e^{ m_{0}|u| } \lvert K(u) \rvert du\\
            &\leq    \sum^{ }_{i\in \mathcal{J}} \IE \biggl\{ \int_{\vert u\vert \leq 1}^{}
				\lvert \hat K_i(u)  \rvert
            e^{m_{0}|u|}du \biggr\}\\
           &\leq     2\sum^{ }_{i\in \mathcal{J}} \IE 
           \biggl\{ \vert X_{i}Y_{i} \vert e^{m_{0}|Y_{i}|} \biggr\} \leq     2 a_{n}^{-2} \sum^{ }_{i\in \mathcal{J}} \IE 
           \bigl\{ \vert a_{n}^{2} T_{i}^{2} \vert e^{m_{0}T_{i}} \bigr\}\\
           &\leq  C n  a_{n}^{-2}   b.
		\end{aligned}
		\label{eq:r4b}
	\end{equation}
	This completes the proof. 
\end{proof}

\section{Proof of Lemma \ref{lem:11_24_3}}\label{apdxB}
This section includes three subsections. In \cref{sub:B1}, we prove \cref{lem:11_24_3}. Before that, we give some preliminary lemmas, whose proofs are given in \cref{sub:B2,sub:B3}.
\subsection{Proof of \autoref{lem:11_24_3}}%
\label{sub:B1}
For any $\mathbf{i}\in [n]_k$, we write $i_l$ be the $l$-th element of $\mathbf{i}$ and write $\pi(\mathbf{i}) = (\pi(i_1), \dots, \pi(i_k))$, Let $A(\mathbf{i}) = \{ i_1, \dots, i_k \}$ be the set of elements in $\mathbf{i}$. For any $\mathbf{i}, \mathbf{j} \in [n]_k$ and any
matrix
$(x_{i,j})_{1 \leq i , j \leq n}$, let
$x_{\mathbf{i},\mathbf{j}} = (x_{i,j} : i \in A(\mathbf{i}), j \in A(\mathbf{j}))$. Let $T_{\mathbf{i},\mathbf{j}} = \sum_{l = 1}^k \lvert X_{i_l,j_l} \rvert$.
For any two positive integers $m\leq n$, we write $n_{(m)}:= \prod ^{m}_{i=1} (n-i+1)$ as the \emph{descending factorial}.
	Let $h(t) = \E \{ \Psi_{\beta,t}(W) \}$.
	The following preliminary lemmas are useful in the proof of \cref{lem:11_24_3}. 
 
\begin{lemma}
	\label{lem-B1}
	Let $0 \leq m \leq 2$ be an integer and assume that $n \geq 4$. Let $\sigma$ be a uniform
    permutation on $[n - m]$ which is independent of $\mathbf{X}$ and let $S = \sum_{i \in [n - m]} X_{i, \sigma(i)}$. 
	For $k = 1,2$, and for any $\mathbf{i},\mathbf{j} \in [n - m]_k$, let $\zeta_{\mathbf{i},\mathbf{j}} \coloneqq \zeta (X_{\mathbf{i},\mathbf{j}})$ be a positive function of $X_{\mathbf{i},\mathbf{j}}$. 
	We have 
	\begin{align}
		\E \{ \zeta_{\mathbf{i}, \sigma(\mathbf{i})} \Psi_{\beta,t}(S) \}
		& \leq 4 b^{1/16} \E \Psi_{\beta,t}(S) \max_{\mathbf{v} \in [n-m]_k} \E \{ \zeta_{\mathbf{i},\mathbf{v}} e^{t T_{\mathbf{i},\mathbf{v}}} \}, \label{eq-lB1-a}\\
		\E \{ \zeta_{\sigma^{-1}(\mathbf{j}), \mathbf{j}} \Psi_{\beta,t}(S) \}
		& \leq 4 b^{1/16} \E \Psi_{\beta,t}(S) \max_{\mathbf{v} \in [n-m]_k} \E \{ \zeta_{\mathbf{i},\mathbf{v}} e^{t T_{\mathbf{i},\mathbf{v}}} \}, \label{eq-lB1-b}\\
		\E \{ \zeta_{\sigma^{-1}(\mathbf{j}), \sigma(\mathbf{i})} \Psi_{\beta,t}(S) \}
		& \leq 4 b^{1/8} \E \Psi_{\beta,t}(S) \max_{\mathbf{u},\mathbf{v} \in [n-m]_k} \E \{ \zeta_{\mathbf{u},\mathbf{v}} e^{t (T_{\mathbf{i},\mathbf{v}}+ T_{\mathbf{u},\mathbf{j}})}  \}. \label{eq-lB1-c}
	\end{align}
\end{lemma}
\begin{proof} 
We only prove \cref{eq-lB1-a}, because \cref{eq-lB1-b,eq-lB1-c} can be shown similarly.
For any $\mathbf{i} \in [n - m]_k$,  let 
	\begin{math}
		S^{(\mathbf{i})} = \sum_{i' \not\in A(\mathbf{i})} X_{i', \sigma(i')}. 
	\end{math}
	By definition, we have 
    \begin{equation}
        \begin{aligned}
			\E \{ \zeta_{\mathbf{i}, \sigma(\mathbf{i})} \Psi_{\beta,t}(S) \}
		& = \frac{1}{(n - m)_k}\sum_{\mathbf{j}\in [n-m]_{k}} \E \{ \zeta_{\mathbf{i},\mathbf{j}} \Psi_{\beta,t}(S) \vert
		\sigma( \mathbf{i}	 ) = \mathbf{j} \} \\
		& \leq \frac{1}{(n - m)_k}\sum_{\mathbf{j}\in [n-m]_{k}} \E \{ \zeta_{\mathbf{i},\mathbf{j}} e^{t T_{\mathbf{i},\mathbf{j}}} \Psi_{\beta,t}(S^{(\mathbf{i})}) \vert
		\sigma( \mathbf{i}	 ) = \mathbf{j} \}.
		\end{aligned}
       \label{eq-lemB1-1}
    \end{equation}
    Since $X_{\mathbf{i},\mathbf{j}}$  is conditionally independent of
    $S^{(\mathbf{i})}$ given the event that $\sigma(\mathbf{i}) = \mathbf{j}$, and
	$X_{\mathbf{i},\mathbf{j}}$ is independent of $\sigma$, the last conditional expectation in \cref{eq-lemB1-1} can be rewritten as 
	\begin{equ}
		\E \{ \zeta_{\mathbf{i},\mathbf{j}} e^{t T_{\mathbf{i},\mathbf{j}}} \Psi_{\beta,t}(S^{(\mathbf{i})}) \vert
		\sigma( \mathbf{i}	 ) = \mathbf{j} \}
		& = \E \{ \zeta_{\mathbf{i},\mathbf{j}} e^{t T_{\mathbf{i},\mathbf{j}}} \} \E \{ \Psi_{\beta,t}(S^{(\mathbf{i})}) \vert
		\sigma( \mathbf{i}	 ) = \mathbf{j} \} .
		\label{eq-lemB1-2}
	\end{equ}
	For the second term on the RHS of \cref{eq-lemB1-2}, 
	note that  $\Psi_{\beta,t}(w + x) \leq e^{t|x|} \Psi_{\beta,t}(w)$ and $\lvert S - S^{(\mathbf{i})} \rvert \leq T_{\mathbf{i},\mathbf{j}}$ given $\sigma(\mathbf{i}) = \mathbf{j}$.
	With $\epsilon = \alpha_{n}^{-2/3}$, by H\"older's inequality, and noting that $T_{\mathbf{i},\mathbf{j}}$ is independent of $\sigma$, we have 
	\begin{equ}
		\MoveEqLeft
		\E \{ \Psi_{\beta,t}(S^{(\mathbf{i})})  \vert \sigma(\mathbf{i}) = \mathbf{j}\}\\
		& \leq \E \{ e^{t T_{\mathbf{i},\mathbf{j}}} \Psi_{\beta,t}(S) \vert \sigma(\mathbf{i}) = \mathbf{j}	\}\\
		& \leq \bigl( \E \bigl\{ e^{(1 + \epsilon)t T_{\mathbf{i},\mathbf{j}}/\epsilon}\bigr\} \bigr)^{\epsilon/(1 + \epsilon)} \bigl( \E \bigl\{ \Psi_{\beta,t}^{1 + \epsilon}(S)\bigr\} \bigm\vert \sigma(\mathbf{i}) = \mathbf{j}\bigr)^{1/(1 + \epsilon)}. 
        \label{eq-lB1-21}
	\end{equ}
	Noting that $0 \leq t \leq \beta  \leq
	\alpha_{n}^{1/3}/64$, we obtain
	\begin{equ}
		0 < \epsilon < 1, \quad 
		(1 + \epsilon)t / \epsilon \leq 2t / \epsilon \leq \alpha_n/32, \quad \epsilon \beta t \leq 0.001.  
        \label{eq-com-eps}
	\end{equ}
	For the first term in the RHS of \cref{eq-lB1-21}, because $k=1,2$ and $\mathbf{i},\mathbf{j}\in
	[n-m]_{k}$, by \cref{eq:11_24_6}, we obtain 
	\begin{equ}
		\E \bigl\{ e^{(1 + \epsilon)t T_{\mathbf{i},\mathbf{j}}/\epsilon}\bigr\}
		\leq \E e^{\alpha_n T_{\mathbf{i},\mathbf{j}}/32} 
		\leq \max_{i,j \in [n-m]}\E e^{\alpha_n \lvert X_{i,j} \rvert/16} 
		\leq b^{1/16}.
        \label{eq-lB1-22}
	\end{equ}
	Noting that $\Psi_{\beta,t}(w) \leq 2 e^{ t \beta } + 1 \leq 3 e^{t \beta}$,
	we have 
	\begin{align*}
		\Psi_{\beta,t}(w)^{1 + \epsilon} \leq (3 e^{t \beta})^{\epsilon} \Psi_{\beta,t}(w) \leq 4 \Psi_{\beta,t}(w).
	\end{align*}
	Thus, 
	\begin{equ}
		\E \{ \Psi_{\beta,t}(S)^{1 + \epsilon} \vert \sigma(\mathbf{i}) = \mathbf{j}\} \leq 4 \E \{ \Psi_{\beta,t}(S) \vert \sigma(\mathbf{i}) = \mathbf{j}\} .
        \label{eq-lB1-13}
	\end{equ}
	Combining \cref{eq-lemB1-1,eq-lemB1-2,eq-lB1-21,eq-lB1-22,eq-lB1-13}, we have 
	\begin{align*}
		\E \{ \zeta_{\mathbf{i},\mathbf{j}} e^{t T_{\mathbf{i},\mathbf{j}}} \Psi_{\beta,t}(S^{(\mathbf{i})}) \vert
		\sigma( \mathbf{i}	 ) = \mathbf{j} \}
		& \leq 4 b^{1/16} \max_{\mathbf{v} \in [n-m]_k} \E \{ \zeta_{\mathbf{i},\mathbf{v}} e^{t T_{\mathbf{i},\mathbf{v}}} \} \E \{ \Psi_{\beta,t}(S) \vert \sigma(\mathbf{i}) = \mathbf{j}\}.
	\end{align*}
	Now, taking average over $\mathbf{j} \in [n - m]_k$ yields \cref{eq-lB1-a}. Using a similar argument, we obtain \cref{eq-lB1-b} and \cref{eq-lB1-c}.
\end{proof}

\begin{lemma}
    \label{lemB_3}
	For $\mathbf{i}, \mathbf{j}$, let $\xi_{\mathbf{i}, \mathbf{j}} \coloneqq \xi (X_{\mathbf{i},\mathbf{j}})$ be a function of $X_{\mathbf{i},\mathbf{j}}$ such that $\E  \xi_{\mathbf{i},
\pi(\mathbf{i})}  = 0$. For any $\mathbf{i} \in [n]_2$ and $\mathbf{i}' \in [n]_2^{(\mathbf{i})}$, where $[n]_2^{(\mathbf{i})}  \coloneqq \{ (k,l) \in [n]_2 : k,l \in [n] \setminus A(\mathbf{i}) \}$, we have
    \begin{equation*}
       \begin{aligned}
      \MoveEqLeft |\E \{ 
          \xi_{\mathbf{i},\pi(\mathbf{i})} \xi_{\mathbf{i}',\pi(\mathbf{i}')} 
	   \Psi_{\beta,t}(W) \}|\\ 
      &\leq C b (1 + t^{2}) h(t) n^{-4}\sum_{\mathbf{j}, \mathbf{j}' \in
           [n]_2}
                           \bigl( \E \{ \lvert\xi_{\mathbf{i},\mathbf{j}}\rvert
                                   T_{\mathbf{i},\mathbf{j}}^{2} e^{2 t\lvert T_{\mathbf{i},\mathbf{j}}\rvert}\}\IE\{
                                   \lvert\xi_{\mathbf{i}',\mathbf{j}'}\rvert\}+\E \{\lvert
                                   \xi_{\mathbf{i}',\mathbf{j}'}\rvert
         T_{\mathbf{i}',\mathbf{j}'}^{2} e^{2 t\lvert T_{\mathbf{i}',\mathbf{j}'}\rvert}\}\IE\{
     \lvert\xi_{\mathbf{i},\mathbf{j}}\rvert \}\bigr)\\
&\quad+C b (1 + t^{2}) h(t) n^{-4}\sum_{\mathbf{j}, \mathbf{j}' \in
[n]_2} \bigl(\alpha_{n}^{-2}+n^{-1}+\1(E_{\mathbf{j},\mathbf{j}'})\bigr)
                            \E \{ \lvert\xi_{\mathbf{i},\mathbf{j}}\rvert
                                   \}\IE\{ \lvert\xi_{\mathbf{i}',\mathbf{j}'}\rvert\},
    \end{aligned}
        \label{eq:B_84}
    \end{equation*}
    where $E_{\mathbf{j},\mathbf{j}'}= \{A(\mathbf{j})\cap
	A(\mathbf{j}') \neq \emptyset \}$.
\end{lemma}
Recalling that $D_{\mathbf{i},\mathbf{j}}$ is defined in \cref{4khatu}, we give the following
lemmas.
\begin{lemma}
    \label{lem:B_4}
	For $\mathbf{i}, \mathbf{j}\in [n]_2$,let 
    $ {g}_{\mathbf{i},\mathbf{j}}(u)=D_{\mathbf{i},\mathbf{j}} \bigl( \1 ( - D_{\mathbf{i},
   \mathbf{j}   }  \leq u \leq 0) - \1 ( 0 < u \leq -D_{\mathbf{i},\mathbf{j}} )	 \bigr)$ 
   and $\bar{g}_{\mathbf{i},\mathbf{j}}(u)={g}_{\mathbf{i},\mathbf{j}}(u)-\IE {g}_{\mathbf{i},\pi(\mathbf{i})}(u)$. We
   have 
   \begin{multline}
       \biggl\lvert \int_{|u|\leq 1}^{} |u|^{v} e^{2t \lvert u\rvert}\E \Biggl\{ \Bigl(\sum^{
       }_{\mathbf{i}\in [n]_{2}} \bar{g}_{\mathbf{i},\pi(\mathbf{i})}(u)\Bigr)^{2}  \Psi_{\beta,t}(W) \Biggr\}du \biggr\rvert \\
       \leq C b^{2} ( n^{2} \alpha_{n}^{-5-v} +
			n \alpha_{n}^{-3-v}) (1+t^{2}) h(t), \quad \text{for $v = 0, 1$}.       
       \label{eq:B.84}
   \end{multline}
\end{lemma}
\begin{lemma}
    \label{lem:B_108}
	For $\mathbf{i}, \mathbf{j}\in [n]_2$, let 
     $ H_{\mathbf{i},\mathbf{j}}(u)=D^{2}_{\mathbf{i},\mathbf{j}} $ 
   and $\bar{H}_{\mathbf{i},\mathbf{j}}(u)={H}_{\mathbf{i},\mathbf{j}}(u)-\IE {H}_{\mathbf{i},\pi(\mathbf{i})}(u)$. We
   have 
\begin{equation}
       \begin{aligned}
           \MoveEqLeft  \biggl\lvert \E \biggl\{ \biggl(\sum^{ }_{\mathbf{i}\in [n]_{2}}
           \bar{H}_{\mathbf{i},\pi(\mathbf{i})}\biggr)^{2}  \Psi_{\beta,t}(W) \biggr\} \biggr\rvert
           \leq C b^{2} ( n^{4} \alpha_{n}^{-6} +
			n^{3} \alpha_{n}^{-4} ) (1+t^{2}) h(t).       
       \end{aligned}
       \label{eq:B.107}
   \end{equation}
\end{lemma}
The proofs of \cref{lemB_3,lem:B_4} are given in \cref{sub:B3}. The proof of \cref{lem:B_108} is similar to that of \cref{lem:B_4} and thus we omit the details.

We are now ready to give the proof of \cref{lem:11_24_3}.
\def\EX{\E^{\mathbf{X}}}
\proof[Proof of \cref{lem:11_24_3}]
We prove \cref{a21el3a}--\cref{eq:ap2r_3} one by one. 

\medskip 
{\it\noindent (i). Proof of \cref{a21el3a}.}
Let $\bar X_{i,j} = X_{i,j} - a_{i,j}$. By H\"older's inequality, we have 
\begin{equ}
	\E \{ |R| \Psi_{\beta,t}(W) \}
	& \leq h^{1/2}(t) ( \E \{ R^2 \Psi_{\beta,t}(W) \} )^{1/2}. 
    \label{eq-l71-R}
\end{equ}
By \cref{Ex2,eq-s7-R}, we have 
\begin{equ}
	\E \{ R^2 \Psi_{\beta,t}(W) \} 
	 & = \frac{1}{n^2} \E \biggl\{ \biggl\lvert \sum_{i = 1}^n \sum_{j = 1}^n \bar X_{i,j} \biggr\rvert^2 \Psi_{\beta,t}(W) \biggr\} \\
	 & = \frac{1}{n^2} \sum_{i = 1}^n \sum_{j = 1}^n \sum_{i' = 1}^n \sum_{j' = 1}^n  \E \{ \bar X_{i,j} \bar X_{i',j'} \Psi_{\beta,t}(W) \}.
    \label{eq-l71-R0}
\end{equ}
Now, for fixed $i,j, i', j' \in [n]$, 
\begin{align*}
	\MoveEqLeft[0]
	\E \{ \bar X_{i,j} \bar X_{i',j'} \Psi_{\beta,t}(W) \}\\
	& = \E \{ \bar X_{i,j} \bar X_{i',j'} \Psi_{\beta,t}(W), \pi (i) = j, \pi(i') = j' \}  + \E \{ \bar X_{i,j} \bar X_{i',j'} \Psi_{\beta,t}(W) , \pi(i) = j, \pi(i') \neq j' \} \\
	& \quad + 
	\E \{ \bar X_{i,j} \bar X_{i',j'} \Psi_{\beta,t}(W), \pi (i) \neq j, \pi(i') = j' \} + \E \{ \bar X_{i,j} \bar X_{i',j'} \Psi_{\beta,t}(W) , \pi(i) \neq j, \pi(i') \neq j' \} . 
\end{align*}
For the first term,  with $W^{(i,i')} = W - \sum^{ }_{k\in \{i,i'\}}  X_{k,\pi(k)}$, let consider
the corresponding conditional expectation,  
\begin{align*}
	\begin{split}
		\MoveEqLeft|\E \{ \bar X_{i,j} \bar X_{i',j'} \Psi_{\beta,t}(W) \vert \pi (i) = j, \pi(i') = j' \}|\\
	& \leq \frac{1}{2}\E \{ (\bar X_{i,j}^2 + \bar X_{i',j'}^2) \Psi_{\beta,t}(W) \vert \pi(i) = j, \pi(i') = j' \}  \\
	& \leq \frac{1}{2} \E \{ (\bar X_{i,j}^2 + \bar X_{i',j'}^2) e^{t ( \lvert X_{i,j} \rvert + \lvert X_{i',j'} \rvert ) } \Psi_{\beta,t}(W^{(i,i')}) \vert \pi(i) = j, \pi(i') = j' \} \\
	& = \frac{1}{2} \E \{ (\bar X_{i,j}^2 + \bar X_{i',j'}^2) e^{t ( \lvert X_{i,j} \rvert + \lvert X_{i',j'} \rvert ) } \} \E \{ \Psi_{\beta,t}(W^{(i,i')}) \vert \pi(i) = j, \pi(i') = j' \},
	\end{split}
\end{align*}
where in the last line we used the fact that $(X_{i,j}, X_{i',j'})$ and $W^{(i,i')}$ are conditionally independent given $\pi(i) = j$ and $\pi(i') = j'$. Recalling that that $t \leq \alpha_n^{1/3}/4 \leq \alpha_n/4$ and by \cref{eq:11_24_6}, we have 
\begin{equ}
	\E \{ \bar X_{i,j}^2 e^{t \lvert X_{i,j} \rvert} \}
	& \leq C \alpha_n^{-2} \E \{ \lvert \alpha_n X_{i,j} \rvert^2 e^{ \alpha_n \lvert X_{i,j} \rvert/4 } \} \\
	& \leq C \alpha_n^{-2}\E \{ e^{\alpha_n \lvert X_{i,j} \rvert/2} \} \leq C b^{1/2} \alpha_n^{-2}.
    \label{eq-l71-Xbar}
\end{equ}
Choosing $\epsilon = \alpha_n^{-2/3}$ and according to \cref{eq:11_24_6} and \cref{eq-com-eps}, we have
\begin{align*}
	\E \{ e^{(1 + \epsilon)t ( \lvert X_{i,j} \rvert + \lvert X_{i',j'} \rvert ) / \epsilon} \} & \leq \E \{ e^{\alpha_n( \lvert X_{i,j} \rvert + \lvert X_{i',j'} \rvert )/2} \} \leq C b, 
\end{align*}
and 
\begin{align*}
	\MoveEqLeft\E \{ \Psi_{\beta,t}(W^{(i,i')}) \vert \pi(i) = j, \pi(i') = j' \}\\
	& \leq \E \{ e^{t \lvert W - W^{(i,j)} \rvert}\Psi_{\beta,t}(W) \vert \pi(i) = j, \pi(i') = j'\} \\
	& \leq ( \E \{ e^{(1 + \epsilon)t ( \lvert X_{i,j} \rvert + \lvert X_{i',j'} \rvert ) / \epsilon} \} )^{\frac{\epsilon}{1 + \epsilon}} ( \E \{ \Psi_{\beta,t}^{1 + \epsilon}(W) \vert \pi(i) = j, \pi(i') = j' \} )^{\frac{1}{1+\epsilon}}\\
    & \leq C b^{1/2}  \E \{ \Psi_{\beta,t}(W) \vert \pi(i) = j, \pi(i') = j' \} .
\end{align*}
Therefore, we have 
\begin{multline}
	\lvert \E \{ \bar X_{i,j} \bar X_{i',j'} \Psi_{\beta,t}(W) \vert \pi(i) = j, \pi(i') = j' \} \rvert \\
	\leq C b \alpha_n^{-2} \E \{ \Psi_{\beta,t}(W) \vert \pi(i) = j, \pi(i') = j' \}.
    \label{eq-l71-R1}
\end{multline}
Moreover, noting that $X_{i',j'}$ is independent of $(X_{i,j}, W)$ given $\pi(i) = j$ and $\pi(i')
\neq j'$, we have 
\begin{multline}
	\E \{ \bar X_{i,j} \bar X_{i',j'} \Psi_{\beta,t}(W) \vert \pi(i) = j, \pi(i') \neq j' \}\\
	= \E \{ \bar X_{i',j'} \} \E \{ \bar X_{i,j} \Psi_{\beta,t}(W) \vert \pi(i) = j, \pi(i') \neq j' \} = 0. 
    \label{eq-l71-R2}
\end{multline}
Similarly, 
\begin{equ}	
	\E \{ X_{i,j} X_{i',j'} \Psi_{\beta,t}(W) \vert \pi(i) \neq j, \pi(i') = j' \}  = 0 .
    \label{eq-l71-R3}
\end{equ}
Furthermore, if $\pi(i) \neq j$ and $\pi(i') \neq j'$, we have 
\begin{align*}
	\MoveEqLeft\E \{ \bar X_{i,j} \bar X_{i',j'} \Psi_{\beta,t}(W) \vert \pi(i) \neq j, \pi(i') \neq j' \} \\
	& = \E \{ \bar X_{i,j} \bar X_{i',j'} \} \E \{ \Psi_{\beta,t}(W) \vert \pi (i) \neq j, \pi(i') \neq j' \} \\
	& = 
	\begin{cases}
		\E \{ \bar X_{i,j}^2 \} \E \{ \Psi_{\beta,t}(W) \vert \pi(i) \neq j, \pi(i') \neq j' \} & \text{if $i = i'$ and $j = j'$}, \\
		0 & \text{otherwise}.
	\end{cases}
\end{align*}
By \cref{eq-l71-Xbar}, we obtain
\begin{multline}
	\label{eq-l71-R4}
	\lvert \E \{ \bar X_{i,j} \bar X_{i',j'} \Psi_{\beta,t}(W) \vert \pi(i) \neq j, \pi(i') \neq j' \} \rvert \\ 
	\leq C b \alpha_n^{-2} \E \{ \Psi_{\beta,t}(W) \vert \pi(i) \neq j, \pi(i') \neq j' \} \1 ( (i,j) = (i',j') ).
\end{multline}
Substituting \cref{eq-l71-R1,eq-l71-R2,eq-l71-R3,eq-l71-R4} to \cref{eq-l71-R0} and using \cref{eq-l71-R}
yields \cref{a21el3a}.

\medskip
{\it \noindent (ii). Proof of \cref{ap2el3b}.}
Recalling the definition of $\hat{K}_{1}$ in \cref{ek1} and $\hat{K}(u)$ in \cref{4khatu}, we have
\begin{equation}
    \begin{aligned}
        \hat{K}_{1}= \frac{1}{4n} \sum^{ }_{\mathbf{i}\in [n]_{2}}
        D_{\mathbf{i},\pi(\mathbf{i})}^{2}.
    \end{aligned}
    \label{eq:B_100}
\end{equation}
By \cref{eq:B_100,eq:2020_11_8}, one can verify (see, e.g.,  Eq. (3.10)  in  \cite{chen2015error}) that 
\begin{math}
    \vert \IE \hat{K}_{1}-1\vert\leq 2/\sqrt{n}.
\end{math}
Thus, 
\begin{equation}
    \begin{aligned}
        \IE \{ \lvert \hat{K}_{1}-1\rvert \Psi_{\beta,t}(W)\}\leq \IE \{ \lvert \hat{K}_{1}-\IE
        \hat{K}_{1}\rvert \Psi_{\beta,t}(W)\}+ \frac{2}{\sqrt{n}} \IE \{\Psi_{\beta,t}(W)\}.
    \end{aligned}
    \label{eq:B_102}
\end{equation}
For the first term of the R.H.S. of \cref{eq:B_102}, recalling that $h(t) = \E \Psi_{\beta,t}(W)$,
by H\"older's inequality, we have 
\begin{equation}
    \begin{aligned}	
    \MoveEqLeft \IE \{ \lvert \hat{K}_{1}-\IE \hat{K}_{1}\rvert \Psi_{\beta,t}(W)\}
	\\&	\leq \frac{h^{1/2}(t)}{4n} \biggl(\IE \biggl\{ \biggl( \sum^{ }_{\mathbf{i}\in [n]_{2}}
                  (D_{\mathbf{i},\pi(\mathbf{i})}^{2}-\IE
              D_{\mathbf{i},\pi(\mathbf{i})}^{2})\biggr)^{2}
   \Psi_{\beta,t}(W)\biggr\}\biggr)^{1/2}  .
    \end{aligned}
    \label{eq:B_103}
\end{equation}
  Applying \cref{lem:B_108} to the expectation in the RHS of \cref{eq:B_103},
 we obtain
\begin{equation}
    \begin{aligned}
        \IE \{ \lvert \hat{K}_{1}-\IE \hat{K}_{1}\rvert \Psi_{\beta,t}(W)\} \leq C b
        (n \alpha_{n}^{-3}+ n^{1/2} \alpha_{n}^{-2}) h(t).
    \end{aligned}
    \label{eq:B_111}
\end{equation}
Combining \cref{eq:B_111,eq:B_102} yields \cref{ap2el3b}.

\medskip
{\it \noindent (iii). Proof of \cref{ap2el3c}.}
Recalling $\hat K(u)$ as in \cref{4khatu}, we have 
\begin{equation}
    \begin{aligned}
		E \{ \hat K_{2,t} \Psi_{\beta,t}(W) \}
	& = \int_{-\infty}^{\infty} |u | e^{t |u|} \E \{ \hat K(u) \Psi_{\beta,t}(W) \} du \\
	& \leq  \frac{1}{4n} \sum_{\mathbf{i} \in [n]_2} 
    \E \{ |D_{\mathbf{i},\pi(\mathbf{i})}|^{3} e^{t \lvert D_{\mathbf{i},\pi(\mathbf{i})}\rvert} \Psi_{\beta,t}(W) \}  .
	\end{aligned}
    \label{eq:B_98}
\end{equation}
Then, applying \cref{lem-B1} with $k = 2$, $m = 0$, $\sigma = \pi$, $S = W$, and $\zeta_{\mathbf{i},\mathbf{j}}(u) =
|D_{\mathbf{i},\mathbf{j}}|^{3} e^{t \lvert D_{\mathbf{i},\mathbf{j}}\rvert}$, we have for
any $\mathbf{i}\in [n]_{2}$,
\begin{equation}
    \begin{aligned}
		\MoveEqLeft
        \IE \{|D_{\mathbf{i},\pi(\mathbf{i})}|^{3} e^{t \lvert
D_{\mathbf{i},\pi(\mathbf{i})}\rvert} \Psi_{\beta,t}(W)\}\\
		&\leq C b^{1/8} h(t)  \max_{ \mathbf{j} \in [n]_2} 
        \E \{ |D_{\mathbf{i},\mathbf{j}}|^{3} e^{t \lvert
D_{\mathbf{i},\mathbf{j}}\rvert + t T_{\mathbf{i},\mathbf{j}}   } \}\\
		&\leq Cb^{1/8} h(t) \alpha_{n}^{-3} \max_{j_1,j_2 \in [n]}\sum_{i \in \{ i_1,i_2 \}} \sum_{j \in \{j_1,j_2\}} \E \{ \lvert \alpha_n X_{i,j} \rvert^3 e^{ 6 t \lvert X_{i,j} \rvert } \}\\
		& \leq C b \alpha_{n}^{-3} h(t),
    \end{aligned}
    \label{eq:B_99}
\end{equation}
where the last inequality follows from \cref{eq:11_24_6}.
Therefore, by \cref{eq:B_98,eq:B_99},  we have 
\begin{align*}
	\E \{ \hat K_{2,t} \Psi_{\beta,t}(W) \} & \leq C b n \alpha_{n}^{-3}. 
\end{align*}

\medskip
{\it \noindent (iv). Proofs of \cref{eq-7.4,ap2eq:9_20_12}.}
Recall the definitions in 
\cref{4khatu,eq:sec2_2,eq:sec2_3}, 
and we have 
\begin{equation}
    \begin{aligned}
        \IE \{\hat{K}_{3,t} \Psi_{\beta,t}(W)\}
		& =   \int_{\left|u\right|\leq 1}^{} e^{2t\left|u\right|} \IE \{(\hat{K}(u)-K(u))^{2}
		\Psi_{\beta,t}(W)\} du. \
    \end{aligned}
    \label{k5hat}
\end{equation}
By 
\cref{lem:B_4} with $v=0$, we complete the proof of \cref{eq-7.4}. 
By \cref{lem:B_4} with $v=1$,  the inequality
\cref{ap2eq:9_20_12} follows similarly.

\medskip
{\it \noindent (v). Proofs of \cref{eq:ap2r_3}.} Recalling the definition of $\hat{K}(u)$ in
\cref{4khatu}, by Fubini's theorem we have 
\begin{equation}
    \begin{aligned}
        \sup_{0\leq t\leq \alpha_{n}^{1/3}/64} M_{t}\leq \frac{1}{n} \sum^{ }_{\mathbf{i}\in
        [n]_2}\IE \{e^{\alpha_{n} \lvert D_{\mathbf{i},\pi(\mathbf{i})}\rvert/64 }\lvert
    D_{\mathbf{i},\pi(\mathbf{i})}\rvert^{2}\} \leq C b n \alpha^{-2} 
    \end{aligned}
    \label{eq:B_136}
\end{equation}
where the last inequality follows from the similar argument in \cref{eq:B_99}.

\subsection{Some useful lemmas}%
\label{sub:B2}
In order to prove \cref{lemB_3,lem:B_4}, we need to show some preliminary lemmas. 
Recall that $\mathcal{S}_n$ is the collection of all permutations over $[n]$. 
\begin{lemma}
	\label{lemB5}
	For $n \geq 4$, $m = 0, 1, 2$, let $S$ and $\sigma$ be defined as in \cref{lem-B1}. For any $i,j \in [n-m]$, we have 
	\begin{align}
		|\E \{ X_{i, \sigma(i)} \Psi_{\beta,t}'(S) \}| & \leq C (n^{-1} \alpha_{n}^{-1} + \alpha_{n}^{-2}) b^{1/2} (1 + t^2) \E \Psi_{\beta,t}(S), \label{eq-lB5-a}\\
		|\E \{ X_{\sigma^{-1}(j), j} \Psi_{\beta,t}'(S) \}| & \leq C (n^{-1} \alpha_{n}^{-1} + \alpha_{n}^{-2}) b^{1/2} (1 + t^2) \E \Psi_{\beta,t}(S), \label{eq-lB5-b}\\ 
		|\E \{ X_{\sigma^{-1}(j), \sigma(i)} \Psi_{\beta,t}'(S) \}| & \leq C (n^{-1} \alpha_{n}^{-1} + \alpha_{n}^{-2}) b^{1/2} (1 + t^2) \E \Psi_{\beta,t}(S). \label{eq-lB5-c}
	\end{align}
\end{lemma}
\begin{proof}
	[Proof of \cref{lemB5}]
	We only prove \cref{eq-lB5-a}, because \cref{eq-lB5-b,eq-lB5-c} can be shown similarly. 
	Note that 
	\begin{equ}
		\E \{ X_{i, \sigma(i)} \Psi_{\beta,t}'(S) \}
		& = \frac{1}{n - m} \sum_{j = 1}^{n - m} \E \{ X_{i,j} \Psi_{\beta,t}'(S) \vert \sigma (i) = j \}\\
		& = \frac{1}{n - m} \sum_{j = 1}^{n - m} \E \{ X_{i,j} \Psi_{\beta,t}'(S^{(i)}) \vert \sigma(i) = j\} \\
		& \quad + \frac{1}{n - m} \sum_{j = 1}^{n - m} \E \{ X_{i,j} ( \Psi_{\beta,t}'(S) - \Psi_{\beta,t}'(S^{(i)}) )  \vert \sigma(i) = j \}\\
		& := I_1 + I_2. 
        \label{eq-lb5-1}
	\end{equ}	
	Denote by $\tau_{i,j}$ the transposition of $i$ and $j$, and define  
	\begin{align*}
        \sigma_{i,j} 
		= 
		\begin{cases}
			\sigma & \text{if $\sigma(i) = j$},\\
			\sigma \circ \tau_{i,\sigma^{-1}(j)} & \text{if $\sigma(i) \neq j$}.
		\end{cases}
	\end{align*}
    Then $\sigma_{i,j}(i) = j$. For any given distinct $k_1, \dots, k_{n - m - 1} \in [n - m] \setminus \{ i \}$ and $l_1, \dots, l_{n - m - 1} \in[n - m]\setminus \{j\}$, denote by $A $ the event that $\{\sigma_{i,j}(k_u)= l_u, u = 1, \dots, n - m - 1  \}$. Then, 
    \begin{equation*}
        \begin{aligned}
           \IP(A)
			&= \IP(A, \sigma(i) = j) 
            + \sum_{u = 1}^{n - m - 1} \IP(A,
            \sigma(i)=l_u, \sigma(k_u)=j)
		  \\&= \frac{1}{(n-m)!}+ (n-m-1)\frac{1}{(n-m)!} = \frac{1}{(n - m - 1)!}.
        \end{aligned}
    \end{equation*}
	On the other hand, we have 
	\begin{align*}
	\IP(\sigma(k_u)=l_u, u = 1, \dots, n - m - 1 \vert \sigma(i)=j) = \frac{1}{(n - m - 1)!}.	
	\end{align*}
    This proves that $	\law(\sigma_{i,j}) = \law(\sigma \vert \sigma(i) = j)$. Moreover,
	with $S^{(i)} = S - X_{i,\sigma(i)}$, $S_{i,j} = \sum_{i' = 1}^n X_{i',\sigma_{i,j}(i')}$, and let $S_{i,j}^{(i)} = S_{i,j} -
	X_{i,j}$, it follows that 
	\begin{align*}
	 \law(S_{i,j}^{(i)}) = \law (S^{(i)} \vert \sigma (i) = j). 
	\end{align*}
	Noting that $X_{i,j}$ is independent of $\Psi_{\beta,t}(S^{(i)})$ conditional on $\sigma(i) = j$, and recalling that $\E X_{i,j} = a_{i,j}$, 
	we have 
	\begin{align*}
		\E \{ X_{i,j} \Psi_{\beta,t}'(S^{(i)}) \vert \sigma(i) = j\}
		& = a_{i,j} \E \{ \Psi_{\beta,t}'(S^{(i)}) \vert \sigma(i) = j \}
		= a_{i,j} \E \{ \Psi_{\beta,t}'(S_{i,j}^{(i)}) \}.
	\end{align*}
    Therefore, recalling that $\sum_{j \in [n]} a_{i,j} = 0$, by assumption \cref{Ex2}, we obtain
	\begin{equ}
		I_1 & =  \frac{ \E \Psi_{\beta,t}'(S) }{n - m} \sum_{j \in [n - m]} a_{i,j} + \frac{1}{n - m} \sum_{j \in [n - m]} a_{i,j} \bigl( \E \Psi_{\beta,t}'(S_{i,j}^{(i)}) - \E \Psi_{\beta,t}'(S) \bigr)\\
		& = I_{11} + I_{12}	, 
        \label{eq-lB5-3}
	\end{equ}
	where 
	\begin{align*}
		I_{11} & =  - \frac{ \E \Psi_{\beta,t}'(S) }{n - m} \sum_{j \in [n] \setminus [n - m]} a_{i,j}, \\
		I_{12} & = \frac{1}{n - m} \sum_{j \in [n - m]} a_{i,j} \bigl( \E \Psi_{\beta,t}'(S_{i,j}^{(i)}) - \E \Psi_{\beta,t}'(S) \bigr).
	\end{align*}
	For $I_{11}$, by \cref{eq:11_24_6} and Jensen's inequality, 
	\begin{equ}
		\max_{i,j} \lvert a_{i,j} \rvert 
		& \leq \max_{i , j} \E \lvert X_{i,j} \rvert \leq  \alpha_n^{-1} \max_{i,j} \E \{ \lvert \alpha_{n} X_{i,j} \rvert \}\\
		& \leq \alpha_{n}^{-1} \max_{i,j} \log \E e^{ \lvert \alpha_n X_{i,j} \rvert } \leq \alpha_{n}^{-1} \log b. 
        \label{eq-maxab}
	\end{equ}
	Thus, by \cref{ePsitt}, noting that $0 \leq m \leq 2$ and $n - m \geq n / 2$, we have 
	\begin{equ}
		| I_{11} | \leq \frac{t m \alpha_n^{-1} \log b}{n - m} \E \Psi_{\beta,t}(S) \leq C t n^{-1} \alpha_{n}^{-1} \log b \E \Psi_{\beta,t}(S).
        \label{eq-lB5-11}
	\end{equ}
	For $I_{12}$, note that 
	\begin{align*}
		S - S_{i,j}^{(i)} 
		& =  ( X_{i,\sigma(i)} + X_{\sigma^{-1}(j), j} - X_{\sigma^{-1}(j), \sigma(i)} ) \1 ( \sigma (i) \neq j ) + X_{i,\sigma(i)} \1(\sigma(i) = j)\\
		& \leq \lvert X_{i,\sigma(i)} \rvert + \lvert X_{\sigma^{-1}(j), j}  \rvert + \lvert X_{\sigma^{-1}(j), \sigma(i)} \rvert.
	\end{align*}
	Moreover,  
	\begin{align*}
		\MoveEqLeft[1]
		|\Psi_{\beta,t}'(S_{i,j}^{(i)}) - \Psi_{\beta,t}'(S) |\\
		& \leq t \lvert S - S_{i,j}^{(i)} \rvert e^{t \lvert S - S_{i,j}^{(i)} \rvert} \Psi_{\beta,t}(S)\\
		& \leq 3 t \Psi_{\beta,t}(S) \Bigl( \lvert X_{i,\sigma(i)} \rvert e^{3 t \lvert X_{i,\sigma(i)} \rvert}  + \lvert X_{\sigma^{-1}(j), j}  \rvert e^{3t \lvert X_{\sigma^{-1}(j), j} \rvert} + \lvert X_{\sigma^{-1}(j), \sigma(i)} \rvert e^{3 t \lvert X_{\sigma^{-1}(j), \sigma(i)} \rvert} \Bigr). 
	\end{align*}
	Applying \cref{lem-B1} with $k = 1$ and $\zeta_{i,j} = \lvert X_{i,j} \rvert e^{3t \lvert X_{i,j} \rvert}$, and noting that $t \leq \alpha_n / 64$, we have 
	\begin{equ}
		\E |\Psi_{\beta,t}'(S_{i,j}^{(i)}) - \Psi_{\beta,t}'(S) |
		& \leq 36 b^{1/8} t \E \Psi_{\beta,t}(S) \max_{i,j \in [n]} \E \{ \lvert X_{i,j} \rvert e^{4 t \lvert X_{i,j} \rvert} \} \\
		& \leq 36 \alpha_{n}^{-1} b^{1/8} t \E \Psi_{\beta,t}(S) \max_{i,j \in [n]} \E \{ \lvert \alpha_n X_{i,j} \rvert e^{\alpha_n \lvert X_{i,j} \rvert/16} \} \\
		& \leq C \alpha_{n}^{-1} b^{1/8}t  \E \Psi_{\beta,t}(S) \max_{i,j \in [n]} \E \{ e^{\alpha_n \lvert X_{i,j} \rvert/8} \} \\
		& \leq C \alpha_n^{-1} b^{1/4} t \E \Psi_{\beta,t}(S). 
        \label{eq-lB5-2}
	\end{equ}
	By \cref{eq-lB5-2,eq-maxab}, we obtain  
	\begin{equ}
		|I_{12}| \leq C \alpha_n^{-2} t(b^{1/4}  \log b)  \E \Psi_{\beta,t}(S).
        \label{eq-lB5-12}
	\end{equ}
	Combining \cref{eq-lB5-11,eq-lB5-12} yields 
	\begin{equ}
		\vert I_1\vert \leq C b^{1/2} (n^{-1} \alpha_{n}^{-1} + \alpha_{n}^{-2}) t \E \Psi_{\beta,t}(S) .
        \label{eq-lB5I1}
	\end{equ}
	For $I_2$, observing that 
	\begin{align*}
		\E \{ X_{i,j} ( \Psi_{\beta,t}'(S) - \Psi_{\beta,t}'(S^{(i)}) )  \vert \sigma(i) = j \}
		& \leq t^2 \E \{ \lvert X_{i,j} \rvert^2 e^{t \lvert X_{i,j} \rvert} \Psi_{\beta,t}(S) \}, 
	\end{align*}
	we have 
	\begin{align*}
		|I_2| & \leq t^2 \E \{ \lvert X_{i,\sigma(i)} \rvert^2 e^{t \lvert X_{i,\sigma(i)} \rvert} \Psi_{\beta,t}(S) \}.
	\end{align*}
	Applying \cref{lem-B1} with $k = 1$ and $\zeta_{i,j} = \lvert X_{i,j} \rvert^2 e^{t \lvert X_{i,j} \rvert}$, we have 
	\begin{equ}
	|I_2| & \leq C t^2 b^{1/8} h(t) \max_{i,j \in [n]} \E \{ \lvert X_{i,j} \rvert^2 e^{2t \lvert X_{i,j} \rvert} \} \\
			& \leq C t^2 b^{1/4} \alpha_{n}^{-2} \E \Psi_{\beta,t}(S).
        \label{eq-lB5I2}
	\end{equ}
	Combining \cref{eq-lB5I1,eq-lB5I2} yields \cref{eq-lB5-a}.
\end{proof}


Recall that $\tau_{i,j}$ is the transposition of $i$ and $j$. For $n \geq 4$, $m = 0, 1, 2$, and any permutation $\sigma \in \mathcal{S}_{n - m}$, define the transform 
	\begin{align}
		\label{eq-Pcalpi}
		\mathcal{P}_{\mathbf{i},\mathbf{j}}\sigma
		= 
		\begin{cases}
			\sigma & \text{if $\sigma(\mathbf{i}) = \sigma(\mathbf{j})$}, \\
			\sigma \circ \tau_{\sigma^{-1}(j_1), i_1} & \text{if $\sigma(i_1) \neq j_1$ and $\sigma(i_2) =
			\sigma(j_2)$}, \\
			\sigma \circ \tau_{\sigma^{-1}(j_2), i_2} & \text{if $\sigma(i_1) = j_1$ and $\sigma(i_2) \neq
			\sigma(j_2)$}, \\
				\sigma \circ \tau_{\sigma^{-1}(j_2), i_1} \circ  \tau_{\sigma^{-1}(j_1),i_2} \circ \tau_{i_1,i_2}  & \text{if $\sigma(i_1) \neq j_1$ and $\sigma(i_2) \neq \sigma(j_2)$}. 
		\end{cases}
	\end{align}
	The transformation \cref{eq-Pcalpi} was constructed by \cite{Goldstein2005}, and further applied
	by \cite{chen2015error} to prove a Berry--Esseen bound for combinatorial central limit
	theorems. In the following lemmas, we use this transformation to calculate the conditional
	expectations of
	functions of $W$ given $\pi(i_1) = j_1$ and $\pi(i_2) = j_2$.
\begin{lemma}
    \label{lem:B_6}
	Let  $S$ and $\sigma$ be defined as in \cref{lem-B1}.
    For any $\mathbf{i}=(i_{1},i_{2})\in [n-m]_{2}$, $\mathbf{j}=(j_{1},j_{2})\in
    [n-m]_{2}$ and $1\leq p,q\leq 2$, we have
    \begin{multline}
        \E \{ |X_{\sigma^{-1}(j_{q}),
            \sigma(i_{p})}| \Psi_{\beta,t}'(S) \1 (\sigma(i_1) = j_1 \text{ or } \sigma(i_2) = j_2)
		   \}
		   \leq C b t n^{-1} \alpha_{n}^{-1} \E \Psi_{\beta,t}(S).
        \label{eq:B_113}
    \end{multline}
\end{lemma}
\begin{proof}[Proof of \cref{lem:B_6}]
Let $\Gamma_{1}=\{\sigma(i_1) = j_1 \text{ or } \sigma(i_2) = j_2\}$ and $\Gamma_{u,v}=\{\sigma(u)=j_{q},
\sigma(i_{p})= v\}$. 
By the law of total expectation, for any $1\leq p,q\leq 2$, we have
\begin{equation}
    \begin{aligned}
         \MoveEqLeft\E \{ |X_{\sigma^{-1}(j_{q}),
			 \sigma(i_{p})}| \Psi_{\beta,t}'(S)\1\{\Gamma_{1}\} 
           \}\\
&=
\sum^{}_{u,v\in
    [n - m]}  \E \bigl\{ |X_{u, v}| \Psi_{\beta,t}'(S) \1 (\Gamma_{1})\bigm\vert \Gamma_{u,v} \bigr\} \IP(\Gamma_{u,v}),\\
&\leq
\sum^{}_{u,v\in
    [n - m]}  t\E \bigl\{ |X_{u,
	v}|e^{t \lvert X_{u,j_{q}}\rvert} \Psi_{\beta,t}(S^{(u)}) \1 (\Gamma_{1})\bigm\vert \Gamma_{u,v} \bigr\} \IP(\Gamma_{u,v}),
    \end{aligned}
    \label{eq:B_120}
\end{equation}
where we used \cref{ePsitt} in the last line.
Since $(X_{u,v},X_{u,j_{q}})$ is independent of $(S^{(u)},\pi)$, we have
\begin{equation}
    \begin{aligned}
       \MoveEqLeft \E \bigl\{ |X_{u,
    v}|e^{t \lvert X_{u,j_{q}}\rvert} \Psi_{\beta,t}(S^{(u)}) \1 (\Gamma_{1})\bigm\vert \Gamma_{u,v}
    \bigr\}\\&\leq \IE \{ |X_{u,
    v}|e^{t \lvert X_{u,j_{q}}\rvert}\}\E \bigl\{  \Psi_{\beta,t}(S^{(u)}) \1 (\Gamma_{1})\bigm\vert \Gamma_{u,v}
    \bigr\}
             \\&\leq C b^{1/8} t \alpha_{n}^{-1} \E \bigl\{  \Psi_{\beta,t}(S^{(u)}) \1 (\Gamma_{1})\bigm\vert \Gamma_{u,v}
    \bigr\}. 
    \end{aligned}
    \label{eq:B_124}
\end{equation}
By H\"older's inequality, we have
\begin{equation}
    \begin{aligned}
		\MoveEqLeft[1]  \E \bigl\{  \Psi_{\beta,t}(S^{(u)}) \1 (\Gamma_{1})\bigm\vert \Gamma_{u,v}
    \bigr\}\\&\leq \E \bigl\{ e^{t \lvert X_{u,j_{q}}\rvert} \Psi_{\beta,t}(S) \1 (\Gamma_{1})\bigm\vert \Gamma_{u,v}
    \bigr\}
           \\&\leq (\IE \{e^{(1+\varepsilon) \lvert
			   X_{u,j_{q}}\rvert/\varepsilon}\1 (\Gamma_{1})\bigm\vert \Gamma_{u,v}\})^{\varepsilon/(1+\varepsilon)} \Bigl(\E \bigl\{
               \Psi_{\beta,t}^{1+\varepsilon}(S) \1 (\Gamma_{1})\bigm\vert \Gamma_{u,v}
\bigr\}\Bigr)^{1/(1+\varepsilon)}.
    \end{aligned}
    \label{eq:B_126}
\end{equation}
  By the property of conditional expectation and the fact that
$\mathbf{X}$ is independent
of $\sigma$, we have the right hand side of \cref{eq:B_126} is equal to
\begin{equation}
	\begin{aligned}
		\MoveEqLeft \Bigl(\IE \{e^{(1+\varepsilon) \lvert
					X_{u,j_{q}}\rvert}\}\Bigr)^{\varepsilon/(1+\varepsilon)}\IP
					(\Gamma_{1}\vert B_{2}
		)^{\varepsilon/(1+\varepsilon)} \Bigl(\E \bigl\{
			\Psi_{\beta,t}^{1+\varepsilon}(S) \bigm\vert \Gamma_{1} \cap \Gamma_{u,v}
	\bigr\}\Bigr)^{1/(1+\varepsilon)}\IP (\Gamma_{1}\vert \Gamma_{u,v}
		)^{1/(1+\varepsilon)}
		   \\&\leq C b^{1/64}  \E \bigl\{
			   \Psi_{\beta,t}^{1+\varepsilon}(S) \bigm\vert \Gamma_{1}  \cap \Gamma_{u,v}
	\bigr\}
			 \IP( \Gamma_{1}\vert \Gamma_{u,v}
),
	\end{aligned}
	\label{eq:101}
\end{equation}
where the inequality follows from \cref{eq-lB1-22,eq-lB1-13} and the fact that $\Psi_{\beta,t}\geq 1$.
By the property of conditional expectation,
\begin{equation}
	\begin{aligned}
	\MoveEqLeft	\E \bigl\{
		\Psi_{\beta,t}^{1+\varepsilon}(S) \bigm\vert \Gamma_{1} \cap  \Gamma_{u,v}
	\bigr\}\IP( \Gamma_{1}\vert
	\Gamma_{u,v})=\E \bigl\{
		\Psi_{\beta,t}(S) \1 (\Gamma_{1})\bigm\vert \Gamma_{u,v} 
\bigr\}.
	\end{aligned}
	\label{eq:100}
\end{equation}
Combining \cref{eq:B_124,eq:B_126,eq:100,eq:101,eq:B_120}, we have 
\begin{equation}
    \begin{aligned}
         \MoveEqLeft\E \{ |X_{\sigma^{-1}(j_{q}),
             \sigma(i_{p})}| \Psi_{\beta,t}'(S) \1 (\Gamma_{1})
             \}\leq  C b^{1/4} t \alpha_{n}^{-1}\E \{  \Psi_{\beta,t}(S) \1 (\Gamma_{1})
     \}.
 \end{aligned}
    \label{eq:B_125}
\end{equation}
For the expectation term on the right hand side of \cref{eq:B_125}, 
\begin{equation}
    \begin{aligned}
    \MoveEqLeft   \E \{  \Psi_{\beta,t}(S) \1 (\Gamma_{1})
     \} \\
     &= \sum^{ }_{v_{1},v_{2}\in [n-m]}  \E \{  \Psi_{\beta,t}(S) \1 (\Gamma_{1})
     \vert \sigma(i_1) = v_1, \sigma(i_2) = v_2\}\\
                  &\quad\quad\times \IP(\sigma(i_1) = v_1, \sigma(i_2) = v_2)\\
                  &= \sum^{ }_{v_{1},v_{2}\in [n-m]}\1 (v_{1} = j_1 \text{ or } v_{2} = j_2)  \E \{  \Psi_{\beta,t}(S) 
     \vert \sigma(i_1) = v_1, \sigma(i_2) = v_2\}\\
                  &\quad\quad\times \IP(\sigma(i_1) = v_1, \sigma(i_2) = v_2).
 \end{aligned}
 \label{eq:B_127}
\end{equation}
For $\mathbf{i}=(i_{1},i_{2})$, $\mathbf{v}=(v_{1},v_{2})$,
 let $\sigma_{\mathbf{i},\mathbf{v}}=\mathcal{P}_{\mathbf{i},\mathbf{v}} \sigma$ and
$S_{\mathbf{i},\mathbf{v}} = \sum^{n-m }_{r=1} X_{r,\sigma_{\mathbf{i},\mathbf{v}}(r)}$.
By (3.14) of \cite{chen2015error} (see also Lemma 4.5 of \cite{chen2010normal}), we have
\begin{equation}
    \begin{aligned}
        \E \{  \Psi_{\beta,t}(S) 
        \vert \sigma(i_1) = v_1, \sigma(i_2) = v_2\}=\IE \{\Psi_{\beta,t}(S_{\mathbf{i},\mathbf{v}})\}.
    \end{aligned}
    \label{eq:B_128}
\end{equation}
Moreover, by the construction of $S_{\mathbf{i},\mathbf{v}}$, it follows that
\begin{equation}
    \begin{aligned}
        \lvert S-S_{\mathbf{i},\mathbf{v}}\rvert
		& \leq |X_{i_{1},v_{1}}| + |X_{i_{2},v_{2}}|+ \sum^{}_{i\in \{i_{1},i_{2}\}} \sum^{}_{v\in \{v_{1},v_{2}\}}|X_{\sigma^{-1}(v),\sigma(i)}| \\
		& \quad + \sum^{ }_{i\in \{i_{1},i_{2}\}} \lvert
        X_{i,\sigma(i)}\rvert + \sum^{ }_{v\in \{v_{1},v_{2}\}} \lvert X_{\sigma^{-1}(v),v}\rvert.  
    \end{aligned}
    \label{eq:B_129}
\end{equation}
By \cref{eq:3_15_2} and H\"older's inequality we have 
\begin{equation}
    \begin{aligned}
        \IE \{\Psi_{\beta,t}(S_{\mathbf{i},\mathbf{v}})\}
		& \leq \IE \{e^{t \lvert
                S-S_{\mathbf{i},\mathbf{v}}\rvert} \Psi_{\beta,t}(S)\}\\
		& \leq (\IE \{e^{(1+\varepsilon)t \lvert
        S-S_{\mathbf{i},\mathbf{v}}\rvert/\varepsilon}\})^{\varepsilon/(1+\varepsilon)}\IE \{
    \Psi_{\beta,t}(S)^{1+\varepsilon}\}.
    \end{aligned}
    \label{eq:B_130}
\end{equation}
By the similar argument to \cref{eq-lB1-22,eq-lB1-13} again, we obtain
\begin{equation}
    \begin{aligned}
      \IE \{\Psi_{\beta,t}(S_{\mathbf{i},\mathbf{v}})\}\leq 
      C b^{1/2} \IE \{\Psi_{\beta,t}(S)\}.
  \end{aligned}
    \label{eq:B_131}
\end{equation}
Combining \cref{eq:B_127,eq:B_128,eq:B_130}, we obtain
\begin{equation}
    \begin{aligned}
         \MoveEqLeft\E \{  \Psi_{\beta,t}(S) \1 (\Gamma_{1})
             \}\\
             &\leq C b^{1/2} \E \{  \Psi_{\beta,t}(S) \}\sum^{ }_{v_{1},v_{2}\in [n-m]}\1 (v_{1} = j_1 \text{ or } v_{2} = j_2)   
        \IP(\sigma(i_1) = v_1, \sigma(i_2) = v_2)\\&\leq C b^{1/2} n^{-1} \IE \{\Psi_{\beta,t}(S)\}.
    \end{aligned}
    \label{eq:B_132}
\end{equation}
By \cref{eq:B_125,eq:B_132}, we complete the proof.
\end{proof}

	The following lemma, whose proof is based on \cref{lemB5,lem:B_6}, plays an important role in the proof of \cref{lemB_3}. 
	\begin{lemma}
		\label{lemB.9}
		Let $\pi$, $\mathbf{X}$ and $W$ be defined as in \cref{thm:11_24_1}, and recall that $\mathcal{P}_{\mathbf{i},\mathbf{j}}$ is defined as in \cref{eq-Pcalpi}. 
		For $\mathbf{i},\mathbf{j} \in [n]_2$, $\mathbf{i}' \in [n]_2^{(\mathbf{i})}$ and $\mathbf{j}' \in [n]_2^{(\mathbf{j})}$, let $\mathcal{I} = (\mathbf{i},\mathbf{i}')$, $\mathcal{J} = (\mathbf{j},\mathbf{j}')$, and 
		\begin{align*}
			\pi_{\mathbf{i},\mathbf{j}} = \mathcal{P}_{\mathbf{i},\mathbf{j}} \pi, \quad \pi_{\mathcal{I},\mathcal{J}} = \mathcal{P}_{\mathbf{i}',\mathbf{j}'} \pi_{\mathbf{i},\mathbf{j}}, \quad W_{(\mathcal{I},\mathcal{J})}
			^{(\mathcal{I})} = \sum_{i' \in [n] \setminus \{ i_1,i_2,i_1',i_2' \}} X_{i', \pi_{\mathcal{I},\mathcal{J}}(i')}.
		\end{align*}
		Then, 
		\begin{equ}
			|\E \{ \Psi_{\beta,t}(W_{\mathcal{I},\mathcal{J}}^{(\mathcal{I})}) \} - h(t)| 
		& \leq C b^2 (n^{-1}+\alpha_n^{-2}) (1 + t^2) h(t).
            \label{eq-B7-a}
		\end{equ}
	\end{lemma}

	\begin{proof}
	Recall that $h(t) = \E \Psi_{\beta,t}(W)$. To bound the difference between $\E \Psi_{\beta,t}(W_{\mathcal{I},\mathcal{J}})$ and $\E \Psi_{\beta,t}(W)$, we consider the following three steps. In the first step, we construct an auxiliary random variable
	$S_{\mathbf{i}',\mathbf{j}'}^{(\mathbf{i}, \mathbf{i}')}$ that is
	close to $W$ and has the same
	distribution as $W_{\mathcal{I},\mathcal{J}}^{(\mathcal{I})}$. In the rest, 
	we apply Taylor's expansion to calculate the difference of the expectations. 

	\medskip 
	{\it Step 1. Constructing $S_{\mathbf{i},\mathbf{j}}^{(\mathbf{i},\mathbf{i}')}$}.
	Note that $\pi$ is a random permutation chosen uniformly from $\mathcal{S}_{n}$, 
	and it follows from Eq. (3.14) of \cite{chen2015error} (see also Lemma 4.5 of \cite{chen2010normal}) that 
	\begin{equation}
		\law ( \pi_{\mathbf{i},\mathbf{j}} ) = \law (\pi | \pi(\mathbf{i}) = \mathbf{j}).
		\label{eq-piij7}
	\end{equation}
	Write $W_{\mathbf{i},\mathbf{j}} = \sum_{i' \in [n]} X_{i', \pi_{\mathbf{i},\mathbf{j}}(i')}$ and $W_{\mathbf{i},\mathbf{j}}^{(\mathbf{i})} = \sum_{i' \not\in A(\mathbf{i})} X_{i',\pi_{\mathbf{i},\mathbf{j}}(i')}$, and it follows from \cref{eq-piij7} that $\law(W_{\mathbf{i},\mathbf{j}}) = \law (W \vert
	\pi(\mathbf{i}) = \mathbf{j})$ and $\law(W_{\mathbf{i},\mathbf{j}}^{(\mathbf{i})}) = \law (W^{(\mathbf{i})} \vert \pi(\mathbf{i}) = \mathbf{j})$. To calculate $\E \Psi_{\beta,t}(W_{\mathbf{i},\mathbf{j}})$, we introduce an auxiliary permutation $\sigma$ as follows.
	Let $\sigma$ be a uniform permutation from $[n] \setminus \{ i_1,i_2 \}$ to $[n] \setminus \{ j_1,j_2 \}$, independent of
		everything
		else, and let 
		\begin{equation}
			\label{eq-S7i}
			S^{(\mathbf{i})} = \sum_{i' \not\in A(\mathbf{i})} X_{i', \sigma(i')}. 
		\end{equation}
		It also follows from Lemma 4.5 of \cite{chen2010normal} that 
		\begin{equation}
			\label{eq-S7ii}
			\law(W_{\mathbf{i},\mathbf{j}}^{(\mathbf{i})}) = \law (S^{(\mathbf{i})}).
		\end{equation}
		Moreover, noting that $\{ i_1,i_2 \} \cap \{ i_1',i_2' \} = \emptyset$, using 
		\cref{eq-piij7} twice implies that
		\begin{equ}
			\law ( \pi_{\mathcal{I},\mathcal{J}} ) = \law (\pi | \pi(\mathcal{I}) = \mathcal{J}).
            \label{eq-piIJJ}
		\end{equ}
		Recalling \cref{eq-Pcalpi}, we define 
		\begin{align*}
			\sigma_{\mathbf{i}',\mathbf{j}'} = \mathcal{P}_{\mathbf{i}',\mathbf{j}'} \sigma, \quad S_{\mathbf{i}',\mathbf{j}'}^{(\mathbf{i})} = \sum_{i' \in [n] \setminus \{ \mathbf{i} \}} X_{i', \sigma_{\mathbf{i}',\mathbf{j}'}(i')}, \quad S_{\mathbf{i}',\mathbf{j}'}
			^{(\mathbf{i},\mathbf{i}')} = 
			\sum_{i' \in [n] \setminus ( \{ \mathbf{i} \} \cup \{ \mathbf{i}' \} )} X_{i', \sigma_{\mathbf{i}',\mathbf{j}'}(i')}.
		\end{align*}
		Then, it follows by definition that $\law(W_{\mathcal{I},\mathcal{J}}^{(\mathcal{I})}) = \law( S_{\mathbf{i}',\mathbf{j}'}^{(\mathbf{i},\mathbf{i}')} )$.

		\medskip 
		{\it Step 2. Bounding $ \lvert \E \Psi_{\beta,t}(S_{\mathbf{i}',\mathbf{j}'}^{(\mathbf{i},\mathbf{i}')}) - h(t) \rvert $.} We first bound $\lvert \E \Psi_{\beta,t}(S_{\mathbf{i}',\mathbf{j}'}^{(\mathbf{i},\mathbf{i}')}) - \E \Psi_{\beta,t}
		(S^{(\mathbf{i})}) \rvert$, and the bound of $ \lvert \E \Psi_{\beta,t}(S^{(\mathbf{i})}) - h(t)\rvert $ can be obtained similarly. 
By Taylor's expansion, 
\begin{equ}
	\Psi_{\beta,t}(w)
	& = \Psi_{\beta,t}(w_0) + (w - w_0) \Psi_{\beta,t}'(w_0) + \frac{1}{2} (w - w_0)^2 \E \{
	\Psi_{\beta,t}''(w_0 + U (w - w_0))(1-U)\},
    \label{eq-Tay7}
\end{equ}
where $U$ is a uniform random variable over the interval $[0,1]$ independent of all others. 
	Applying Taylor's expansion \cref{eq-Tay7} with $w_0 = S^{(\mathbf{i})}$ and $w = S_{\mathbf{i}',\mathbf{j}'}^{(\mathbf{i},\mathbf{i}')}$, we have  
	\begin{equ}
		\E \{ \Psi_{\beta,t}(S_{\mathbf{i}',\mathbf{j}'}^{(\mathbf{i},\mathbf{i}')}) \}
		& = \E \{ \Psi_{\beta,t}(S^{(\mathbf{i})}) \} + \E \bigl\{ \bigl(S_{\mathbf{i}',\mathbf{j}'}^{(\mathbf{i},\mathbf{i}')} - S^{(\mathbf{i})}\bigr) \Psi_{\beta,t}'(S^{(\mathbf{i})}) \bigr\} \\
		& \quad + \E \bigl\{ \bigl(S_{\mathbf{i}',\mathbf{j}'}^{(\mathbf{i},\mathbf{i}')} -
		S^{(\mathbf{i})}\bigr)^2 \Psi_{\beta,t}''\bigl( S^{(\mathbf{i})} + U
(S_{\mathbf{i}',\mathbf{j}'}^{(\mathbf{i},\mathbf{i}')} - S^{(\mathbf{i})})\bigr)(1-U)  \bigr\}.
        \label{eq-B_91}
	\end{equ}
Denote by $B_{\mathbf{i},\mathbf{j}}$ the event that $\{ \sigma(i_1) = j_1
	    \text{ or } \sigma(i_2') = j_2'\}$ and denote by $B_{\mathbf{i},\mathbf{j}}^c$ the complement of $B_{\mathbf{i},\mathbf{j}}$.
	For the second term of \cref{eq-B_91}, by the construction of $\mathcal{P}_{\mathbf{i}',\mathbf{j}'}$, we have 
	\begin{align*}
		\begin{split}
			\MoveEqLeft
		|\E \{ (S_{\mathbf{i}',\mathbf{j}'}^{(\mathbf{i},\mathbf{i}')} - S^{(\mathbf{i})}) \Psi_{\beta,t}'(S^{(\mathbf{i})}) \}|\\
		& \leq |\E \{ X_{\sigma^{-1}(j_1'),\sigma(i_1')} \Psi_{\beta,t}'(S^{(\mathbf{i})}) \1
        (B_{\mathbf{i}',\mathbf{j}'})\}|\\
        & \quad +  |\E \{ X_{\sigma^{-1}(j_2'),\sigma(i_2')} \Psi_{\beta,t}'(S^{(\mathbf{i})})
            \1
        (B_{\mathbf{i}',\mathbf{j}'})\}|\\
		& \quad +  |\E \{ X_{\sigma^{-1}(j_1'),\sigma(i_2')} \Psi_{\beta,t}'(S^{(\mathbf{i})}) \1
        (B_{\mathbf{i}',\mathbf{j}'}^{c}) \}|\\
		& \quad +  |\E \{ X_{\sigma^{-1}(j_2'),\sigma(i_1')} \Psi_{\beta,t}'(S^{(\mathbf{i})}) \1
        (B_{\mathbf{i}',\mathbf{j}'}^{c}) \}|\\
		& \quad + \sum_{i' \in \{ i_1', i_2' \}}\lvert \E \{ X_{i', \sigma(i')} \Psi_{\beta,t}'(S^{(\mathbf{i})}) \} \rvert + \sum_{j' \in \{ j_1', j_2' \}}\lvert \E \{ X_{\sigma^{-1}(j'), j'} \Psi_{\beta,t}'(S^{(\mathbf{i})}) \} \rvert\\
		& \leq \sum_{i' \in \{ i_1',i_2' \}} \sum_{j' \in \{ j_1',j_2' \}} \E \{ |X_{\sigma^{-1}(j'), \sigma(i')}| \Psi_{\beta,t}'(S^{(\mathbf{i})}) \1 (B_{\mathbf{i}',\mathbf{j}'})  \}\\
        & \quad + |\E \{ X_{\sigma^{-1}(j_1'),\sigma(i_2')} \Psi_{\beta,t}'(S^{(\mathbf{i})})\}| +
        |\E \{ X_{\sigma^{-1}(j_2'),\sigma(i_1')} \Psi_{\beta,t}'(S^{(\mathbf{i})})\}|\\
		& \quad + \sum_{i' \in \{ i_1', i_2' \}}\lvert \E \{ X_{i', \sigma(i')} \Psi_{\beta,t}'(S^{(\mathbf{i})}) \} \rvert + \sum_{j' \in \{ j_1', j_2' \}}\lvert \E \{ X_{\sigma^{-1}(j'), j'} \Psi_{\beta,t}'(S^{(\mathbf{i})}) \} \rvert.
		\end{split}
	\end{align*}
	Applying \cref{lemB5,lem:B_6} with $S = S^{(\mathbf{i})}$ under a relabeling of indices, we obtain 
	\begin{equ}
		|\E \{ (S_{\mathbf{i}',\mathbf{j}'}^{(\mathbf{i},\mathbf{i}')} - S^{(\mathbf{i})}) \Psi_{\beta,t}'(S^{(\mathbf{i})}) \}|
		& \leq C (n^{-1} \alpha_{n}^{-1} + \alpha_{n}^{-2}) b (1 + t^2) \E \Psi_{\beta,t}(S^{(\mathbf{i})}). 
        \label{eq-B91'}
	\end{equ}

	For the third term of \cref{eq-B_91}, we have 
	\begin{align}
		\label{eq-B93}
		\begin{split}
			\MoveEqLeft
			\bigl\lvert \E \bigl\{ (S_{\mathbf{i}',\mathbf{j}'}^{(\mathbf{i},\mathbf{i}')} -
			S^{(\mathbf{i})})^2 \Psi_{\beta,t}''\bigl( S^{(\mathbf{i})} + U
        (S_{\mathbf{i}',\mathbf{j}'}^{(\mathbf{i},\mathbf{i}')} -
S^{(\mathbf{i})})\bigr)(1-U)  \bigr\} \bigr\rvert\\
			& \leq t^2 \E \bigl\{ (S_{\mathbf{i}',\mathbf{j}'}^{(\mathbf{i},\mathbf{i}')} - S^{(\mathbf{i})})^2 e^{t \lvert (S_{\mathbf{i}',\mathbf{j}'}^{(\mathbf{i},\mathbf{i}')} - S^{(\mathbf{i})}) \rvert} \Psi_{\beta,t}( S^{(\mathbf{i})} )  \bigr\} \\
			& \leq 24 t^2 \biggl\{\sum_{k \in \{ i_1,i_2,i_1',i_2' \}} \sum_{l \in \{ j_1,j_2,j_1',j_2' \}} \E \bigl\{ |X_{\sigma^{-1}(l), \sigma(k)}|^2 e^{24 t \lvert X_{\sigma^{-1}(l), \sigma(k)} \rvert} \Psi_{\beta,t}(S^{(\mathbf{i})}) \bigr\}  \\
			& \hspace{2cm}+ \sum_{i \in \{ i_1,i_2,i_1',i_2' \}} \E \bigl\{ X_{i, \sigma(i)}^2 e^{24 t \lvert X_{i,\sigma(i)} \rvert} \Psi_{\beta,t}(S^{(\mathbf{i})}) \bigr\}\\
			&\hspace{2cm}  + \sum_{j \in \{ j_1,j_2,j_1',j_2' \}}  \E \bigl\{ X_{\sigma^{-1}(j), j}^2 e^{24 t \lvert X_{i,\sigma(i)} \rvert} \Psi_{\beta,t}(S^{(\mathbf{i})}) \bigr\}\biggr\}.
		\end{split}
	\end{align}
	Applying \cref{lem-B1} with $S = S^{(\mathbf{i})}$ and $\zeta_{i,j} = \lvert X_{i,j} \rvert^2 e^{24 t \lvert X_{i,j} \rvert}$ under a relabeling of indices, and recalling that $t \leq \alpha_n/64$, we obtain 
	\begin{align*}
		\begin{split}
			\MoveEqLeft
		\E \bigl\{ |X_{\sigma^{-1}(l), \sigma(k)}|^2 e^{24 t \lvert X_{\sigma^{-1}(l), \sigma(k)} \rvert} \Psi_{\beta,t}(S^{(\mathbf{i})}) \bigr\} \\
		& \leq 4 b^{1/8} \E \Psi_{\beta,t}(S^{(\mathbf{i})})\max_{u,v \in [n]} \E \bigl\{ |X_{u,v}|^2 e^{24 t \lvert X_{u, v} \rvert + t ( \lvert X_{k, v} \rvert
		+ \lvert X_{u,l} \rvert) } \bigr\} \\
		& \leq C b^{1/8} \alpha_{n}^{-2} \E \Psi_{\beta,t}(S^{(\mathbf{i})}) \max_{u,v \in [n]} \E \bigl\{ ( 1 + \lvert \alpha_{n} X_{u,v} \rvert^2  ) e^{13\alpha_n \lvert X_{u,v} \rvert/32} \bigr\}\\
		& \leq C b^{1/8} \alpha_{n}^{-2} \E \Psi_{\beta,t}(S^{(\mathbf{i})}) \max_{u,v \in [n]} \E \bigl\{  e^{\alpha_n \lvert X_{u,v} \rvert/2} \bigr\}\\
		& \leq C b \alpha_n^{-2} \E \Psi_{\beta,t}(S^{(\mathbf{i})}), 
		\end{split}
	\end{align*}
	where we used \cref{eq:11_24_6} in the last line. 
	Similarly, we obtain 
	\begin{align*}
		\E \bigl\{ X_{i, \pi(i)}^2 e^{24 t \lvert X_{i,\pi(i)} \rvert} \Psi_{\beta,t}(S^{(\mathbf{i})}) \bigr\} & \leq C b \alpha_n^{-2} \E \Psi_{\beta,t}(S^{(\mathbf{i})}), \\
		\E \bigl\{ X_{\pi^{-1}(j), j}^2 e^{24 t \lvert X_{\pi^{-1}(j), j} \rvert} \Psi_{\beta,t}(S^{(\mathbf{i})}) \bigr\} & \leq C b \alpha_n^{-2} \E \Psi_{\beta,t}(S^{(\mathbf{i})}). 
	\end{align*}
	Hence, we have 
	\begin{align*}
		\text{R.H.S. of \cref{eq-B93}}
		& \leq C b \alpha_n^{-2}(1 + t^2)\E \Psi_{\beta,t}(S^{(\mathbf{i})}).
	\end{align*}
    Together with \cref{eq-B_91,eq-B91'}, we have
	\begin{equ}
		|\E \{ \Psi_{\beta,t}(W_{\mathcal{I},\mathcal{J}}^{(\mathcal{I})}) \} 
        - \E \Psi_{\beta,t}(S^{(\mathbf{i})}) | \leq C b (n^{-1}+\alpha_n^{-2}) (1 + t^2) \E \Psi_{\beta,t}(S^{(\mathbf{i})}). 
        \label{eq-B81}
	\end{equ}
	A similar argument yields 
	\begin{equ}
		|\E \Psi_{\beta,t}(S^{(\mathbf{i})}) - h(t) | \leq C b (n^{-1}+\alpha_n^{-2}) (1 + t^2) h(t).
        \label{eq-B82}
	\end{equ}
  Then applying \cref{lem-B1} to the last expectation of \cref{eq-B81} and 
   similar to \cref{eq-B93}, we
   have 
   \begin{equation}
       \begin{aligned}
       \E \Psi_{\beta,t}(S^{(\mathbf{i})})=\E \Psi_{\beta,t}(W_{\mathbf{i},\mathbf{j}}^{(\mathbf{i})}) \leq \IE \{
              \Psi_{\beta,t}(W)e^{t|W_{\mathbf{i},\mathbf{j}}^{(\mathbf{i})}-W|} \}  \leq C
			 b h(t).
       \end{aligned}
       \label{eq:B_172}
   \end{equation}
	Combining \cref{eq-B81,eq-B82,eq:B_172} and recalling that $t \leq \alpha_n / 64$, we obtain
	\begin{align*}
		|\E \{ \Psi_{\beta,t}(W_{\mathcal{I},\mathcal{J}}^{(\mathcal{I})}) \} - h(t)| 
		& \leq C b^2 (n^{-1}+\alpha_n^{-2}) (1 + t^2) h(t).
	\end{align*}
	This completes the proof.
	\end{proof}

    \subsection{Proofs of \autoref{lemB_3} and \ref{lem:B_4}}%
\label{sub:B3}

First, we apply \cref{lem-B1,lemB.9} to prove \cref{lemB_3}.
\begin{proof}[Proof of \cref{lemB_3}]
Applying the Taylor expansion \cref{eq-Tay7} to $\Psi_{\beta,t}(W)$ yields
\begin{equation}
   \begin{aligned}
	   \MoveEqLeft
		\E \{ \xi_{\mathbf{i},\pi(\mathbf{i})} \xi_{\mathbf{i}', \pi(\mathbf{i}')} \Psi_{\beta,t}(W) \}\\
		& = \frac{1}{(n)_4} \sum_{\mathbf{j} \in [n]_2} \sum_{\mathbf{j}' \in [n]_2^{(\mathbf{j})}} \E \{ \xi_{\mathbf{i},\pi(\mathbf{i})} \xi_{\mathbf{i}', \pi(\mathbf{i}')} \Psi_{\beta,t}(W) \vert \pi(\mathcal{I}) = \mathcal{J}\}\\
		& = Q_1 + Q_2 + Q_3, 
	\end{aligned}
    \label{eq:B_83}
\end{equation}
where
	\begin{align*}
		Q_1 & = \frac{1}{(n)_4} \sum_{\mathbf{j} \in [n]_2} \sum_{\mathbf{j}' \in
		[n]_2^{(\mathbf{j})}} \E \{ \xi_{\mathbf{i},\mathbf{j}} \xi_{\mathbf{i}', \mathbf{j}'}
	\Psi_{\beta,t}(W^{(\mathcal{I})}) \vert \pi(\mathcal{I}) = \mathcal{J}\},\\
		Q_2 & = \frac{1}{(n)_4} \sum_{\mathbf{j} \in [n]_2} \sum_{\mathbf{j}' \in
		[n]_2^{(\mathbf{j})}} \E \{ \xi_{\mathbf{i},\mathbf{j}} \xi_{\mathbf{i}', \mathbf{j}'}
	V_{\mathcal{I},\mathcal{J}}\Psi_{\beta,t}'(W^{(\mathcal{I})}) \vert \pi(\mathcal{I}) =
\mathcal{J}	\},\\
		Q_3 & = \frac{1}{(n)_4} \sum_{\mathbf{j} \in [n]_2} \sum_{\mathbf{j}' \in
		    [n]_2^{(\mathbf{j})}} \E \{ \xi_{\mathbf{i},\mathbf{j}} \xi_{\mathbf{i}', \mathbf{j}'}
		    V_{\mathcal{I},\mathcal{J}}^2\Psi_{\beta,t}''(W^{(\mathcal{I})} + U(W -
            W^{(\mathcal{I})}))(1-U) \vert
		\pi(\mathcal{I}) = \mathcal{J}\},
	\end{align*}
	and where $\mathcal{I} = (i_1,i_2,i_1',i_2')$, $\mathcal{J} = (j_1,j_2,j_1',j_2')$,  $W^{(\mathcal{I})} = \sum_{i \in [n] \setminus A(\mathcal{I})} X_{i \pi(i)}$, $V_{\mathcal{I}, \mathcal{J}} = X_{i_1,j_1} + X_{i_2,j_2} + X_{i_1',j_1'} +
	X_{i_2',j_2'}$ and $U$ is a uniform random variable on $[0,1]$ independent of $\mathbf{X}$ and $\pi$.

	For $Q_1$, as $(\xi_{\mathbf{i},\mathbf{j}}, \xi_{\mathbf{i}',\mathbf{j}'})$ is conditionally independent of $W^{(\mathcal{I})}$ given $\pi(\mathcal{I}) = \mathcal{J}$, and as $(\xi_{\mathbf{i},\mathbf{j}}, \xi_{\mathbf{i}',\mathbf{j}'})$ is
	also independent of $\pi$, we have 
	\begin{align}
		\begin{split}
			\E \{ \xi_{\mathbf{i},\mathbf{j}} \xi_{\mathbf{i}', \mathbf{j}'} \Psi_{\beta,t}(W^{(\mathcal{I})}) \vert \pi(\mathcal{I}) = \mathcal{J}\}
			& = 
            \E \{ \xi_{\mathbf{i},\mathbf{j}} \xi_{\mathbf{i}', \mathbf{j}'}  \} \E \{ \Psi_{\beta,t}(W^{(\mathcal{I})}) \vert \pi(\mathcal{I}) = \mathcal{J}\}.
		\end{split}
        \label{eq:B_20}
	\end{align} 
	Let $W_{\mathcal{I},\mathcal{J}}$ be defined as in \cref{lemB.9}. By \cref{eq-piIJJ}, we have 
	\begin{align*}
		\E \{ \Psi_{\beta,t}(W^{(\mathcal{I})}) \vert \pi(\mathcal{I}) = \mathcal{J}\}
		& =  \E \{ \Psi_{\beta,t}(W_{\mathcal{I},\mathcal{J}}^{(\mathcal{I})}) \}\\
		& = h(t) + \lvert \E\{ \Psi_{\beta,t}(W_{\mathcal{I},\mathcal{J}}^{(\mathcal{I})}) \} - h(t)  \rvert.
	\end{align*}
	Taking average over $\mathbf{j} \in [n]_2, \mathbf{j}' \in [n]_2^{(\mathbf{j})}$ on both sides of \cref{eq:B_20} and applying \cref{lemB.9} gives 
	\begin{equ}
		|Q_1|
		& \leq \frac{h(t)}{(n)_4} \biggl\lvert \sum_{\mathbf{j} \in [n]_2} \sum_{\mathbf{j}' \in [n]_2^{(\mathbf{j})}} \bigl(\E \xi_{\mathbf{i}, \mathbf{j}} \E \xi_{\mathbf{i}',\mathbf{j}'}\bigr) \biggr\rvert  \\
		& \quad + \frac{1}{(n)_4} \sum_{\mathbf{j} \in [n]_2} \sum_{\mathbf{j}' \in [n]_2^{(\mathbf{j})}} \bigl(\E |\xi_{\mathbf{i}, \mathbf{j}}| \E |\xi_{\mathbf{i}',\mathbf{j}'}| \lvert \E\{ \Psi_{\beta,t}(W_{\mathcal{I},\mathcal{J}}^{(\mathcal{I})}) \} - h(t)  \rvert\bigr)\\
		& \leq  \frac{h(t)}{(n)_4} \biggl\lvert \sum_{\mathbf{j} \in [n]_2} \sum_{\mathbf{j}' \in [n]_2^{(\mathbf{j})}} \E \xi_{\mathbf{i}, \mathbf{j}} \E \xi_{\mathbf{i}',\mathbf{j}'} \biggr\rvert \\
		& \quad + {C b^2 n^{-4} (n^{-1}+\alpha_n^{-2}) (1 + t^2) h(t)} \sum_{\mathbf{j} \in [n]_2} \sum_{\mathbf{j}' \in [n]_2^{(\mathbf{j})}} \bigl(\E |\xi_{\mathbf{i}, \mathbf{j}}| \E |\xi_{\mathbf{i}',\mathbf{j}'}| \bigr).
        \label{eq-Q111}
	\end{equ}
	For the first term of the right hand side of \cref{eq-Q111}, since $\E \xi_{\mathbf{i}',\pi(\mathbf{i}')} = 0$, it follows that $\sum_{\mathbf{j}' \in [n]_2} \E \xi_{\mathbf{i}',\mathbf{j}'} = 0$. Therefore, 
	\begin{equ}
		\biggl\lvert \sum_{\mathbf{j}' \in [n]_2^{(\mathbf{j})}} \E \xi_{\mathbf{i}',\mathbf{j}'} \biggr\rvert
		& = \biggl\lvert- \sum_{\mathbf{j}' \in [n]_2 \setminus [n]_2^{(\mathbf{j})}} \E
            \xi_{\mathbf{i}',\mathbf{j}'} \biggr\rvert
            \leq \sum^{ }_{\mathbf{j}'\in
        [n]_{2}} \1(E_{\mathbf{j},\mathbf{j}'})\E \lvert \xi_{\mathbf{i}',\mathbf{j}'} \rvert,
        \label{eq-B_43}
	\end{equ}
	where $E_{\mathbf{j},\mathbf{j}'}= \{A(\mathbf{j})\cap
	A(\mathbf{j}') \neq \emptyset \}$.
   Hence, we have 
	\begin{equ}
		\lvert Q_1 \rvert
        & \leq C b (1 + t^2) n^{-4} h(t)   \sum_{\mathbf{j}, \mathbf{j}' \in [n]_2}
        (\1(E_{\mathbf{j},\mathbf{j}'}) + \alpha_{n}^{-2}+ n^{-1})\bigl(\E \lvert \xi_{\mathbf{i},\mathbf{j}} \rvert \E \lvert \xi_{\mathbf{i}', \mathbf{j}'} \rvert\bigr).
        \label{eq-B_82}
	\end{equ}

    We now consider $Q_{2}$. Since $\xi_{\mathbf{i},\mathbf{j}} \xi_{\mathbf{i}', \mathbf{j}'}
    V_{\mathcal{I},\mathcal{J}}$ is independent of $W^{(\mathcal{I})}$ conditional on $\pi(\mathcal{I})=\mathcal{J}$, we have 
    \begin{equation}
        \begin{aligned}
            Q_2 & = \frac{1}{(n)_4} \sum_{\mathbf{j} \in [n]_2} \sum_{\mathbf{j}' \in
            [n]_2^{(\mathbf{j})}} \E \{ \xi_{\mathbf{i},\mathbf{j}} \xi_{\mathbf{i}', \mathbf{j}'}
            V_{\mathcal{I},\mathcal{J}}\}\IE \{\Psi_{\beta,t}'(W^{(\mathcal{I})}) \vert
        \pi(\mathcal{I})=\mathcal{J}\} \\
& = \frac{1}{(n)_4} \sum_{\mathbf{j} \in [n]_2} \sum_{\mathbf{j}' \in
            [n]_2^{(\mathbf{j})}} \E \{ \xi_{\mathbf{i},\mathbf{j}} \xi_{\mathbf{i}', \mathbf{j}'}
            V_{\mathcal{I},\mathcal{J}}\}\IE
            \{\Psi_{\beta,t}'(W^{(\mathcal{I})}_{\mathcal{I},\mathcal{J}}) \}\\
&= Q_{2,1}+Q_{2,2}+Q_{2,3},
        \end{aligned}
        \label{eq:B_63}
    \end{equation}
   where 
   \begin{align*}
       Q_{2,1}&=  \frac{1}{(n)_4} \sum_{\mathbf{j} \in [n]_2} \sum_{\mathbf{j}' \in
            [n]_2^{(\mathbf{j})}} \E \{ \xi_{\mathbf{i},\mathbf{j}} \xi_{\mathbf{i}', \mathbf{j}'}
        V_{\mathbf{i},\mathbf{j}} \}\IE
            \{\Psi_{\beta,t}'(W) \},\\ 
            Q_{2,2}&=  \frac{1}{(n)_4} \sum_{\mathbf{j} \in [n]_2} \sum_{\mathbf{j}' \in
            [n]_2^{(\mathbf{j})}} \E \{ \xi_{\mathbf{i},\mathbf{j}} \xi_{\mathbf{i}', \mathbf{j}'}
        V_{\mathbf{i}',\mathbf{j}'} \}\IE
            \{\Psi_{\beta,t}'(W) \},\\
                Q_{2,3}&= \frac{1}{(n)_4} \sum_{\mathbf{j} \in [n]_2} \sum_{\mathbf{j}' \in
            [n]_2^{(\mathbf{j})}} \E \{ \xi_{\mathbf{i},\mathbf{j}} \xi_{\mathbf{i}', \mathbf{j}'}
            V_{\mathcal{I},\mathcal{J}}\}\IE
            \{\Psi_{\beta,t}''(W+U(W^{(\mathcal{I})}_{\mathcal{I},\mathcal{J}}-W))
            (W^{(\mathcal{I})}_{\mathcal{I},\mathcal{J}}-W)\},
       \label{eq:B_64}
   \end{align*}
   and $V_{\mathbf{i},\mathbf{j}}= X_{i_{1},j_{1}}+X_{i_{2},j_{2}}$, $V_{\mathbf{i}',\mathbf{j}'}=
   X_{i_{1}',j_{1}'}+X_{i_{2}',j_{2}'}$.
   Similar to \cref{eq-B_43}, we have 
   \begin{equation}
       \begin{aligned}
           |Q_{2,1}|&\leq C n^{-4} t h(t) \sum_{\mathbf{j},\mathbf{j}' \in [n]_2}\1(E_{\mathbf{j},\mathbf{j}'})
           \vert \E \{ \xi_{\mathbf{i},\mathbf{j}} V_{\mathbf{i},\mathbf{j}}\}\IE\{
             \xi_{\mathbf{i}',\mathbf{j}'}\}\vert,\\
		   |Q_{2,2}|&\leq C n^{-4} t h(t) \sum_{\mathbf{j},\mathbf{j}' \in [n]_2} \1(E_{\mathbf{j},\mathbf{j}'})
		   \vert \E \{ \xi_{\mathbf{i},\mathbf{j}} \}\IE\{
             \xi_{\mathbf{i}',\mathbf{j}'}V_{\mathbf{i}',\mathbf{j}'}\}\vert.
       \end{aligned}
       \label{eq:B_65} 
   \end{equation}
    We next consider $Q_{2,3}$.  By \cref{ePsiww} and using a similar argument for \cref{eq-B93},
   \begin{equation}
	   \begin{split}
		   \MoveEqLeft
	   	 \lvert \IE \{\Psi_{\beta,t}''(W+U(W^{(\mathcal{I})}_{\mathcal{I},\mathcal{J}}-W))
		 (W^{(\mathcal{I})}_{\mathcal{I},\mathcal{J}}-W)\} \rvert \\
		& \leq t^{2} \IE \{\Psi_{\beta,t}(W)e^{t (W^{(\mathcal{I})}_{\mathcal{I},\mathcal{J}}-W)}
		  \lvert W^{(\mathcal{I})}_{\mathcal{I},\mathcal{J}}-W\rvert\}\\
		& \leq C b \alpha_{n}^{-1} (1 + t^2) h(t).
	   \end{split}
       \label{eq:B_67}
   \end{equation}
   Therefore, 
   \begin{align}
		   \lvert Q_{2,3}\rvert
           & \leq  C b (1 + t^{2}) h(t) n^{-4} \alpha_{n}^{-1}  
		   \Bigl(  \sum_{\mathbf{j}, \mathbf{j}' \in [n]_2} 
            \vert \E \{ \xi_{\mathbf{i},\mathbf{j}} V_{\mathbf{i},\mathbf{j}}\}\IE\{
            \xi_{\mathbf{i}',\mathbf{j}'}\}\vert  + \sum_{\mathbf{j}, \mathbf{j}' \in [n]_2} 
            \vert \E \{ \xi_{\mathbf{i},\mathbf{j}} \}\IE\{
             \xi_{\mathbf{i}',\mathbf{j}'}V_{\mathbf{i}',\mathbf{j}'}\}\vert \Bigr)\nonumber. \\
       \label{eq:B_70}
   \end{align}
   By \cref{eq:B_63,eq:B_65,eq:B_70}, we have
   \begin{equation}
       \begin{aligned}
           \lvert Q_{2}\rvert&\leq C b (1 + t^{2}) h(t) n^{-4}\sum_{\mathbf{j}, \mathbf{j}' \in
           [n]_2} \bigl(\alpha_{n}^{-1}+\1(E_{\mathbf{j},\mathbf{j}'})\bigr)
                           \bigl( \lvert\E \{ \xi_{\mathbf{i},\mathbf{j}} V_{\mathbf{i},\mathbf{j}}\}\IE\{
         \xi_{\mathbf{i}',\mathbf{j}'}\}\rvert+\lvert\E \{ \xi_{\mathbf{i}',\mathbf{j}'}
         V_{\mathbf{i}',\mathbf{j}'}\}\IE\{
     \xi_{\mathbf{i},\mathbf{j}}\}\rvert\bigr)\\
&\leq C b (1 + t^{2}) h(t) n^{-4}\sum_{\mathbf{j}, \mathbf{j}' \in
           [n]_2}
                           \bigl( \E \{ \lvert\xi_{\mathbf{i},\mathbf{j}}\rvert
                                   T_{\mathbf{i},\mathbf{j}}^{2}\}\IE\{
                                   \lvert\xi_{\mathbf{i}',\mathbf{j}'}\rvert\}+\E \{\lvert
                                   \xi_{\mathbf{i}',\mathbf{j}'}\rvert
         T_{\mathbf{i}',\mathbf{j}'}^{2}\}\IE\{
     \lvert\xi_{\mathbf{i},\mathbf{j}}\rvert \}\bigr)\\
&\quad+C b (1 + t^{2}) h(t) n^{-4}\sum_{\mathbf{j}, \mathbf{j}' \in
           [n]_2} \bigl(\alpha_{n}^{-2}+\1(E_{\mathbf{j},\mathbf{j}'})\bigr)
                            \E \{ \lvert\xi_{\mathbf{i},\mathbf{j}}\rvert
                                   \}\IE\{ \lvert\xi_{\mathbf{i}',\mathbf{j}'}\rvert\},
       \end{aligned}
       \label{eq:B_73}
   \end{equation}
   where the second inequality follows from Cauchy's inequality and the fact that
   $|V_{\mathbf{i},\mathbf{j}}|\leq T_{\mathbf{i},\mathbf{j}}$ for any
   $\mathbf{i},\mathbf{j}$.
   Finally, for $Q_{3}$, by \cref{ePsiww} again, and noting that $V_{\mathcal{I},\mathcal{J}} = V_{\mathbf{i},\mathbf{j}} + V_{\mathbf{i}', \mathbf{j}'}$, we have 
   \begin{equation}
       \begin{aligned}
          \MoveEqLeft \lvert \E \{ \xi_{\mathbf{i},\mathbf{j}} \xi_{\mathbf{i}', \mathbf{j}'} V_{\mathcal{I},\mathcal{J}}^2 
              \Psi_{\beta,t}''(W^{(\mathcal{I})} + U(W - W^{(\mathcal{I})}))(1-U) \vert \pi(\mathbf{i}) =
              \mathbf{j}; \pi(\mathbf{i}') = \mathbf{j}'\}\\&\leq t^{2}\E \{
              \lvert\xi_{\mathbf{i},\mathbf{j}} \xi_{\mathbf{i}', \mathbf{j}'}\rvert 
			  V_{\mathcal{I},\mathcal{J}}^2 \Psi_{\beta,t}(W^{(\mathcal{I})})
               e^{t\lvert V_{\mathcal{I},\mathcal{J}}\rvert} \vert \pi(\mathbf{i}) =
           \mathbf{j}; \pi(\mathbf{i}') = \mathbf{j}'\}\rvert 
              \\&= t^{2}\E \{
              \lvert\xi_{\mathbf{i},\mathbf{j}} \xi_{\mathbf{i}', \mathbf{j}'}\rvert 
           V_{\mathcal{I},\mathcal{J}}^2 e^{t\lvert V_{\mathcal{I},\mathcal{J}}\rvert}\}\IE \{
           \Psi_{\beta,t}(W^{(\mathcal{I})})
                \vert \pi(\mathbf{i}) =
		\mathbf{j}; \pi(\mathbf{i}') = \mathbf{j}'\}
            \\&= t^{2}\E \{
              \lvert\xi_{\mathbf{i},\mathbf{j}} \xi_{\mathbf{i}', \mathbf{j}'}\rvert 
           V_{\mathcal{I},\mathcal{J}}^2 e^{t\lvert V_{\mathcal{I},\mathcal{J}}\rvert}\}\IE \{
           \Psi_{\beta,t}(W^{(\mathcal{I})}_{\mathcal{I},\mathcal{J}}) \}
\\&\leq  2 t^{2}\E \{
              \lvert\xi_{\mathbf{i},\mathbf{j}} \xi_{\mathbf{i}', \mathbf{j}'}\rvert 
			  (T_{\mathbf{i},\mathbf{j}}^2 e^{2 t\lvert T_{\mathbf{i},\mathbf{j}}\rvert} + T_{\mathbf{i}',\mathbf{j}'}^2 e^{2 t\lvert T_{\mathbf{i}',\mathbf{j}'}\rvert} )\}\IE \{
       \Psi_{\beta,t}(W)e^{t|W^{(\mathcal{I})}_{\mathcal{I},\mathcal{J}}-W|} \},
       \end{aligned}
       \label{eq:B_71}
   \end{equation}
   where we use  $|V_{\mathbf{i},\mathbf{j}}|\leq T_{\mathbf{i},\mathbf{j}}$ for any
   $\mathbf{i},\mathbf{j}$.
   Then applying \cref{lem-B1} to the second expectation in the last line of \cref{eq:B_71} and 
   similar to \cref{eq-B93}, we
   have 
   \begin{equation}
       \begin{aligned}
         \IE \{
             \Psi_{\beta,t}(W)e^{t|W^{(\mathcal{I})}_{\mathcal{I},\mathcal{J}}-W|} \}  \leq C
			 b h(t).
       \end{aligned}
       \label{eq:B_72}
   \end{equation}
   By \cref{eq:B_71,eq:B_72}, we have
\begin{equation}
    \begin{aligned}
        \lvert Q_{3}\rvert \leq C b n^{-4} t^{2} h(t) \sum^{ }_{\mathbf{j},\mathbf{j}'\in
            [n]_{2}}  \bigl( \E \{
              \lvert\xi_{\mathbf{i},\mathbf{j}} \rvert 
		  T_{\mathbf{i},\mathbf{j}}^2 e^{2 t\lvert T_{\mathbf{i},\mathbf{j}}\rvert}\} \E
	  |\xi_{\mathbf{i}', \mathbf{j}'}|+ \E \{
              \lvert\xi_{\mathbf{i}',\mathbf{j}'} \rvert 
		  T_{\mathbf{i}',\mathbf{j}'}^2 e^{2 t\lvert T_{\mathbf{i}',\mathbf{j}'}\rvert}\} \E |\xi_{\mathbf{i}, \mathbf{j}}| \bigr).
    \end{aligned}
    \label{eq:B_74}
\end{equation}
Combining \cref{eq:B_83,eq-B_82,eq:B_73,eq:B_74} yields the desired result.
      \end{proof}  
	  We finish our paper by proving \cref{lem:B_4}, which is based on \cref{lem-B1,lemB_3}. 
\begin{proof}[Proof of \cref{lem:B_4}]
First, for $v=0$, 
expanding the square term in \cref{eq:B.84} yields 
\begin{equation}
    \begin{aligned}
        \IE \Bigl\{\Bigl(\sum^{
        }_{\mathbf{i}\in [n]_{2}} \bar{g}_{\mathbf{i},\pi(\mathbf{i})}(u)\Bigr)^{2} \Psi_{\beta,t}(W)\Bigr\}= H_{1}(u) + H_{2}(u),
    \end{aligned}
    \label{eq:11_27_13}
\end{equation}
where 
\begin{equation*}
    \begin{aligned}
          H_{1}(u) & = \frac{1}{16n^{2}} \sum^{ }_{\mathbf{i}\in [n]_{2}} \sum^{ }_{\mathbf{i}'\in
        [n]_{2}\setminus[n]_{2}^{(\mathbf{i})}}  \IE
        \{\bar g_{\mathbf{i},\pi(\mathbf{i})}(u) \bar g_{\mathbf{i}',\pi(\mathbf{i}')}(u)
        \Psi_{\beta,t}(W)\},\\
        H_{2}(u) & =\frac{1}{16n^{2}} \sum^{ }_{\mathbf{i}\in [n]_{2}} \sum^{ }_{\mathbf{i}'\in
        [n]_{2}^{(\mathbf{i})}}  \IE
        \{\bar g_{\mathbf{i},\pi(\mathbf{i})}(u) \bar g_{\mathbf{i}',\pi(\mathbf{i}')}(u)
        \Psi_{\beta,t}(W)\}.
		    \end{aligned}
\end{equation*}
For $H_{1}(u)$, by Young's inequality,  
\begin{equation}
    \begin{aligned}
        H_{1}(u)&\leq \frac{1}{16n^{2}} \sum^{ }_{\mathbf{i}\in [n]_{2}} \sum^{ }_{\mathbf{i}'\in
        [n]_{2}\setminus[n]_{2}^{(\mathbf{i})}}  \IE
		\bigl\{\bigl(g_{\mathbf{i},\pi(\mathbf{i})}^{2}(u)+ g_{\mathbf{i}',\pi(\mathbf{i}')}^{2}(u) \bigr)
        \Psi_{\beta,t}(W)\bigr\}\\
                &\leq C n^{-1} \sum^{ }_{\mathbf{i}\in [n]_{2}}   \IE
        \{g_{\mathbf{i},\pi(\mathbf{i})}^{2}(u) \Psi_{\beta,t}(W)\}.
    \end{aligned}
    \label{eq:B_93}
\end{equation}
Taking integration on both sides of \cref{eq:B_93} implies 
\begin{equation}
    \begin{aligned}
		\MoveEqLeft
        \int_{\lvert u\rvert\leq 1} e^{2 t \lvert u\rvert} \IE
        \{g_{\mathbf{i},\pi(\mathbf{i})}^{2}(u) \Psi_{\beta,t}(W)\} du\\
		& \leq C n^{-1} \IE\biggl\{
        \Psi_{\beta,t}(W)
    \Bigl(|D_{\mathbf{i},\pi(\mathbf{i})}|^{3} e^{2t \lvert
D_{\mathbf{i},\pi(\mathbf{i})}\rvert}+\IE|D_{\mathbf{i},\pi(\mathbf{i})}|^{3} e^{2t \lvert
D_{\mathbf{i},\pi(\mathbf{i})}\rvert}\Bigr)\biggr\} .
    \end{aligned}
    \label{eq:B_95}
\end{equation}
Applying \cref{lem-B1} with $k=2, m=0, \sigma=\pi$ and
$\zeta_{\mathbf{i},\mathbf{j}}=|D_{\mathbf{i},\mathbf{j}}|^{3} e^{2t \lvert
D_{\mathbf{i},\mathbf{j}}\rvert}$, we have 
\begin{equation}
    \begin{aligned}
        \MoveEqLeft \IE \{|D_{\mathbf{i},\pi(\mathbf{i})}|^{3} e^{2t \lvert
D_{\mathbf{i},\pi(\mathbf{i})}\rvert} \Psi_{\beta,t}(W)\}\\&\leq Cb \IE
      \{\Psi_{\beta,t}(W)\}  \max_{ \mathbf{v} \in [n]_2} 
        \E \{ |D_{\mathbf{i},\mathbf{v}}|^{3} e^{2t \lvert
D_{\mathbf{i},\mathbf{v}}\rvert} e^{ t  \lvert \sum^{2}_{r=1}
        X_{i_{r},v_{r}} \rvert   } \}\\
                                                           &\leq Cb^{1/16} \alpha_{n}^{-3} \IE
      \{\Psi_{\beta,t}(W)\}  \max_{ \mathbf{v} \in [n]_2} 
      \E \{ |\alpha_{n} D_{\mathbf{i},\mathbf{v}}|^{3} e^{2t \lvert
D_{\mathbf{i},\mathbf{v}}\rvert} e^{ t  \lvert \sum^{2}_{r=1}
        X_{i_{r},v_{r}} \rvert   } \}\\
&\leq Cb\alpha_{n}^{-3} \IE
      \{\Psi_{\beta,t}(W)\} ,
    \end{aligned}
    \label{eq:B_94}
\end{equation}
and similarly, 
\begin{equation}
    \begin{aligned}
        \IE|D_{\mathbf{i},\pi(\mathbf{i})}|^{3} e^{2t \lvert
        D_{\mathbf{i},\pi(\mathbf{i})}\rvert}\leq C b \alpha_{n}^{-3}.
    \end{aligned}
    \label{eq:B_96}
\end{equation}
By \cref{eq:B_93,eq:B_95,eq:B_94,eq:B_96}, we have 
\begin{equation}
    \begin{aligned}
        \int_{\lvert u\rvert\leq 1} e^{2 t \lvert u\rvert} H_{1}(u) du \leq C b n
        \alpha_{n}^{-3} \IE \{\Psi_{\beta,t}(W)\} .
    \end{aligned}
    \label{eq:B_97}
\end{equation}
For $H_{2}(u)$, by \cref{lemB_3}, with $\xi_{\mathbf{i},\mathbf{j}}= g_{\mathbf{i},\mathbf{j}}(u)$, we have any $\mathbf{i}\in
[n]_{2},\mathbf{i}'\in [n]_{2}^{(\mathbf{i})},$ and 
$\mathbf{j},\mathbf{j}'\in [n]_{2}$,
\begin{equation}
    \begin{aligned}
      \MoveEqLeft    \IE
        \{\bar g_{\mathbf{i},\pi(\mathbf{i})}(u) \bar g_{\mathbf{i}',\pi(\mathbf{i}')}(u)
    \Psi_{\beta,t}(W)\}\\
    &\leq C b (1 + t^{2}) h(t) n^{-4}\sum_{\mathbf{j}, \mathbf{j}' \in
           [n]_2}
            \E \{ \lvert \bar g_{\mathbf{i},\mathbf{j}}(u)\rvert
                                   T_{\mathbf{i},\mathbf{j}}^{2} e^{2 t\lvert T_{\mathbf{i},\mathbf{j}}\rvert}\}\IE\{
        \lvert\bar
        g_{\mathbf{i}',\mathbf{j}'}(u)\rvert\}\\&\quad+C b (1 + t^{2}) h(t) n^{-4}\sum_{\mathbf{j}, \mathbf{j}' \in
           [n]_2}\E \{\lvert
                                   \bar g_{\mathbf{i}',\mathbf{j}'}(u)\rvert
         T_{\mathbf{i}',\mathbf{j}'}^{2} e^{2 t\lvert T_{\mathbf{i}',\mathbf{j}'}\rvert}\}\IE\{
     \lvert\bar g_{\mathbf{i},\mathbf{j}}(u)\rvert \}\\
&\quad+C b (1 + t^{2}) h(t) n^{-4}\sum_{\mathbf{j}, \mathbf{j}' \in
[n]_2} \bigl(\alpha_{n}^{-2}+n^{-1}+\1(E_{\mathbf{j},\mathbf{j}'})\bigr)
                            \E \{ \lvert\bar g_{\mathbf{i},\mathbf{j}}(u)\rvert
                                   \}\IE\{ \lvert\bar g_{\mathbf{i}',\mathbf{j}'}(u)\rvert\}\\
&:=H_{21}(u)+H_{22}(u)+H_{23}(u).
    \end{aligned}
    \label{eq:B_133}
\end{equation}
First we consider $H_{21}(u)$. Since $\IE \{|\bar g_{\mathbf{i},\mathbf{j}}(u)|\}\leq 2 \IE
\vert  D_{\mathbf{i},\mathbf{j}}\vert$, by Fubini's theorem, we have for any $\mathbf{i}\in
[n]_{2},\mathbf{i}'\in [n]_{2}^{(\mathbf{i})},$ and 
$\mathbf{j},\mathbf{j}'\in [n]_{2}$,
\begin{equation}
    \begin{aligned}
      \MoveEqLeft\biggl\lvert\int_{|u|\leq 1}^{} e^{2t \lvert u\rvert} \E \{
           \lvert \bar g_{\mathbf{i},\mathbf{j}}(u)\rvert 
          T_{\mathbf{i},\mathbf{j}}^2 e^{2 t\lvert T_{\mathbf{i},\mathbf{j}}\rvert}\}\IE \{ \lvert \bar g_{\mathbf{i}', \mathbf{j}'}(u)\rvert \}du\biggr\rvert\\
           &\leq 2 \biggl\lvert\int_{|u|\leq 1}^{} e^{2t \lvert u\rvert} \E \{
           \lvert \bar g_{\mathbf{i},\mathbf{j}}(u)\rvert 
           T_{\mathbf{i},\mathbf{j}}^2 e^{2 t\lvert T_{\mathbf{i},\mathbf{j}}\rvert}\}\IE \{
  \lvert D_{\mathbf{i}',\mathbf{j}'}\rvert \}du\biggr\rvert\\
	  &\leq  2 \IE\Bigl\{ \Bigl(| D_{\mathbf{i},\mathbf{j}}|^{2} e^{2t \lvert D_{\mathbf{i},\mathbf{j}}\rvert}+
	   \IE\bigl\{ |
			  D_{\mathbf{i},\pi(\mathbf{i})}|^{2} e^{2t \lvert D_{\mathbf{i},\pi(\mathbf{i})}\rvert}\bigr\}\Bigr) T_{\mathbf{i},\mathbf{j}}^2 e^{2 t\lvert T_{\mathbf{i},\mathbf{j}}\rvert} \Bigr\}\IE \{
  \lvert D_{\mathbf{i}',\mathbf{j}'}\rvert \},
    \end{aligned}
    \label{eq:B_85}
\end{equation}
where use the same notation as in \cref{lemB_3}. Then, by \cref{eq:11_24_6} and Young's
inequality, we have
\begin{equation}
    \begin{aligned}
       \MoveEqLeft\IE\Bigl\{ | D_{\mathbf{i},\mathbf{j}}|^{2} e^{2t \lvert D_{\mathbf{i},\mathbf{j}}\rvert}
	    T_{\mathbf{i},\mathbf{j}}^2 e^{2 t\lvert T_{\mathbf{i},\mathbf{j}}\rvert} \Bigr\}\IE \{
  \lvert D_{\mathbf{i}',\mathbf{j}'}\rvert \}\\
        &\leq C \alpha_{n}^{-5}  \IE\bigl\{ | \alpha_{n}^{-1} D_{\mathbf{i},\mathbf{j}}|^{2} e^{2t \lvert
        D_{\mathbf{i},\mathbf{j}}\rvert}(\alpha_{n}^{-1}T_{\mathbf{i},\mathbf{j}})^2 e^{t\lvert
T_{\mathbf{i},\mathbf{j}}\rvert}\bigr\}\IE \{\alpha_{n}^{-1}\lvert D_{\mathbf{i}',\mathbf{j}'}\rvert\} \\
        &\leq C b \alpha_{n}^{-5}. 
    \end{aligned}
    \label{eq:B_86}
\end{equation}
Similar to \cref{eq:B_86}, 
\begin{equation}
    \begin{aligned}
        \IE\Bigl\{ | D_{\mathbf{i},\mathbf{j}}|^{2} e^{2t \lvert
        D_{\mathbf{i},\mathbf{j}}\rvert}\Bigr\}
	    \IE \Bigl\{T_{\mathbf{i},\mathbf{j}}^2 e^{2 t\lvert T_{\mathbf{i},\mathbf{j}}\rvert} \Bigr\}\IE \{
  \lvert D_{\mathbf{i}',\mathbf{j}'}\rvert \}\leq C b
      \alpha_{n}^{-5}.
    \end{aligned}
    \label{eq:B_87}
\end{equation}
Using \cref{eq:B_85,eq:B_86,eq:B_87}, we have 
\begin{equation}
    \begin{aligned}
       \biggl\lvert\int_{|u|\leq 1}^{} e^{2t \lvert u\rvert} \E \{
           \lvert \bar g_{\mathbf{i},\mathbf{j}}(u)\rvert 
          T_{\mathbf{i},\mathbf{j}}^2 e^{2 t\lvert T_{\mathbf{i},\mathbf{j}}\rvert}\}\IE \{ \lvert
  \bar g_{\mathbf{i}', \mathbf{j}'}(u)\rvert \}du\biggr\rvert \leq C b \alpha_{n}^{-5}.
    \end{aligned}
    \label{eq:B_88}
\end{equation}
Furthermore, we have, 
\begin{equation}
    \begin{aligned}
        \biggl\lvert\int_{|u|\leq 1}^{} e^{2t \lvert u\rvert} H_{21}(u) du\biggr\rvert\leq C b
        (1+t^{2}) \alpha_{n}^{-5} h(t) 
    \end{aligned}
    \label{eq:B_134}
\end{equation}
Moreover, by the same argument, we have 
\begin{equation}
    \begin{aligned}
        \biggl\lvert\int_{|u|\leq 1}^{} e^{2t \lvert u\rvert} H_{22}(u) du\biggr\rvert&\leq C b
        (1+t^{2}) \alpha_{n}^{-5} h(t)\\
        \biggl\lvert\int_{|u|\leq 1}^{} e^{2t \lvert u\rvert} H_{23}(u) du\biggr\rvert&\leq C b
(1+t^{2}) (\alpha_{n}^{-5}+n \alpha_{n}^{-3}) h(t)
    \end{aligned}
    \label{eq:B_90}
\end{equation}
By \cref{eq:B_134,eq:B_90}, we have 
we have 
\begin{equation}
    \begin{aligned}
        \int_{\lvert u\rvert\leq 1} e^{2 t \lvert u\rvert} H_{2}(u) du \leq C b^{2}
        (n^{2}\alpha_{n}^{-5}+n \alpha_{n}^{-3}) (1+t^{2})\IE \{\Psi_{\beta,t}(W)\}.
    \end{aligned}
    \label{eq:B_92}
\end{equation}
By \cref{eq:B_97,eq:B_92}, we complete the proof \cref{eq:B.84} for $v=0$. The inequality \cref{eq:B.84} for the case $v=1$ can be shown similarly.
\end{proof}
\end{appendix}

\section*{Acknowledgements}
Liu Song-Hao was partially supported by Hong Kong RGC GRF 14304917 and 14302418. Zhuo-Song Zhang was supported by the Singapore Ministry of Education Academic Research Fund Tier 2 grant MOE2018-T2-2-076. Both authors would like to thank Qi-Man Shao and Xiao Fang for
their helpful
advises.



\bibliographystyle{imsart-nameyear} 
\bibliography{reference}       





\end{document}